\documentclass[reqno,11pt]{amsart}

\usepackage[cal=cm, bb=ams, frak=esstix, scr=mma]{mathalfa}
\usepackage{amssymb, amsmath, latexsym, amssymb, amsfonts, amsbsy, amsthm, amsfonts, bookmark, bm, cancel, caption, color, graphicx, float, hyperref, inputenc, mathtools, mathrsfs, multicol, pgfplots, scalerel, stackengine, tcolorbox, tikz-cd, titlesec} 

\usepackage[inline]{enumitem} 
\usepackage[backend=biber,style=numeric,sortcites,natbib=true, url=false, doi=false]{biblatex}
\renewbibmacro{in:}{} 
\hypersetup{pdflinkmargin=0.5pt}   
\pgfplotsset{width=10cm,compat=1.9} 
\usepgfplotslibrary{fillbetween, external}

\renewcommand{\phi}{\varphi} 

\stackMath   
\titleformat{\section}{\LARGE}{\thesection.}{0.33em}{\centering}
\titleformat{\subsection}{\Large\bfseries}{\thesubsection.}{0.33em}{}
\titleformat{\subsubsection}{\large\bfseries}{\thesubsubsection.}{0.33em}{}
\allowdisplaybreaks
\renewcommand{\comment}[1]{}

\newtheorem{theorem}{Theorem}[section]
\newtheorem*{definition}{Definition}
\newtheorem{lemma}{Lemma}[section]
\newtheorem{remark}[lemma]{Remark}
\newtheorem{corollary}[lemma]{Corollary} 

\newtheorem{proposition}[lemma]{Proposition}
\numberwithin{equation}{section}  

\let\widehat\undefined
\newcommand\widehat[1]{%
\savestack{\tmpbox}{\stretchto{%
  \scaleto{%
    \scalerel*[\widthof{\ensuremath{#1}}]{\kern-.6pt\bigwedge\kern-.6pt}%
    {\rule[-\textheight/2]{1ex}{\textheight}} 
  }{\textheight}%
}{0.5ex}}%
\stackon[2pt]{#1}{\tmpbox}%
}
\newcommand\widecheck[1]{ \savestack{\tmpbox}{\stretchto{\scaleto{
    \scalerel*[\widthof{\ensuremath{#1}}]{\kern-.6pt\bigvee\kern-.6pt}
    {\rule[-\textheight/2]{1ex}{\textheight}} 
  }{\textheight} }{0.5ex}} \stackon[2pt]{#1}{\tmpbox} }
\let\xxrightharpoonup\xrightharpoonup    
\let\xrightharpoonup\undefined  
\newcommand{\xrightharpoonup}[1]{\ifthenelse{\equal{#1}{\star}}
{\xxrightharpoonup{\raisebox{-0.25ex}[0pt][-1.5ex]{\(\scriptstyle \,\,#1\,\,\)}}}
{\xxrightharpoonup{\raisebox{0ex}[0pt][-1.5ex]{\(\scriptstyle \,\,#1\,\,\)}}}}

\newcommand{\Sup}[1]{\raisebox{0.5ex}{\scalebox{0.8}{\(\displaystyle \sup_{#1}\;\)}}}

\title[Blow-up]{Stable Self-Similar Blow-Up in Nonlinear Wave Equations with Quadratic Time-Derivative Nonlinearities} 
\author{Jie Liu}
\address[J. Liu]{Department of Mathematics, New York University Abu Dhabi, Saadiyat Island, P.O. Box 129188, Abu Dhabi, United Arab Emirates.}
\email{jl15817@nyu.edu} 
\author{Faiq Raees}
\address[F. Raees]{Courant Institute of Mathematical Sciences, New York University, 251 Mercer Street, New York, United States of America.\newline \indent Department of Mathematics, New York University Abu Dhabi, Saadiyat Island, P.O. Box 129188, Abu Dhabi, United Arab Emirates.}
\email{fr872@nyu.edu}

\begin{document}

\begin{abstract} 
    We study singularity formation in two one-dimensional nonlinear wave models with quadratic time-derivative nonlinearities. The non-null model violates the null condition and typically develops finite-time blow-up; the null-form model is Lorentz-invariant and enjoys small-data global existence, yet still admits blow-up for large data. Building on our earlier work on spatial-derivative nonlinearities, we construct and classify a five-parameter family of generalized self-similar blow-up solutions that captures the observed dynamics. We prove that no smooth exact self-similar profiles exist, while the generalized self-similar solutions—exhibiting logarithmic growth—provide the correct blow-up description inside backward light cones. We further establish asymptotic stability for the relevant branches, including the ODE-type blow-up in both models. These results yield a coherent and unified picture of blow-up mechanisms in time-derivative nonlinear wave equations.
\end{abstract}
\maketitle
\section{Introduction}

We investigate the singularity formation for two classes of one-dimensional nonlinear wave equations (NLW) featuring quadratic nonlinearities in the first derivatives:
\begin{equation}\label{eq:1.■u=(∂ₜu)²}
	\partial_{tt} u - \partial_{xx} u = (\partial_t u)^2
\end{equation}
and
\begin{equation}\label{eq:1.■u=(∂ₜu)²−(∂ₓu)²}
	\partial_{tt} u - \partial_{xx} u = (\partial_t u )^2 - (\partial_x u )^2.
\end{equation}
The study of blow-up dynamics for such equations is fundamental to the understanding nonlinear hyperbolic phenomena. This paper continues our systematic investigation into the self-similar blow-up mechanisms in one-dimensional NLW with quadratic derivative nonlinearities, extending our previous study \cite{ghoul2025blow} of the spatial-derivative case, $\Box u = (\partial_x u)^2$ . Here, we focus on the equally important time-derivative counterparts.

\subsection{Motivation and background}

The study of singularity formation for nonlinear wave equations has a long and rich history. A foundational result \cite{john1981blow} by Fritz John demonstrated that, for certain nonlinearities, solutions to NLW can develop singularities in finite time, even for arbitrarily small and smooth initial data. Equation~\eqref{eq:1.■u=(∂ₜu)²} serves as a canonical example of such behavior: it violates the celebrated null condition~\cite{klainerman1986null,christodoulou1986global}, and its nonlinear term $(\partial_t u)^2$ is well known to induce finite-time blow-up. This makes it a prototypical model for investigating singularity formation in nonlinear hyperbolic dynamics.

In contrast, equation~\eqref{eq:1.■u=(∂ₜu)²−(∂ₓu)²} represents the archetypal model satisfying the null condition. The quadratic form $(\partial_t u)^2 - (\partial_x u)^2$ (and its higher-dimensional counterpart $(\partial_t u)^2 - |\nabla_x u|^2$) is Lorentz invariant—a structure of central importance in physics. It naturally arises in geometric and relativistic models such as the Wave Maps system~\cite{chao1980cauchy,tataru2005rough} and the Yang-Mills equations (in the Lorentz gauge)~\cite{klainerman1995finite}, where the null structure ensures the long-time stability of solutions.

While the local theory for these models is classical—smooth local well-posedness follows from standard energy methods—their long-time dynamics are markedly different. For~\eqref{eq:1.■u=(∂ₜu)²}, finite-time blow-up is expected even for small data. In contrast, for the null-form equation~\eqref{eq:1.■u=(∂ₜu)²−(∂ₓu)²}, the null structure yields small-data global existence (GWP), even in the one-dimensional setting where linear waves exhibit no decay~\cite{luli2018one}.

Nevertheless, despite small-data global existence, the null-form equation~\eqref{eq:1.■u=(∂ₜu)²−(∂ₓu)²} also admits blow-up for large initial data. Indeed, both~\eqref{eq:1.■u=(∂ₜu)²} and~\eqref{eq:1.■u=(∂ₜu)²−(∂ₓu)²} possess the explicit ODE-type blow-up solution 
\begin{align}\label{ode-blowup}
    u(t) = -\log(T-t),
\end{align}
indicating that finite-time singularities are a natural phenomenon for both systems. This raises a fundamental and largely unresolved question: \textit{what are the precise blow-up dynamics and stability mechanisms underlying such singularities?}

The objective of this paper is to provide a detailed characterization of these dynamics by constructing and analysing self-similar blow-up solutions. This work directly extends the methodology developed in our previous paper \cite{ghoul2025blow}, where we studied the stable self-similar blow-up for $\Box u = (\partial_x u)^2$. We now apply this framework to both \eqref{eq:1.■u=(∂ₜu)²} and \eqref{eq:1.■u=(∂ₜu)²−(∂ₓu)²}, aiming to provide a unified understanding of singularity formation across this entire family of 1D NLW with quadratic derivative nonlinearities.

We note the recent, independent work of Gough \cite{gough2025stable}, which appeared as our present work was nearing completion. That work also investigates equation \eqref{eq:1.■u=(∂ₜu)²}, adopting the spectral-theoretic framework introduced in our prior work \cite{ghoul2025blow}. Our analysis, however, is broader in scope and distinct in several key aspects. First, it covers both the non-null case~\eqref{eq:1.■u=(∂ₜu)²} and, crucially, the null-form equation~\eqref{eq:1.■u=(∂ₜu)²−(∂ₓu)²}, which is not treated in~\cite{gough2025stable}. Second, for the overlapping case~\eqref{eq:1.■u=(∂ₜu)²}, our results complete the stability analysis of the ODE-type blow-up solution \eqref{ode-blowup} (corresponding to the $p=1$ case in~\cite{gough2025stable}), which was left open therein. Furthermore, we obtain an accurate spectral gap and improved convergence rates in the asymptotic stability of generalized self-similar blow-up profiles.

\subsection{Main results}
The self-similar ansatz for the both equations \eqref{eq:1.■u=(∂ₜu)²} and \eqref{eq:1.■u=(∂ₜu)²−(∂ₓu)²} takes the form
\begin{align*}
	u(t,x) = U\left(y\right), \quad \ y = \frac{x-x_0}{T-t}, \quad T>0,\ x_0\in \mathbb{R}.
\end{align*}
Under this ansatz, exact self-similar solutions to  \eqref{eq:1.■u=(∂ₜu)²} and \eqref{eq:1.■u=(∂ₜu)²−(∂ₓu)²}  satisfy
\begin{align*}
		2 y \partial_y U  + (y^2 - 1) \partial_{yy} U = y^2(\partial_y U)^2. 
\end{align*}
and
\begin{align*}
	2 y \partial_y U  + (y^2 - 1) \partial_{yy} U = (y^2-1)( \partial_y U)^2. 
\end{align*}
respectively.


Our first result establishes the existence of a family of smooth generalized self-similar blow-up solutions.

\begin{theorem}[Existence of smooth generalized self-similar blow-up solutions] \label{thm:existence} \quad 
	\begin{enumerate}
		\item For any $T>0$ and $x_0\in \mathbb{R}$, there are no smooth exact self-similar blow-up solutions to \eqref{eq:1.■u=(∂ₜu)²} and \eqref{eq:1.■u=(∂ₜu)²−(∂ₓu)²} in the backward lightcone $$\Gamma(T,x_0) :=\left\{(t, x) \in\left[0, T\right) \times \mathbb{R}:\left|x-x_0\right| \le T-t\right\}.$$
		\item There exists a five-parameter family of smooth generalized self-similar blow-up solutions to  \eqref{eq:1.■u=(∂ₜu)²} and \eqref{eq:1.■u=(∂ₜu)²−(∂ₓu)²} with logarithmic growth
		\begin{align*}
			u_{\alpha,\beta,\kappa,T,x_0} = -\alpha\log\left(1-\frac{t}{T}\right) + U_{\alpha,\beta}\left(\frac{x-x_0}{T-t}\right) +\kappa
		\end{align*}
		well defined in the backward lightcone $\Gamma(T,x_0)$ for all $T>0, \kappa, x_0\in \mathbb{R},$
		where for  \eqref{eq:1.■u=(∂ₜu)²},  $(\alpha,\beta) \in (0,1]\times \{0, \infty\}$ and
		\begin{align*}
			U_{\alpha,0}(y)=  -\alpha \ln (1 - \sqrt{ 1-\alpha} y), \qquad U_{\alpha,\infty}(y)=  -\alpha \ln (1 + \sqrt{ 1-\alpha} y);
		\end{align*}
		and for \eqref{eq:1.■u=(∂ₜu)²−(∂ₓu)²}, 	\((\alpha,\beta)\in \mathbb{Z}_+ \times \mathbb{R}_+ ,\) and
		\begin{align*}
			U_{\alpha,\beta}(y) =  - \log \left( (1+ y)^\alpha + \beta(1 - y)^\alpha\right) + \log (1+\beta).
		\end{align*}
	\end{enumerate}   
\end{theorem}

\begin{remark}[Global-in-space smooth blow-up solutions]
	By the finite propagation speed of wave equations and a cut-off of $u_{\alpha,\beta,\kappa,T,x_0}$, we can also extend the solution out of the lightcone and obtain global in space smooth blow-up solutions to \eqref{eq:1.■u=(∂ₜu)²} and \eqref{eq:1.■u=(∂ₜu)²−(∂ₓu)²}. 
\end{remark}
\begin{remark}
    For $\alpha=1$, $u_{\alpha,\beta,\kappa,T,x_0}$ is reduced to the ODE blow-up solution \eqref{ode-blowup} of \eqref{eq:1.■u=(∂ₜu)²} when $\beta =0,\infty$ and of \eqref{eq:1.■u=(∂ₜu)²−(∂ₓu)²} when $\beta=1$, up to a constant.
\end{remark}


The following two theorems establish the asymptotic stability of the generalized self-similar blow-up solutions $u_{\alpha,\infty,\kappa,T,x_0}$ for  \eqref{eq:1.■u=(∂ₜu)²},  $(\alpha,\beta) \in (0,1]\times \{0, \infty\}$, and for  \eqref{eq:1.■u=(∂ₜu)²−(∂ₓu)²}, 	\((\alpha,\beta)\in \{1\} \times \mathbb{R}_+ \).  
\begin{theorem}[Asymptotic stability of generalized self-similar blow-up solutions for \eqref{eq:1.■u=(∂ₜu)²}]
	\label{thm:main-1}
    Let $\alpha_0\in(0,1]$, $T_0>0, \kappa_0, x_0\in \mathbb{R}$, $k\ge 5$. For all $0<\delta<1$, there exists a constant $\epsilon>0$ such that for any real-valued functions $(f, g)\in H^{k+1}(\mathbb{R})\times H^k(\mathbb{R})$ with 
    $$\|(f, g)\|_{H^{k+1}(\mathbb{R})\times H^k(\mathbb{R})} \le \epsilon $$
    there exist parameters $\alpha^*\in (0,1], T^*>0, \kappa^*\in \mathbb{R}$, and a unique solution $u: \Gamma(T^*, x_0)\mapsto \mathbb{R}$ of  \eqref{eq:1.■u=(∂ₜu)²} with initial data
	$$ u(0,x) = u_{\alpha_0,\infty,\kappa_0,T_0,x_0}(0,x) + f(x), \quad \partial_t u(0,x) = \partial_t u_{\alpha_0,\infty,\kappa_0,T_0,x_0}(0,x) + g(x), \quad x\in B_{T^*}(x_0), $$
	such that the bounds
	\begin{align*}
		&(T^*-t)^{-\frac{1}{2} + s} \| u(t,\cdot) - u_{\alpha^*, \infty, \kappa^*, T^*, x_0}\|_{\dot{H}^{s}(B_{T^*-t}(x_0))} \lesssim (T^*-t)^{1-\delta} ,
	\end{align*}
	for $s=0,1,\ldots,k+1$, and 
	\begin{align*}    
		(T^*-t)^{-\frac{1}{2} + s} \| \partial_t u(t,\cdot) - \partial_t u_{\alpha^*, \infty, \kappa^*, T^*, x_0}\|_{ \dot{H}^{s-1}(B_{T^*-t}(x_0))} \lesssim (T^*-t)^{1-\delta} ,
	\end{align*}
	for $s=1,\ldots,k+1$, hold for all $0\le t < T^*.$ Moreover, 
	\begin{align*}
		|\alpha^*-\alpha_0| + |\kappa^* -\kappa_0| + \left|\frac{T^*}{T_0}-1\right| \lesssim \epsilon.
	\end{align*}
\end{theorem}
\begin{remark}
	The same result also holds for $u_{\alpha,0,\kappa,T,x_0}$. Indeed, we have 
	$$u_{\alpha,0,\kappa,T,x_0}(t,x) = u_{\alpha,\infty,\kappa,T,-x_0}(t,-x),$$ 
	and any solution $u(t,x)$ to \eqref{eq:1.■u=(∂ₜu)²} implies that $u(t,-x)$ also solves \eqref{eq:1.■u=(∂ₜu)²}. 
\end{remark}

\begin{theorem}[Asymptotic stability of generalized self-similar blow-up solutions for \eqref{eq:1.■u=(∂ₜu)²−(∂ₓu)²}]
	\label{thm:main-2}
	Let $\beta_0\in \mathbb{R}_+$, $T_0>0, \kappa_0, x_0\in \mathbb{R}$, and \(k \geq k_{\beta _0} :=2 + \beta _0 ^{-1} + \beta _0\). For all $0<\delta<1$, there exists a constant $\epsilon>0$ such that for any real-valued functions $(f, g)\in H^{k+1}(\mathbb{R})\times H^k(\mathbb{R})$ with 
	$$\|(f, g)\|_{H^{k+1}(\mathbb{R})\times H^k(\mathbb{R})} \le \epsilon $$
	there exist parameters $\beta^*\in \mathbb{R}_+, T^*>0, \kappa^*\in \mathbb{R}$, and a unique solution $u: \Gamma(T^*, x_0)\mapsto \mathbb{R}$ of  \eqref{eq:1.■u=(∂ₜu)²−(∂ₓu)²} with initial data
	$$ u(0,x) = u_{1,\beta_0,\kappa_0,T_0,x_0}(0,x) + f(x), \quad \partial_t u(0,x) = \partial_t u_{1,\beta_0,\kappa_0,T_0,x_0}(0,x) + g(x), \quad x\in B_{T^*}(x_0), $$
	such that the bounds
	\begin{align*}
		&(T^*-t)^{-\frac{1}{2} + s} \| u(t,\cdot) - u_{1, \beta^*, \kappa^*, T^*, x_0}\|_{\dot{H}^{s}(B_{T^*-t}(x_0))} \lesssim (T^*-t)^{1-\delta} ,
	\end{align*}
	for $s=0,1,\ldots,k+1$, and 
	\begin{align*}    
		(T^*-t)^{-\frac{1}{2} + s} \| \partial_t u(t,\cdot) - \partial_t u_{1, \beta^*, \kappa^*, T^*, x_0}\|_{ \dot{H}^{s-1}(B_{T^*-t}(x_0))} \lesssim (T^*-t)^{1-\delta} ,
	\end{align*}
	for $s=1,\ldots,k+1$, hold for all $0\le t < T^*.$ Moreover, 
	\begin{align*}
		|\beta^*-\beta_0| + |\kappa^* -\kappa_0| + \left|\frac{T^*}{T_0}-1\right| \lesssim \epsilon.
	\end{align*}
\end{theorem}

\begin{remark}
	For $\alpha \in \mathbb{N}_{\ge 2}$ and $0<\beta<\infty$, we conjecture that $u_{\alpha,\beta,\kappa,T,x_0}$ are only co-dimensionally stable.  The spectral analysis of the corresponding linearized equation reveals the presence of some additional unstable eigenvalues, i.e., eigenvalues with positive real part, which not associated with any symmetries of the equation. Consequently, we expect that, for $\alpha \in \mathbb{N}_{\ge 2}$ and $0<\beta<\infty$, perturbations around $u_{\alpha,\beta,\kappa,T,x_0}$ would always converge to the ground state  $u_{1,\beta,\kappa,T,x_0}$. In other words, $u_{1,\beta,\kappa,T,x_0}, \beta\in \mathbb{R}_+$ serve as stable attractors for the dynamics.
\end{remark}

\subsection{Related work}
The study of singularity formation in nonlinear wave equations is a broad field. One of its most significant branches focuses on characterizing the precise dynamics of blow-up, particularly the existence and stability of self-similar solutions. A comprehensive review of this specific subfield was provided in our prior work \cite[Section 1.2]{ghoul2025blow}. As detailed therein, this line of research was pioneered by the foundational work of the Merle and Zaag \cite{merle2003determination,merle2007existence,merle2008openness,merle2012existence,merle2012isolatedness,merle2015stability,merle2016dynamics} on power-type nonlinearities (e.g., $\Box u = |u|^{p-1} u$). This was later developed into a robust spectral-theoretic framework by the Donninger group \cite{donninger2012stable,donninger2014stable,donninger2016blowup,donninger2017stable,donninger2017strichartz,donninger2020blowup,wallauch2023strichartz,ostermann2024stable} for analysing the stability of the ODE blow-up profiles. We refer the interested reader to \cite{liu2025note,ghoul2025blow} for a detailed discussion of these developments, which form the context for the analytical approach adopted here.

For derivative-type nonlinearities, the literature on blow-up dynamics is far more limited. Most existing studies have focused on lifespan estimates and the geometric properties of singularities. We refer to the works of Rammaha~\cite{rammaha1995upper,rammaha1997note}, Sasaki and coauthors~\cite{sasaki2018regularity,sasaki2023lifespan}, and Shao–Takamura–Wang~\cite{shao2024blow}.

To the best of our knowledge, the first rigorous characterization of type-I blow-up dynamics for a derivative nonlinear wave equation was achieved by Speck~\cite{speck2020stable}, who established the stability of the ODE-type blow-up in the quasilinear model
$$-u_{tt} + \frac{\Delta u}{1+u_t} = -(u_t)^2,$$
using a delicate energy method. In contrast, for the semilinear equation~\eqref{eq:1.■u=(∂ₜu)²}, the blow-up mechanism remains largely unexplored; see, for instance, the discussion in~\cite[Section~1]{speck2020stable}.

We also highlight the recent contributions of Chen–McNulty–Schörkhuber~\cite{chen2023singularity} and McNulty~\cite{mcnulty2024singularity}, who analysed the stable self-similar blow-up in Skyrme-type models. Their approach exploits a crucial structural transformation that removes the derivative nonlinearity, thereby allowing the application of the spectral-theoretic machinery originally developed for power-type nonlinearities.

\vspace{0.5cm}
In the proofs of Theorems~1.2 and~1.3, we build on the methods developed in our earlier work \cite{ghoul2025blow}. Several ingredients, however---most notably the spectral analysis---are problem-specific and substantially more delicate. For example, a key step in \cite{ghoul2025blow} is the observation that for all $\operatorname{Re}\lambda>-1$, the only eigenvalues of the linearized operator are $0$ and $1$, yielding a spectral gap of size~$1$ between the stable and unstable spectra. By contrast, as noted in \cite{gough2025stable}, for equation \eqref{eq:1.■u=(∂ₜu)²} the author could only show that for $\operatorname{Re}\lambda\ge 0$ the only eigenvalues are $0$ and $1$, which led to an inaccurate conjecture of ``instability’’ of the ODE blow-up solution. In this work, we recover the full spectral gap by a more refined Frobenius analysis.

A second difficulty arises from the degeneracy at $\alpha=1$: the first $\alpha$-derivative of the generalized self-similar profile $U_{\alpha,\beta}$ is singular at $\alpha=1$, causing an apparent discontinuity in the $0$-generalized eigenfunction $\mathbf{g}_{0,\alpha}$ as $\alpha\to 1$, which complicates the nonlinear modulation analysis (see Lemma \ref{lem:7.Taylor Expansion}). We resolve this issue by constructing a more suitable choice of $\mathbf{g}_{0,\alpha}$ and by performing a separate, carefully tailored analysis in the case $\alpha=1$.

\subsection{Outline of the paper}
In Section~\ref{sec:Self-similar solutions and ODE blow up}, we construct smooth generalized self-similar blow-up solutions through direct ODE analysis, proving Theorem~\ref{thm:existence}.
Sections~\ref{sec:Mode stability of first}–\ref{sec:Mode stability of second} establish mode stability of the generalized self-similar solutions.
Section~\ref{sec:Spectral Analysis} performs a detailed spectral analysis of the linearized operator, following the framework introduced in~\cite{ghoul2025blow}.
In Section~\ref{sec:Growth bound of the linearised flow}, we establish linearized stability along the stable subspace.
Finally, Section~\ref{sec:Nonlinear Stability} completes the proof of nonlinear stability using the Lyapunov–Perron method.

\subsection{Notation}
We adopt the same notation as in our previous work \cite{ghoul2025blow}.
In particular, for $a,b\in \mathbb{R},$ we use the notation $a\lesssim b$ to mean that there exists a universal constant $C>0$ such that $a\le Cb$, and $a\lesssim_{\alpha} b$ indicates that the constant $C$ depends on $\alpha$. $H^k$ denotes the usual Sobolev space. For $k\ge 0$, we define $\mathcal{H}^k =H^{k+1} \times H^{k}$. We use boldface notation for tuples of functions, for example:
\begin{align*}
	\mathbf{f} \equiv\left(f_1, f_2\right) \equiv\left[\begin{array}{l}
		f_1 \\
		f_2
	\end{array}\right] \quad \text { or } \quad \mathbf{q}(t, .) \equiv\left(q_1(t, .), q_2(t, .)\right) \equiv\left[\begin{array}{l}
		q_1(t, .) \\
		q_2(t, .)
	\end{array}\right].
\end{align*}
Linear operators that act on tuples of functions are also displayed in boldface notation. For a closed linear operator $\mathbf{L}$ on a Banach space, we denote its domain by $\mathcal{D}(\mathbf{L})$, its spectrum by $\sigma(\mathbf{L})$, and its point spectrum by $\sigma_p(\mathbf{L})$. The resolvent operator is denoted by $\mathbf{R}_{\mathbf{L}}(z):=(z-\mathbf{L})^{-1}$ for $z \in \rho(\mathbf{L})=\mathbb{C} \backslash \sigma(\mathbf{L})$. The space of bounded operators on a Banach space $\mathcal{X}$ is denoted by $\mathcal{L}(\mathcal{X})$.

For details on the spectral theory of linear operators, we refer to \cite{kato1995}. The theory of strongly continuous operator semigroups is covered in the textbook \cite{engel2000one}.

\section{Self-similar solutions and ODE blow-up}\label{sec:Self-similar solutions and ODE blow up}
\subsection{Self-similar solutions to \texorpdfstring{\eqref{eq:1.■u=(∂ₜu)²}}{\ref{eq:1.■u=(∂ₜu)²}}.}
For \( T > 0\) and \(x_0 \in \mathbb{R}\), we introduce the following self-similar variables 
\begin{equation}\label{eq:2.Self-similar u,s,y}
  \begin{cases} u = U(s,y) , \\
  s = - \ln (T - t) ,\\
y = \frac{x - x_0}{ T -t}   \end{cases}  
\end{equation}
The equation \eqref{eq:1.■u=(∂ₜu)²} can be transformed into 
\begin{equation}\label{eq:2.■u=(∂ₜu)² Transformation}
  \partial_{ss} U + \partial_s U  + 2 y \partial_{ys} U  + 2 y \partial_y U  + (y^2 - 1) \partial_{yy} U = ( \partial_s U + y \partial_y U)^2. 
\end{equation}
Some points are in order: 
\begin{itemize}
  \item \textbf{Self-similar solutions:} The self-similar solutions of \eqref{eq:1.■u=(∂ₜu)²} are always singular in the light cone \( \lvert y \rvert \leq 1\). More precise, we have 
  \[  U = \int \frac{4}{2y + (y^2 - 1) \ln   \lvert \frac{1 + y}{1 - y} \rvert + d(y^2 - 1)} , \quad d \in \mathbb{R}. 
  \] 
 and note that for any \(d \in \mathbb{R}\), there exists a value \( y_\star(d) \in (-1,1)\) such that 
  \[ 2y_\star + (y_\star^2 - 1) \ln   \left\lvert \frac{1 + y_\star}{1 - y_\star} \right\rvert + d(y_\star^2 - 1)  = 0. 
  \] 
  Indeed, this follows from intermediate value theorem once one notes that denominator tends to \( \pm 2\) for \( x = \pm 1\) respectively. 
  \item \textbf{ODE blow-up:} For this equation, we have a trivial ODE blow-up solution which is independent of the spatial variable i.e. 
  \[  u = - \ln  ( T - t) . 
  \] 
  Moreover, by finite propagation speed, we can construct blow-up solutions for smooth and compact supported initial data. Unlike power nonlinearity \(  \lvert u \rvert^{ p-1} u\), \eqref{eq:1.■u=(∂ₜu)²} is not Lorentz invariant so we can not obtain a family of type-I blow-up solutions by Lorentz transformation. 
  \item \textbf{Generalized self-similar blow-up:} Nevertheless, for $0<\alpha< 1, \beta\in [0,\infty], \kappa\in \mathbb{R}$, we find a new family of ``generalized self-similar" blow-up solutions of the following form 
  \[  U_{\alpha , \beta, \kappa}(s,y) =  \alpha s  + \tilde{U} _{\alpha, \beta}(y) + \kappa, 
  \]  
  where \( \tilde{U}_{\alpha, \beta}(y)\) satisfies the equation 
  \begin{equation}\label{eq:2.Linearisation around s} 
   \alpha + 2 y \partial_y \tilde{U} _{\alpha, \beta} + (y^2 - 1) \partial_{yy} \tilde{U} _{\alpha, \beta} = (\alpha + y \partial_y \tilde{U} _{\alpha, \beta})^2,
  \end{equation} 
  and the extra parameter $\beta$ is to be specified.
 We first note that 
  \[\tilde{U} _{\alpha,0}:=  -\alpha \ln (1 - \sqrt{ 1-\alpha} y) \quad  \mbox{~and~} \quad  \tilde{U} _{\alpha,\infty}:=  -\alpha \ln (1 + \sqrt{ 1-\alpha} y) 
  \] 
  are two exact smooth solutions to \eqref{eq:2.Linearisation around s}.  Since \eqref{eq:2.Linearisation around s} is a Riccati type equation (under the change of variable \( z = \alpha  + y \partial_y \tilde{U}_{\alpha, \beta}\)), we can solve the whole family of solutions to \eqref{eq:2.Linearisation around s} by letting \( \partial_y \tilde{U} _{\alpha, \beta} = \partial_y \tilde{U} _{\alpha, \infty} + g_{\alpha,\beta}\) where \( g_{\alpha,\beta}\) satisfies 
  \[  (y^2 - 1)  g_{\alpha,\beta}' + \left( 2 (1-\alpha) y  + \frac{2 \alpha \sqrt{1-\alpha} y^2  }{ 1 + \sqrt{1 - \alpha}y }\right)g_{\alpha,\beta} = y^2 g^2_{\alpha,\beta}. 
  \] 
  Then we can solve the above equation as
\begin{align}\label{eq:2.Def of gCB}
  g_{\alpha, \beta}(y)  &= \frac{(\frac{1 + y}{ 1 - y})^{ \sqrt{ 1-\alpha}} ( y + \frac{1}{\sqrt{1-\alpha}})^{-2}}{\beta + \int_{-1}^y \frac{z^2}{1 - z^2} (\frac{1 + z}{1 - z})^{\sqrt{ 1-\alpha}}(z + \frac{1}{\sqrt{ 1-\alpha}})^{ - 2}\mathrm{\,d} z} \nonumber\\
  & = \frac{(\frac{1 + y}{ 1 - y})^{ \sqrt{ 1-\alpha}} ( y + \frac{1}{\sqrt{1-\alpha}})^{-2}}{\beta + \frac{\sqrt{1-\alpha}}{2\alpha} \frac{1-\sqrt{1-\alpha} y}{1+\sqrt{1-\alpha} y} (\frac{1 + y}{1 - y})^{\sqrt{ 1-\alpha}}} \nonumber\\
  & = \frac{2\alpha(1-\alpha)(1+y)^{ \sqrt{ 1-\alpha}} }{2\alpha\beta (1-y)^{ \sqrt{ 1-\alpha}}(1+\sqrt{1-\alpha} y)^2 + \sqrt{1-\alpha} (1+y)^{\sqrt{ 1-\alpha}} (1-(1-\alpha)y^2) }, 
\end{align}
where $0<\alpha<1$ and the parameter \(\beta\ge 0\) satisfies
\[ \beta = \frac{1-\alpha}{g_{\alpha,\beta}(0)} - \frac{\sqrt{1-\alpha}}{2\alpha}. \]
And we define for $0<\beta<\infty$
\begin{align*}
    \tilde{U} _{\alpha, \beta}(y) &:=  \tilde{U} _{\alpha, \infty}(y) + \int_0^y g_{\alpha,\beta}(z) dz\\
    & =  \int_0^y \frac{\alpha\sqrt{1-\alpha}\left(\sqrt{1-\alpha}(1+y)^{\sqrt{1-\alpha}} -2\alpha\beta(1-y)^{\sqrt{1-\alpha}} \right)}{\sqrt{1-\alpha}(1+y)^{\sqrt{1-\alpha}} (1-\sqrt{1-\alpha}y) + 2\alpha\beta(1-y)^{\sqrt{1-\alpha}} (1+\sqrt{1-\alpha}y) }
\end{align*}
Note that by definition, for $0<\alpha< 1$, $\tilde{U} _{\alpha, 0}(y)$ and $\tilde{U} _{\alpha, \infty}$ (and thus $U_{\alpha,0,\kappa}(y)$ and $U_{\alpha,\infty,\kappa}(y)$) are smooth in the light cone $|y|\le 1$. However, for $0<\beta<\infty$, $\tilde{U} _{\alpha, \beta}(y)$ does not belong to $C^k[-1,1]$ if $k>1+\sqrt{1-\alpha}$.

\begin{remark}
    For $\alpha =1$, we have
$$g_{1, \beta}(y)  = \frac{-2}{2y  + \ln   \lvert \frac{1 - y}{1 + y} \rvert + \beta},$$
which is always singular for $0\le \beta <\infty.$ Moreover,  $\tilde{U} _{1,0} =\tilde{U} _{1,\infty} \equiv 0$,  and $U_{1,0,\kappa} = U_{1,\infty,\kappa} =s+\kappa$ which coincides with the ODE blow-up solution. 
\end{remark}

\end{itemize} 
\begin{remark}\label{rmk:2.Range of stability}
  Similar to the non-linear wave equation with nonlinearity \( u_x^2\), smooth generalized self-similar profiles \( U_{\alpha,\beta,\kappa}\) only exist when \( 0 < \alpha\leq 1\) and \(  \beta=0, \infty\). Therefore, we may conjecture that \( U_{\alpha,\beta,\kappa}\) with \( 0 < \alpha\leq 1\) and \( \beta = 0,\infty\) are stable and determine the asymptotic behaviour of smooth blow-up solutions. 
\end{remark}
\begin{remark}[Log blow-up]\label{rmk:2.Low blow-up}
  An interesting phenomenon of the blow-up solutions to derivative nonlinear wave is that exact self-similar solutions are singular whereas the ODE blow-up solution has a log correction to the scaling. Log-blow up rate seems to be generic. 
\end{remark} 
\subsection{Self-similar solutions to \texorpdfstring{\eqref{eq:1.■u=(∂ₜu)²−(∂ₓu)²}}{\ref{eq:1.■u=(∂ₜu)²−(∂ₓu)²}}.}
For \( T > 0\) and \( x_0 \in \mathbb{R} \), we introduce the following self-similar variables 
\begin{equation} 
  \begin{cases} u = U(s,y) , \\
  s = - \ln (T - t) ,\\
y = \frac{x - x_0}{ T -t}.  \end{cases}  
\end{equation}
The \eqref{eq:1.■u=(∂ₜu)²−(∂ₓu)²} can be transformed into 
\begin{equation}\label{eq:2.■u=(∂ₜu)²−(∂ₓu)² Transformation}
  \partial_{ss} U  + \partial_s U + 2 y \partial_{ys} U  + 2 y \partial_y U  + (y^2 - 1 ) \partial_{yy} U  = (\partial_s U  + y \partial_y U)^2 - (\partial_y U)^2. 
\end{equation}
Few points are in order: 
\begin{itemize}
  \item \textbf{Self-similar solutions:} The exact self-similar solutions of \eqref{eq:1.■u=(∂ₜu)²−(∂ₓu)²} are always singular within the light cone \(  \lvert y \rvert \leq 1\). More precisely, we have 
  \[  U(y): = - \int \frac{1}{d(y^2 - 1)  + \frac{1}{2} (y^2 - 1) \ln   \lvert \frac{1 + y}{1 - y} \rvert} , \quad d \in \mathbb{R}\] 
  which is singular at \( y = \pm 1\). 
  \item  \textbf{ODE blow-up:} For this equation, we also have a trivial ODE blow-up solution which is independent of the spatial variable, i.e., 
  \[  u = - \ln  (T - t).  
  \] 
  And by the finite propagation speed we can construct blow-up solutions for smooth and compactly supported initial data. By the Lorentz invariance of \eqref{eq:1.■u=(∂ₜu)²−(∂ₓu)²}, we obtain a family of type-I blow-up solutions: 
  \[  u(t,x) = - \ln  (T - t+ d x)   , \quad -1 \leq d \leq 1. 
  \]  
  \item \textbf{Generalized self-similar blow-up:} Similarly, for $\alpha>0, \beta\in [0,\infty], \kappa\in \mathbb{R}$, there are other families of generalized self-similar blow-up solutions of the following form 
     \begin{equation}\label{eq:2.Uαβκ for (∂ₜu)²−(∂ₓu)²}
     U_{\alpha,\beta, \kappa} = \alpha s + \tilde{U}_{\alpha,\beta}(y) + \kappa  
   \end{equation} 
  where \( \tilde{U}_{\alpha,\beta}\) satisfies 
  \[ \begin{cases}  \alpha +  2 y \partial_y\tilde{U}_{\alpha,\beta}  + (y^2 - 1 ) \partial_{yy} \tilde{U}_{\alpha,\beta}  = (\alpha + y \partial_y \tilde{U}) ^2  - (\partial_y \tilde{U})^2, \\
  \tilde{U}_{\alpha,\beta}(0) = 0,
\end{cases}  
  \] 
  and the extra parameter $\beta$ is to be specified. Indeed, for each \(\alpha > 0\) we can solve the above equation as 
 \begin{align}\label{eq:2.Equation for Ũαβ for ∂ₓu²}
  \tilde{U}_{\alpha , \beta}  (y) &:= - \alpha \int_{0}^y \frac{(1 + y')^{ \alpha - 1} - \beta (1 - y')^{\alpha - 1}}{(1 + y')^\alpha + \beta(1 - y')^{\alpha}} \mathrm{\,d} y'  \nonumber \\
  & =  - \log \left( (1+ y)^\alpha + \beta(1 - y)^\alpha\right) + \log (1+\beta) 
  , \quad \beta> 0. 
 \end{align} 
   In particular, for \((\alpha,\beta)\in \mathbb{Z}_+ \times \mathbb{R}_+ ,\) \( \tilde{U}_{\alpha,\beta}(y)\) is smooth for all \(  \lvert y \rvert \leq 1\). Moreover, for \(\alpha = 1\), 
   \[  \tilde{U}_{1, \beta} (y)  = - \log  \left(1+ \frac{1- \beta}{1+ \beta}y\right), \quad \beta>0,
   \] 
   in which case we go back to the ODE-family blow-up solutions.
\end{itemize} 

\begin{remark}\label{rmk:2.Range of stability for ∂ₓu²}
  Compared to \eqref{eq:1.■u=(∂ₜu)²}, there are countable many families of smooth generalized self-similar blow-up profiles to \eqref{eq:1.■u=(∂ₜu)²−(∂ₓu)²}, i.e.,  \({U}_{\alpha,\beta,\kappa}\) with \( (\alpha,\beta) \in \mathbb{Z}_+ \times \mathbb{R}_+\). However, we only expect the ground state \(\alpha=1, \beta\in \mathbb{R}_+\) is stable, while for \(\alpha\ge 2,\beta\in \mathbb{R}_+\), these families of blow-up solutions are only finite co-dimensional stable. Moreover, these smooth profiles together shall determine the asymptotic behaviour of general smooth blow-up solutions. 
\end{remark}

\section{Mode stability of the generalized self-similar blow-up to \texorpdfstring{\eqref{eq:1.■u=(∂ₜu)²}}{\ref{eq:1.■u=(∂ₜu)²}}.}\label{sec:Mode stability of first}
In this section, we take the first step in studying the linear stability of the smooth generalized self-similar blow-up solutions \(U_{\alpha,\beta,\kappa}\) to the \eqref{eq:1.■u=(∂ₜu)²}. Specifically, we aim to establish the mode stability of the linearised equation at \( U_{\alpha,\beta,\kappa}\) demonstrating that the linearised operator has no unstable eigenvalues other than \(0\) and \(1\), which arise from the symmetries of \eqref{eq:1.■u=(∂ₜu)²}. 

Let \( U= U_{\alpha,\beta,\kappa}+ \eta(s,y)\) in the equation \eqref{eq:2.■u=(∂ₜu)² Transformation} which gives us the following linearised equation for \( \eta \): 
\begin{equation}\label{eq:3.Linearisation around Uαβκ}
       \partial_{ss} \eta +\partial_s \eta + 2 y \partial_{ys}\eta  + 2 y \partial_y \eta + (y^2 - 1) \partial_{yy} \eta  -  2 ( \partial_s U_{\alpha,\beta,\kappa}  + y \partial_y U_{\alpha,\beta,\kappa} )(\partial_s \eta + y \partial_y \eta)= (\partial_s \eta + y \partial_y \eta)^2 
 \end{equation}
Denote \(\mathbf{q} = (q_1, q_2)^\intercal = (\eta , \partial_s \eta + y \partial_y \eta)^\intercal\), then we can rewrite \eqref{eq:3.Linearisation around Uαβκ} as a first-order PDE system 
\begin{equation}\label{eq:3.Linearisation around Uαβκ PDE}
  \partial_s \begin{pmatrix}q_1 \\q_2 \\ \end{pmatrix} = \begin{pmatrix} q_2 - y \partial_y q_1 \\  - q_2 - y \partial_y q_2  + \partial_{yy} q_1 + 2 (\partial_s U _{\alpha,\beta,\kappa} + y \partial_y U _{\alpha,\beta,\kappa}) q_2   \\ \end{pmatrix} + \begin{pmatrix} 0  \\ q_2^2 \\ \end{pmatrix} = : \mathbf{L}^1_{\alpha,\beta} \begin{pmatrix}q_1 \\q_2 \\ \end{pmatrix} + \mathbf{N}\begin{pmatrix}q_1 \\q_2 \\ \end{pmatrix}
\end{equation}
where the potential 
\[ \partial_s U_{\alpha,\beta,\kappa} + y \partial_y U _{\alpha,\beta,\kappa} = \alpha + y \partial_y \tilde{U}_{\alpha, \infty, \kappa}  + y  g_{\alpha, \beta} .
\]  
In accordance with Remark \ref{rmk:2.Range of stability}, we are only interested in the range \( (\alpha,\beta) \in (0,1] \times  \{0,\infty\}  \), so from now on we restrict to this range. In this range, per the discussion under \eqref{eq:2.Def of gCB}, the potential can be simplified as follows: 
\begin{itemize}
  \item For \( 0 < \alpha \le  1 : \) 
     \begin{align}\label{eq:3.∂ₛU+y∂_yU,α,β}
         \begin{split} 
              (\partial_s + y \partial_y) U_{\alpha,0,\kappa} &=   \alpha   + \frac{\alpha \sqrt{1-\alpha } y}{1- \sqrt{1 -\alpha }y } = \frac{\alpha}{1- \sqrt{ 1- \alpha}y } , \\
              (\partial_s + y \partial_y) U_{\alpha,\infty,\kappa} &=   \alpha  -  \frac{\alpha \sqrt{1-\alpha } y}{1+ \sqrt{1 -\alpha } y } =\frac{\alpha}{1 +  \sqrt{ 1- \alpha }y}.
         \end{split}
     \end{align}   
     In particular, the potential is independent of the rescaled time \(s\) and as a consequence, the linearised operator \(\mathbf{L}^1_{\alpha,\infty}\) is independent of the parameters \(T\), \( x_0\) and $\kappa$. 
\end{itemize}
Note that the linearised operator \( \mathbf{L}^1_{\alpha,\infty}\) is non-self-adjoint which makes the stability analysis of \eqref{eq:3.Linearisation around Uαβκ PDE} difficult. 
\begin{remark}\label{rmk:3.Even symmetry}
  Even symmetry leads to just studying the behaviour near \(U_{\alpha,\infty,\kappa}\) as opposed to studying \( U_{\alpha,0,\kappa}\). As such, we are only interested in the linear stability of the operator \( \mathbf{L}^1_{\alpha,\infty}\) for \(\alpha \in (0,1]\).   
\end{remark}
As per the above remark, we will restrict to merely studying the linearised operator \( \mathbf{L}_\alpha := \mathbf{L}^1_{\alpha,\infty}\).
\subsection{Unstable eigenvalues induced by symmetries}
We first study the possible unstable eigenvalues of \(\mathbf{L}_\alpha\) as a step towards proving its mode stability. Note that if \((\lambda , \mathbf{q})\) is a smooth eigenpair of \( \mathbf{L}_\alpha\), then the first component of the eigenvalue equation gives \( - y \partial_y q_1 + q_2 = \lambda q_1\) which, together with the second component, implies 
\begin{align*}
   \lambda  (\lambda + y \partial_y) q_1 &= \lambda q_2 =   (\mathbf{L}_\alpha \mathbf{q})_2 = - (1 + y \partial_y) q_2 + \partial_{yy} q_1 + 2q_2 (\partial_s + y \partial_y) U_{\alpha,\infty,\kappa}\\
   &=  - (1 + y \partial_y) (\lambda  + y \partial_y ) q_1  + \partial_{yy} q_1 + 2 (\partial_s  U_{\alpha,\infty,\kappa}+ y \partial_y U_{\alpha,\infty,\kappa}) (\lambda + y \partial_y)q_1\\
   &= - \lambda q_1 - (\lambda + 2) y \partial_y q_1 - y^2 \partial_{yy} q_1   + \partial_{yy}q_1  + 2(\lambda q_1 + y \partial_y q_1)\frac{\alpha}{1+ \sqrt{ 1- \alpha} y }.
\end{align*} 
The above equation can be simplified as 
\begin{align*}
     \left(\lambda ^2 + \lambda - \frac{2 \alpha \lambda}{1+ \sqrt{ 1- \alpha } y} \right) q_1  + \left(2\lambda + 2 -  \frac{2 \alpha }{1+  \sqrt{ 1- \alpha }y }  \right) y \partial_y q_1  + (y^2 - 1) \partial_{yy} q_1 = 0 .
\end{align*}
Conversely, if \( q_1\) satisfies the above equation then setting \( \mathbf{q} = (q_1, (\lambda + y \partial_y ) q_1)\) implies \(\mathbf{L}_\alpha \mathbf{q} = \lambda \mathbf{q}\). Therefore, we arrive at the following definition: 
\begin{definition}[Eigenvalue of \(\mathbf{L}_\alpha\)]\label{def:3.Eigenvalue for Lα}
  \(\lambda \in \mathbb{C}\) is called an \textbf{eigenvalue} of \( \mathbf{L}_\alpha\) if there exists \( 0 \neq \phi \in C^\infty[-1,1]\) solving 
  \begin{equation}\label{eq:3.Eigenequation for Lα}
     \left(\lambda ^2 + \lambda - \frac{2 \alpha \lambda}{1+ \sqrt{ 1- \alpha }y} \right) \phi  + \left(2\lambda + 2 -  \frac{2 \alpha }{1+ \sqrt{ 1- \alpha }y}  \right) y \partial_y \phi  + (y^2 - 1) \partial_{yy} \phi = 0 . 
  \end{equation}
  We say an eigenvalue \(\lambda\) is \textbf{stable} if its real part is negative, i.e. \(\mathbb{R}{\rm e\,} \lambda < 0\) and otherwise \textbf{unstable}.
\end{definition} 
\begin{remark}[Relaxation of Regularity]\label{rmk:3.Relaxation of Regularity}
  It is equivalent to define the eigenvalues of \(\mathbf{L}_{\alpha}\) such that the corresponding eigenfunctions are in the Sobolev space \( H^k(-1,1)\) for some large \(k\). This is due to the fact that the eigen-equation is a second-order elliptic solution whose solution are always \( C^\infty\)-smooth in \( (-1,1)\). Moreover, following the same idea in \cite{ghoul2025blow}, one can show that any \( H^k(-1,1)\) solution to \eqref{eq:3.Eigenequation for Lα} is also smooth (and even analytical) at the end points $\pm 1$.  
\end{remark}
The symmetries of \eqref{eq:2.■u=(∂ₜu)² Transformation} lead to explicit unstable eigenvalues. More precisely, if \(U(s,y)\) is a solution to \eqref{eq:2.■u=(∂ₜu)² Transformation} and \( 0 \leq t_0 \leq t_1 < \infty\) then the following functions are also a solution to \eqref{eq:2.■u=(∂ₜu)² Transformation}: 
\begin{itemize}
  \item \textbf{Space translation in \(x\):} For any \( a \in \mathbb{R} : \) 
  \begin{align*}
      U_1(s,y) \colon  [ - \log  (t_0 - t_1) , +\infty)& \times ( - a e^s - 1, - ae^s + 1)  \longrightarrow  \mathbb{R}\\
              (s,y) &\longmapsto     U(s, y + ae^s).  
  \end{align*}  
  \item \textbf{Time translation in \(t\):} For any \( b \leq t_0 - t_1 :\)  
  \begin{align*}
      U_2(s,y) \colon [- \log  (t_0 - t_1 - b) , + \infty) &\times [ -1 - be^s, 1 + b e^s] \longrightarrow \mathbb{R} \\
              (s,y) &\longmapsto     U\left(s - \log  (1 + be^s) , \frac{y}{1 + be^s}\right).  
  \end{align*}
  \item \textbf{Time translation in \(s\):} For any \(c \in \mathbb{R}: \)
  \begin{align*}
      U_3(s,y) \colon [ - \log (t_0 - t_1) - c, + \infty) &\times [ -1 , 1] \longrightarrow \mathbb{R} \\
              (s,y) &\longmapsto     U(s + c, y).  
  \end{align*}
  \item \textbf{Translation in \(U\):} For any \(c \in \mathbb{R}: \) 
  \begin{align*}
      U_4(s,y) \colon [ -\log  (t _0 - t_1) , + \infty)& \times  [-1,1] \longrightarrow \mathbb{R} \\
              (s,y) &\longmapsto     U(s,y) + c.  
  \end{align*}
  \item \textbf{Scaling:} For any \(\lambda \in \mathbb{R}: \) 
  \begin{align*}
      U_5(s,y) \colon &[ - \log (t_0 - t_1 - (\lambda - 1)T) - \log  \lambda, + \infty) \times [ -1,1] \longrightarrow \mathbb{R} \\
             & (s,y) \longmapsto     U\left(s - \log  (1 + (\lambda - 1 ) T e^s) + \log  \lambda, \frac{y}{1 + (\lambda - 1) Te^s}\right).  
  \end{align*}
  This invariance can be viewed as a composition of the time translation in \(t \) and \(s\) with \( b= (\lambda - 1 )T\) and \( c = \log  \lambda\). 
\end{itemize} 
In particular, if $U = U_{\alpha,\infty,\kappa}$, then \( U_1 := U(s,y+ae^s)\) is a solution to \eqref{eq:2.■u=(∂ₜu)² Transformation} which can be expressed as \( \partial_s \mathbf{q} = (\mathbf{L} + \mathbf{N}) \mathbf{q} \) with \( \mathbf{q} = (U_1, \partial_s + y \partial_y U_1)\). Taking the derivatives with respect to the parameter \(a\), we obtain 
\[     \partial_a \partial_s \mathbf{q}|_{a = 0} = \partial_a (\mathbf{L} + \mathbf{N}) \mathbf{q}|_{a = 0} = \mathbf{L} \partial_a \mathbf{q}|_{a = 0} + \partial_{\mathbf{q}} \mathbf{N}( \mathbf{q}) \partial_a \mathbf{q}|_{a = 0} = \mathbf{L}_{\alpha} \partial_a \mathbf{q}|_{a = 0}  
\] 
and thus 
 \[ \mathbf{L}_{\alpha} (\partial_a \mathbf{q})|_{a = 0} = (\partial_a \partial_s \mathbf{q})|_{a = 0} = (\partial_a (a e^s \partial_y \mathbf{q}))|_{a = 0} = e^s \partial_y \mathbf{q}|_{ a = 0}  = (\partial_a \mathbf{q})|_{a = 0}. 
 \] 
In particular, we can remark the following: 
\begin{remark}\label{rmk:3.Eigenfunctions for Lα} Using the above strategy we can show that 
  \begin{itemize}
  \item \( \partial_a U_1|_{a = 0}\) and \(\partial_b U_2|_{b=0}\) generate a one-dimensional eigenspace for \( \mathbf{L}_{\alpha}\) with eigenvalue \(1\), 
  \item \(\partial_c U_3|_{c=0}\) generates a one-dimensional eigenspace for \(\mathbf{L}_{\alpha}\) with eigenvalue \( 0\). 
\end{itemize} 
\end{remark}

In summary, the linearised operator \(\mathbf{L}_{\alpha}\) has at-least two unstable eigenvalues \( 0\) and \( 1\) each with an explicit eigenfunction. To account for these unstabilities, we introduce the following notion of mode stability 
\begin{definition}[Mode Stability]\label{def:3.Mode Stability}
  We say that the blow-up solution \( U_{\alpha,\infty,\kappa}\) is \textbf{mode-stable} if the eigenvalues corresponding to the linearised operator \(\mathbf{L}_{\alpha}\) are either stable or 0 or 1. 
\end{definition}

\subsection{Eigen-equation for \texorpdfstring{\(\mathbf{L}_\alpha\)}{Lα}}
To study the mode stability of \( U_{\alpha,\infty,\kappa} = \alpha s  + \tilde{U}_{\alpha,\infty} + \kappa \) with \( 0 < \alpha \leq  1\), we need to solve the corresponding eigen-equation \eqref{eq:3.Eigenequation for Lα} which we recall: 
\begin{equation} 
  (y^2 - 1) \phi''  + \left(2\lambda + 2  - \frac{2 \alpha }{ 1+ \sqrt{1- \alpha}y} \right) y \phi' +\lambda \left(\lambda + 1  - \frac{2  \alpha }{1+ \sqrt{1- \alpha }y  }\right) \phi  = 0 . 
\end{equation} 
 This can be written as the Heun type ODE: 
\begin{equation}\label{eq:3.Eigenequation for Lα Heun}
 \phi''  + \left[  \frac{\lambda - \sqrt{1-\alpha}}{y + 1} + \frac{\lambda + \sqrt{1-\alpha}}{y - 1}  +\frac{2}{y  + \frac{1}{\sqrt{ 1-\alpha}} } \right] \phi' +  \frac{\lambda(\lambda + 1) y + \frac{\lambda(\lambda + 1 - 2\alpha)}{\sqrt{ 1-\alpha}}}{(y-1)(y + 1)(y + \frac{1}{\sqrt{ 1-\alpha}})} \phi = 0 .
\end{equation} 
Note that \eqref{eq:3.Eigenequation for Lα Heun} has four regular singular points \( \pm 1 , - \frac{1}{\sqrt{ 1-\alpha}}\) and \( \infty\). The standard form of a Heun type equation reads as 
\begin{equation}\label{eq:3.Standard Heun type ODE}
   g'' (z) + \left[ \frac{\gamma}{z} + \frac{\delta}{z - 1} + \frac{\varepsilon}{z - d}\right] g'(z) + \frac{abz - c}{z(z - 1) (z - d)}  g(z) = 0. 
\end{equation} 
where \( a,b,c,d,\gamma,\delta,\varepsilon \in \mathbb{C}\) and \( a + b + 1 = \gamma + \delta + \varepsilon\). Letting \( z := \frac{y + 1}{2}\) then \eqref{eq:3.Eigenequation for Lα Heun} can be written in the standard form  
\begin{align}\label{eq:3.Eigenequation for Lα Heun in z}
    \begin{split} 
            &\phi''  + \left[  \frac{\lambda - \sqrt{1-\alpha}}{z} + \frac{\lambda + \sqrt{1-\alpha}}{z - 1}  +\frac{2}{z -\frac{1}{2}(1 - \frac{1}{\sqrt{ 1-\alpha}}) } \right] \phi' \\
            & \qquad\qquad\qquad\qquad\qquad\qquad\qquad\qquad+  \frac{\lambda(\lambda + 1) z + \frac{\lambda(\lambda + 1 )}{2} \left( \frac{1}{\sqrt{ 1-\alpha}}-1\right) -  \frac{\alpha}{\sqrt{ 1-\alpha}}}{z(z-1)\left(z -\frac{1}{2}(1 - \frac{1}{\sqrt{ 1-\alpha}})\right)} \phi = 0 .
    \end{split}
\end{align}  
where we take \(\gamma = \frac{1}{2}(\lambda - \alpha  + 1 - \frac{\alpha}{1 - \frac{1}{\sqrt{ 1-\alpha}}}), \delta = \frac{1}{2} (\lambda - \alpha + 1 - \frac{\alpha}{1  + \frac{1}{\sqrt{ 1-\alpha}}}) , \varepsilon = -1,\) \( d = \frac{1}{2} (1 - \frac{1}{\sqrt{ 1-\alpha}}), \) \( c  = \frac{\lambda}{8}\left[(\lambda + 1 - 4\alpha  ) - \frac{\lambda + 1 - 2\alpha}{\sqrt{ 1-\alpha}} \right] \) and \(a,b\) are determined by solving 
\[  \begin{cases}  a +b + 1 = \gamma + \delta + \varepsilon , \\ ab = \frac{\lambda}{4}(\lambda + 1 - 4\alpha  ). \end{cases}   
\] 
where we take \(\gamma =\lambda - \sqrt{1-\alpha}, \delta = \lambda + \sqrt{1-\alpha} , \varepsilon = 2,  d = \frac{1}{2} (1 - \frac{1}{\sqrt{ 1-\alpha}}), \) \( c  = -\frac{\lambda(\lambda + 1 )}{2} \left( \frac{1}{\sqrt{ 1-\alpha}}-1\right) +  \frac{\alpha}{\sqrt{ 1-\alpha}} \), \(a = \lambda, \) and \(b = \lambda+1\). 
The existence of smooth solutions to such type of ODEs remains largely open and is commonly known as the ``connection problem'' in the literature, see for instance \cite{donninger2024spectral}. However, we bypass this difficulty by using the Lorentz transformation in self-similar variables to convert Heun type ODE \eqref{eq:3.Eigenequation for Lα Heun in z} into a new hypergeometric ODE.  

\subsubsection{Mode Stability of \texorpdfstring{\(\mathbf{L}_\alpha\)}{L¹C∞}}
The idea is motivated by the observation that the generalized self-similar blow-up solution \( U_{\alpha,\infty,\kappa}\) can be transformed to the ODE blow-up solution of a new derivative non-linear wave equation using the Lorentz transformation. More precisely, recall the Lorentz transformation and its inverse 
\begin{equation}\label{eq:3.Lorentz Transformation}
  \begin{cases} t' = \frac{t - \gamma x}{\sqrt{ 1 - \gamma ^2}}, \\ x' = \frac{x - \gamma t}{\sqrt{ 1 - \gamma ^2}}, \end{cases}   \quad \text{and, } \quad \begin{cases} t = \frac{t' + \gamma x'}{\sqrt{ 1- \gamma ^2}}, \\ x= \frac{x' + \gamma t' }{\sqrt{ 1- \gamma ^2}}, \end{cases}  \quad  \text{ where } \gamma \in (-1,1). 
\end{equation}
Let \( v(t',x') := u(t,x)\) and if \(u\) solves \eqref{eq:1.■u=(∂ₜu)²} then \(v\) solves 
\begin{equation}\label{eq:3.■v=(∂ₜv)²}
  \partial_{t't'} v - \partial_{x'x'} v = \frac{1}{1 - \gamma ^2} \left(\partial_{t'} v - \gamma \partial_{x'}v\right)^2.
\end{equation}
And under this transformation, the generalized self-similar blow-up solutions  
\[  U_{\alpha,\infty,\kappa } =  - \alpha \ln (T - t)  -\alpha \ln  ( 1 + \sqrt{ 1-\alpha}  \frac{x - x_0}{T - t}) = - \alpha \ln  \left(T - t + \sqrt{1 - \alpha}(x-x_0)\right) + \kappa 
\] 
corresponds to 
\[  V_{\alpha,\infty,\kappa} := - \alpha \ln  \left(T +\frac{ \gamma \sqrt{ 1 - \alpha} -1}{\sqrt{1 - \gamma ^2}} t' + \frac{\sqrt{1- \alpha}-\gamma}{\sqrt{ 1 -\gamma ^2}}x ' - \sqrt{ 1 - \alpha} x_0  \right) + \kappa. 
\] 
In particular, if we take \( \gamma =  \sqrt{ 1-\alpha}\) then \( V_{\alpha,\infty,\kappa}\) is independent of the spatial variable \(x'\) and is thus an ODE blow-up solution to \eqref{eq:3.■v=(∂ₜv)²}. However, it is not immediate that the stability of \( U_{\alpha,\infty,\kappa}\) directly corresponds to stability of \(V_{\alpha,\infty,\kappa}\) for the following reasons: 
\begin{itemize}
  \item The initial perturbations on \( U_{\alpha,\infty,\kappa}\) and \( V_{\alpha,\infty,\kappa}\) are imposed on different regions: the perturbations to \( U_{\alpha,\infty,\kappa}\) is posited on \( \{(t, x) \in \mathbb{R}^2 : t = 0\}\) whereas for \( V_{\alpha , \infty,\kappa}\), they are posited on \( \{ (t,x) : t' = 0  \} = \{(t,x) : t = \gamma x\}\). 
  \item Stable evolution in \(t'\) does not necessarily imply stable evolution in \(t\). 
\end{itemize} 
Nevertheless, the advantage of such a transformation is that it allows us to study the spectrum of the linearised operator at \( U_{\alpha,\infty,\kappa}\) by analysing the linearised operator at \( V_{\alpha,\infty,\kappa}\) in self-similar variables, thereby facilitating a rigorous proof of the mode stability of \( U_{\alpha,\infty,\kappa}\). To that end, we consider the Lorentz transformation. 

For \( (T',x') \in \mathbb{R}_+ \times \mathbb{R}\), define the corresponding self-similar variables for \( t' \) and \(x'\) as 
\[   s'  = - \log  (T'  -t')  , \quad y' = \frac{x'-x_0}{T'- t'}  
\] 
and let \( V(s',y') = v(t',x')\) then if \(v\) solves \eqref{eq:3.■v=(∂ₜv)²}, \( V\) solves 
\begin{equation}\label{eq:3.■v=(∂ₜv)² Transformation}
  \partial_{s's'} V + \partial_{s'} V + 2 y' \partial_{s' y'} V  + 2 y' \partial_{y'} V + ({y'}^2 - 1) \partial_{y'y'} V  = \frac{1}{1 - \gamma ^2} \left( (\partial_{s'} + y' \partial_y')V - \gamma \partial_{y'}V \right)^2. 
\end{equation}
A direct calculation implies the following: 
\begin{proposition}[Lorentz Transform in self-similar variables]\label{prop:3.Lorentz Transform in self-similar variables}
For \( \gamma \in (-1,1),\) define the change of variables  
\[  V (s',y') = U(s,y)  = U\left(s' - \log \frac{1 - \gamma y'}{\sqrt{ 1 - \gamma ^2}} , \frac{y' - \gamma}{1 - \gamma y'}\right). \]
Then \( U\) solves \eqref{eq:2.■u=(∂ₜu)² Transformation} for \(y \in (-1,1)\) if and only if \(V\) solves \eqref{eq:3.■v=(∂ₜv)² Transformation} for \( y' \in (-1,1)\). 
\end{proposition}
Under the self-similar variables \((s',y')\) the generalized self-similar blow-up solution can be expressed as 
\begin{align*}
    V_{\alpha,\infty, \kappa}(s',y') &= U_{\alpha,\infty,\kappa}\left(s' - \log  \frac{1 - \gamma y'}{\sqrt{ 1- \gamma ^2}}, \frac{y' - \gamma}{1 - \gamma y'}\right) \\
    &= \alpha s' - \alpha \log  \frac{1 - \gamma y'}{\sqrt{ 1 - \gamma ^2}}  -\alpha \ln  \left(1  + \sqrt{1 - \alpha} \frac{y' - \gamma}{1 - \gamma y'}\right) + \kappa \\
    &= \alpha s' - \alpha \ln   \frac{1}{\sqrt{1 - \gamma ^2}} \left(1 - \gamma y'  + \sqrt{ 1 - \alpha } (y' - \gamma)\right) + \kappa \\
    &= \alpha s' - \alpha \ln \frac{1}{\sqrt{1 - \gamma ^2}}\left( 1 - \sqrt{1 - \alpha}\gamma + y' ( \sqrt{ 1- \alpha}-\gamma)\right)  +\kappa. 
\end{align*} 
From now onwards, we fix \( \gamma =  \sqrt{ 1 -\alpha}\). Then 
\[  V_{\alpha,\infty,\kappa}  = \alpha s' - \alpha \ln  \frac{1}{\sqrt{\alpha}} ( \alpha - 0) + \kappa = \alpha s' - \alpha \ln  \sqrt{ \alpha } + \kappa 
\] 
is independent of \(y'\)-variable. To study the linearised operator of \eqref{eq:3.■v=(∂ₜv)² Transformation}, we introduce \(V = V_{\alpha,\infty,\kappa} + \xi\) to obtain   
 \begin{equation}\label{eq:3.Linearisation around V}
  \partial_{s's'} \xi - \partial_{s'} \xi  + 2 y' \partial_{y's'} \xi + 2 \sqrt{ 1 - \alpha} \partial_{y'}\xi  + ({y'}^2 - 1) \partial_{y'y'}\xi =0 . 
 \end{equation}
 Writing \(\mathbf{r} = (r_1, r_2) = (\xi , (\partial_{s'} + y' \partial_{y'} )\xi)\) we can express \eqref{eq:3.Linearisation around V} as a first-order PDE system 
 \begin{equation}\label{eq:3.Linearisation around V PDE}
        \partial_s \begin{pmatrix}r_1 \\r_2 \\ \end{pmatrix} = \begin{pmatrix} r_2 - y' \partial_{y'} r_1 \\ r_2 - y' \partial_{y'} r_2  + \partial_{y'y'} r_1  - 2 \sqrt{ 1 - \alpha} \partial_{y'} r_1\\ \end{pmatrix} = : \mathbf{L}'_\alpha\begin{pmatrix}r_1 \\r_2 \\ \end{pmatrix}. 
 \end{equation}
 As before, we introduce the notion of an eigenvalue for this linearised operator 
 \begin{definition}[Eigenvalue for \(\mathbf{L}'_\alpha\)]\label{def:3.Eigenvalue for L'α}
  \(\lambda \in \mathbb{C}\) is an \textbf{eigenvalue} of \(\mathbf{L}'_\alpha\) if there exists \( 0 \neq \phi \in C^\infty[-1,1]\) such that 
  \begin{equation}\label{eq:3.Eigenequation for L'α}
     (\lambda ^2 - \lambda) \phi  + (2\lambda y' +2 \sqrt{ 1- \alpha}) \partial_{y'} \phi + ({y'}^2 - 1) \partial_{y'y'}\phi = 0. 
  \end{equation}
 \end{definition} 
 Same as Proposition \ref{prop:3.Lorentz Transform in self-similar variables}, the Lorentz transformation transforms the linearised equation at \( U_{\alpha,\infty,\kappa}\) to the linearised equation at \( V_{\alpha,\infty,\kappa}\). 
 \begin{proposition}[Lorentz Transformation of the Linearised Operator]\label{prop:3.Lorentz Transformation of the Linearised Operator}
  \(\eta\) satisfies the linearised-near-\(U_{\alpha,\infty,\kappa}\) \eqref{eq:3.Linearisation around Uαβκ} iff \(\xi\) satisfies the linearised-near-\(V_{\alpha,\infty,\kappa}\) \eqref{eq:3.Linearisation around V}. 
 \end{proposition}
 In particular, we obtain the point-spectrums of \( \mathbf{L}_\alpha\) coincides with that of \( \mathbf{L}'_\alpha\). 
 \begin{corollary}[\(\sigma _p(\mathbf{L}_\alpha) = \sigma _p(\mathbf{L}'_\alpha)\)]\label{cor:3.σₚ(Lα,0,κ)=σₚ(L'α)}
  \(\lambda\) is an eigenvalue for \( \mathbf{L}_\alpha\) if and only if \( \lambda \) is ane eigenvalue for \(\mathbf{L}'_\alpha\). 
 \end{corollary}
 \begin{proof} 
  If \(\lambda\) is an eigenvalue for \( \mathbf{L}_\alpha\) then per Definition \ref{def:3.Eigenvalue for Lα}, there exists a non-trivial smooth function \(\phi \in C^\infty[-1,1]\) such that \( e^{\lambda s} \phi\) satisfies the linearised equation \eqref{eq:3.Linearisation around Uαβκ}. Per Proposition \ref{prop:3.Lorentz Transformation of the Linearised Operator}, the function 
  \[  (e^{\lambda s } \phi(y))\left( s' - \ln  \frac{1 - \gamma y'}{ \sqrt{ 1 - \gamma ^2}} , \frac{y' - \gamma}{1 - \gamma y'} \right) = e^{\lambda s'} \left( \frac{ 1 - \gamma y'}{\sqrt{ 1- \gamma ^2}}\right)^{-\lambda} \phi\left(\frac{y' - \gamma}{1 - \gamma y'}\right)
  \] 
  satisfies the linearised equation \eqref{eq:3.Linearisation around V} with \(\gamma =  \sqrt{ 1 -\alpha}\) implying, per Definition \ref{def:3.Eigenvalue for L'α}, that \( \lambda\) is an eigenvalue for \( \mathbf{L}'_\alpha\). By the same argument run in reverse with the inverse of Lorentz transformation, we obtain the converse. 
 \end{proof}  
 \begin{remark}\label{rmk:3.Transformation of eigenequation from Lα to L'α}
  In fact the transformation 
  \[  \phi \longrightarrow  \left( \frac{ 1 - \gamma y'}{\sqrt{ 1- \gamma ^2}}\right)^{-\lambda} \phi\left(\frac{y' - \gamma}{1 - \gamma y'}\right), \quad \gamma = \sqrt{ 1- \alpha}. 
  \] 
  transforms the eigen-equation for \( \mathbf{L}_\alpha\) to an eigen-equation for \( \mathbf{L}'_\alpha\). 
 \end{remark}
 As such, the mode stability of \( U_{\alpha,\infty,\kappa}\) can be reduced to studying the mode stability of \( V_{\alpha,\infty,\kappa}\). In particular, instead of studying the eigen-equation \eqref{eq:3.Eigenequation for Lα} which has a complicated Heun-type ODE structure \eqref{eq:3.Eigenequation for Lα Heun}, we can study \eqref{eq:3.Eigenequation for L'α} which has a hypergeometric-type structure which is more simpler as its connection problem is well-posed. Moreover, while previously, we had four singular points, now we only have three singular points at \( \pm 1\) and \(\infty\). 

 More precisely, the eigen-equation for \( \mathbf{L}'_\alpha\) is written as 
 \[  (\lambda ^2 - \lambda) \phi + (2 \lambda y' + 2 \sqrt{ 1 - \alpha})\phi'  + ({y'}^2 - 1) \phi ''= 0. 
 \] 
 This can be written as the hyper-geometric ODE:  
 \begin{equation}\label{eq:3.Eigenequation for L'α HypGeom}
  (y' - 1) (y' + 1)\partial_{y'y'} \phi  + (  2 \sqrt{ 1- \alpha} + 2 \lambda y')\partial_{y'} \phi +(\lambda ^2 - \lambda) \phi = 0. 
 \end{equation}
 Indeed, the standard form of a hypergeometric differential equation is given by 
 \begin{equation}\label{eq:3.Standard hypergeometric ODE}
   z' (1 - z') \partial_{z'z'} \psi + (c - (a + b  + 1) z') \partial_{z'} \psi - ab \psi = 0 , \quad a,b,c \in \mathbb{C}.  
 \end{equation} 
 Letting \( y' := 2z ' - 1\) and \(\phi(y') := \psi(z')\), we can rewrite \eqref{eq:3.Eigenequation for L'α HypGeom} as 
 \begin{equation}\label{eq:3.Eigenequation for L'α HypGeom in z'}
   z'(1 - z' ) \partial_{z'z'} \psi  + (  \lambda - \sqrt{ 1 -\alpha}  -2 \lambda  z') \partial_{z'} \psi  - \lambda(\lambda - 1) \psi = 0. 
 \end{equation}
 where we have set \(a = \lambda - 1, b =\lambda \) and \( c= \lambda -\sqrt{ 1 -\alpha}\). 
 As a consequence, we will prove the following: 
 \begin{proposition}\label{prop:3.Mode stability of Uα∞κ}
   For all \(\alpha \in (0,1]\) there exists no non-trivial \(\psi \in C^\infty[-1,1]\) solving \eqref{eq:3.Eigenequation for L'α HypGeom in z'} for \( \mathbb{R}{\rm e\,} \lambda > -1\) except for \( \lambda = 0\) or 1. In particular, in light of Corollary \ref{cor:3.σₚ(Lα,0,κ)=σₚ(L'α)}, \(U_{\alpha,\infty,\kappa}\) is mode-stable. 
 \end{proposition}
 \begin{remark}\label{remark:spectral gap}
     The result goes beyond mode stability, demonstrating that the spectral gap between the stable and unstable eigenvalues is at least 1.  
 \end{remark}

 To prove the above proposition, we will be needing the standard Frobenius analysis. 
 \begin{lemma}[Frobenius Theory]\label{lem:3.Frobenius Theory}
  Consider the second-order differential equation 
  \begin{equation}\label{eq:3.General 2nd-order ODE}
     f''(z) + p(z) f'(z) + q(z) f(z) = 0  
  \end{equation}
  where \( p,q \in H(\mathbb{D}_R')\) are holomorphic functions with at-most a first and second order pole at \(0\) respectively. Let \(s _\pm\) be the solutions to the \textbf{indicial equation} 
  \[  s ^2 + (p_0 - 1) s  + q_0 = 0 , \quad p_0 := \lim_{z \to 0} zp(z) , \quad q_0 := \lim_{z \to 0} z^2 q (z) 
  \] 
  satisfying the ordering \( \mathbb{R}{\rm e\,}(s _+ ) \geq \mathbb{R}{\rm e\,}(s _-)\). Then there exists \( h_\pm \in H(\mathbb{D}_R) : h_\pm(0) = 1\) such that the fundamental system of solutions of \eqref{eq:3.General 2nd-order ODE} on \( \mathbb{D}_R \backslash \mathbb{R}_-\) is given by 
  \[  f_+ (z) := z^{s_+} h_+(z) , \quad f_-(z) := z^{s_-} h_-(z) + c_{s_+ - s_-} f_+(z) \log  (z) 
  \] 
  where 
  \[  c_{\omega} \in  \begin{cases}  \{0\} & \text{if}\quad\omega \not\in \mathbb{N}, \\ 
  \mathbb{C} & \text{if}\quad\omega \in \mathbb{Z}_+ , \\
  \mathbb{C}\backslash \{0\} & \text{if}\quad\omega = 0 \end{cases}   
  \] 
 \end{lemma}
 Dividing by \( z'(1 - z')\) the equation \eqref{eq:3.Eigenequation for L'α HypGeom in z'}, we obtain 
 \begin{equation}\label{eq:3.Eigenequation for L'α HypGeom in z' monomial}
  \partial_{z'z'} \psi + \frac{\lambda - \sqrt{ 1 - \alpha} - 2\lambda z'}{z'(1 - z')} \partial_{z'} \psi  + \frac{\lambda(1 - \lambda)}{z' (1 - z')} \psi = 0 
 \end{equation}
 and thus 
 \[  p(z) =   \frac{\lambda - \sqrt{ 1 - \alpha} - 2\lambda z'}{z'(1 - z')} , \quad q(z) := \frac{\lambda(1 - \lambda)}{z' (1 - z')}. 
 \] 
 The Frobenius analysis allows us to construct a solution of the form 
 \begin{equation}\label{eq:3.Frobenius Solution}
  \psi(z') := (z' - z_0 ' )^s \sum_{n = 0}^\infty \psi _k ( z' - z'_0) ^n,  
 \end{equation}
 where the solution holds in a local (depending on the convergence radius of \(p\) and \(q\)) neighbourhood of the singularity \( z_0 \in \{0,1\}\). The idea now is to show that for \(\lambda : \mathbb{R}{\rm e\,} \lambda > - 1\), the solutions obtained via the Frobenius analysis will fail to be smooth on \([-1,1]\) except possibly when \(\lambda = 0\) or \(1.\)
 \subsubsection{Frobenius analysis at \texorpdfstring{\(z'_0 =0\)}{zᐟ₀=0}.}
 For the singularity \( z_0' := 0\), the indicial equation is given as 
 \[  s ^2 + ( \lambda - \sqrt{1 - \alpha} - 1) s  = 0  
 \] 
 which implies 
 \[   \begin{cases} s_+ = 1 - \lambda + \sqrt{ 1- \alpha} , \quad s_- =  0 & \text{if }  -1< \mathbb{R}{\rm e\,}\lambda \leq 1 + \sqrt{1-\alpha}, \\
s_+ = 0 , \quad s_- = 1 - \lambda +\sqrt{ 1 - \alpha}&\text{if } \mathbb{R}{\rm e\,} \lambda > 1 + \sqrt{1-\alpha}.   \end{cases}  
 \] 
 Moreover set 
\[  h_1(z') :=   F\left( \begin{matrix}      \sqrt{1 - \alpha}  , 1  + \sqrt{1 - \alpha}   \\2 - \lambda + \sqrt{ 1- \alpha} \end{matrix} \, ; z'\right) ,\quad h_2(z) :=    F\left( \begin{matrix} \lambda - 1,\lambda \\ \lambda - \sqrt{1 - \alpha} \end{matrix} \, ; z'\right)  
    \] 
    where \(F \left( \begin{matrix} a,b \\ c  \end{matrix}\,  ; z'\right) \) is the Gauss hypergeometric function which is holomorphic in \( \mathbb{D}_1\). We can now have the following possible situations: 
 \begin{itemize}
  \item If \( -1 < \mathbb{R}{\rm e\,} \lambda \leq 1 + \sqrt{ 1 -\alpha}\) and  
  \begin{itemize}
    \item[\(>\)] If \( s_+ - s_- = s_+ \not\in \mathbb{N}\) then the two independent solutions for \eqref{eq:3.Eigenequation for L'α HypGeom in z' monomial} are given by 
    \[  \psi_+ (z): = {z'}^{s_+} h_1 (z'), \quad \psi_-(z) :=  h_2(z') . 
    \]  
     Note that \( \psi _+\) is not smooth at \( z' = 0\) since \( s_+ \not\in \mathbb{N}\). Thus any local smooth solution around \(0\) must be a constant multiple of \(h _-\) and hence analytic at \( z '=0 \). 
    \item [\(>\)] If \(s_+ \in \mathbb{N}\) then the two independent solutions for \eqref{eq:3.Eigenequation for L'α HypGeom in z' monomial} are given by 
    \[  \psi _+ (z) := {z'}^{s_+} h_1(z') , \quad \psi _-(z) := h_2(z')  + c _{s_+}  \psi _+(z') \log  z ' 
    \] 
    Now note that since \( s_+ \in \mathbb{N}\) and \(\alpha \in (0,1]\) so \( \lambda \in \mathbb{R}\) and thus 
    \[ 0 \leq  s_+  = 1 + \sqrt{1 - \alpha} - \lambda  = 1 + \sqrt{1 - \alpha} - \mathbb{R}{\rm e\,} \lambda < 2 + 1 =3 . 
    \] 

    If \( c_{s _+} \neq 0\) then since \( \log  z'\) fails to be smooth at \( 0\) so any \(C^\infty[-1,1]\) solution to \eqref{eq:3.Eigenequation for L'α HypGeom in z' monomial} must be a scalar multiple of \(\psi _+\). Otherwise any smooth solution to \eqref{eq:3.Eigenequation for L'α HypGeom in z' monomial} is a linear combination of \( \psi _+\) and \(\psi _-\) both of which are analytic at \( 0\). In particular, any smooth near 0 solution to \eqref{eq:3.Eigenequation for L'α HypGeom in z' monomial} is analytic at \(0\). 
    
  \end{itemize} 
  Thus any smooth solution around \( 0\) must be analytic at \( 0\). 

  \item If \( \mathbb{R}{\rm e\,} \lambda > 1 + \sqrt{ 1 -\alpha}\) then the two independent solutions for \eqref{eq:3.Eigenequation for L'α HypGeom in z' monomial} are given by 
   \[  \psi _+(z') := h_2(z') , \quad \psi _-(z') := {z'}^{s_-} h_1(z') + c_{- s_-} \psi _+(z') \log  z' . 
   \] 
   If \(s_- \neq 0\) then \( c_{-s_-} = 0\) and as \(\mathbb{R}{\rm e\,} s_- = 1 + \sqrt{1 - \alpha} - \mathbb{R}{\rm e\,} \lambda < 0 \) so \({z'}^{s_-}\) has a singularity at \(0\) which cannot be cancelled out by \( h_1\) as \( h_1(0)  = 1\). In particular, any smooth near \(0\) solution to \eqref{eq:3.Eigenequation for L'α HypGeom in z' monomial} is a scalar multiple of \(\psi_+\). Similarly, if \( s_- = 0\) then \(c_{-s_-} \neq 0 \) and thus \( \psi _-\) fails to be smooth near \(0\) due to the logarithmic singularity. Again, any smooth near \(0\) solution is a scalar multiple of \(\psi_+\). 
   Therefore, we have shown that any smooth near \(0\) solution to \eqref{eq:3.Eigenequation for L'α HypGeom in z' monomial} is a scalar multiple of \(h_-\) and thus analytic at $0$. 
 \end{itemize}

 \subsubsection{Frobenius analysis at \texorpdfstring{\(z'_0 =1\)}{zᐟ₀=1}.} 
 For the singularity \( z'_0 := 1,\) we first shift the coordinates to relocate the singularity at \(1\) to be at \(0\) by introducing the change of variables \(z' : = 1 - z''\) in which case \eqref{eq:3.Eigenequation for L'α HypGeom in z' monomial} can be written as 
 \begin{equation}\label{eq:3.Eigenequation for L'α HypGeom in z'+1 monomial}
   \partial_{z''z''} \psi  + \frac{\lambda + \sqrt{ 1- \alpha }   - 2\lambda z''  }{(1 - z'') z''} \partial_{z''}\psi-\frac{\lambda(\lambda - 1)}{(1 - z'') z''} \psi = 0. 
 \end{equation}
 Then this can be seen as a specialisation of the standard form of hypergeometric differential equation \eqref{eq:3.Standard hypergeometric ODE} with \( a = \lambda - 1, b = \lambda , c = \lambda - \sqrt{1 - \alpha}\). Thus the indicial equation is given as 
 \[  s ^2 + ( \lambda + \sqrt{1 - \alpha} - 1) s  = 0  
 \] 
 which implies 
 \[   \begin{cases} s_+ = 1 - \lambda - \sqrt{ 1- \alpha} , \quad s_- =  0 & \text{if }  -1< \mathbb{R}{\rm e\,}\lambda \leq 1 - \sqrt{1-\alpha}, \\
s_+ = 0 , \quad s_- = 1 - \lambda + \sqrt{ 1 - \alpha}&\text{if } \mathbb{R}{\rm e\,} \lambda > 1 - \sqrt{1-\alpha}.   \end{cases}  
 \] 
 Now set 
 \[  h_3 (z'') := F \left( \begin{matrix}  -\sqrt{1 - \alpha}  , 1 -  \sqrt{1 - \alpha } \\ 2 - \lambda - \sqrt{ 1- \alpha } \ \end{matrix} ; z''  \right)  , \quad h_4(z'') :=  F \left( \begin{matrix} \lambda - 1,  \lambda \\ \lambda + \sqrt{ 1- \alpha} \ \end{matrix} ; z''\right)
 \] 
 So per Lemma \ref{lem:3.Frobenius Theory}, we have the following possible situations: 
 \begin{itemize}
  \item If \( -1 <\mathbb{R}{\rm e\,} \lambda \leq 1 - \sqrt{ 1 -\alpha}\) and 
  \begin{itemize}
    \item[\(>\)] if \( s_+ -s_- = s_+ \notin \mathbb{N}\) then the fundamental system of solutions to \eqref{eq:3.Eigenequation for L'α HypGeom in z'+1 monomial} is given by 
    \[  \psi _+(z'') = {z''}^{s_+} h_3(z'') , \quad \psi _-(z'') = h_4(z'') 
    \] 
    but since \( s_+ \not\in \mathbb{N}\) thus \( {z''}^{s_+}\) is not analytic at \( 0\) and hence any smooth solution at \(0\) must by a scalar multiple of \( \psi _-\) and thus of \(h_4\).
    \item[\(>\)] if \( s_+ \in \mathbb{N}\) then \( \lambda \in \mathbb{R}\) and \(\alpha \in (0,1]\) implies 
    \[  s_+ = 1 - \lambda - \sqrt{ 1- \alpha} < 2 - 0 = 2 
    \] 
    and thus \(s_+ = 0\) or \(1\). If \(\alpha = 1\) then \( \lambda = 1 - s_+\) which corresponds to the case \( \lambda = 0\) or \(1\) which we ignore. So we restrict to the case \( \alpha \in (0,1)\). 
    
    The fundamental system of solutions to \eqref{eq:3.Eigenequation for L'α HypGeom in z'+1 monomial} is given by 
    \[  \psi _+(z'') := {z''}^{s_+} h_3(z'') , \quad \psi _-(z'') := h_4(z'')  + c_{s_+} \psi _+(z'') \log  z''.  
    \] 
    If \(s_+ = 0\) then \(c_{s_+} \neq 0 \) in which case a logarithmic singularity is present at \(0\) implying that any smooth-near-\(0\) solution must be a scalar multiple of \(h_3\).
    We will argue the same conclusion for when \(s_{+} = 1\). To that end, suppose \( c_{s_+} = 0\) so that there exist some \(h_- \in H(\mathbb{D}_R) : h_-(0) = 1\) solving \eqref{eq:3.Eigenequation for L'α HypGeom in z'+1 monomial}. Plugging 
    \[  \psi(z'') = h_-(z'') = 1 + \sum_{n = 1}^\infty a_n {z''}^n ,   
    \] 
    in \eqref{eq:3.Eigenequation for L'α HypGeom in z'+1 monomial} along with \( \lambda = - \sqrt{ 1- \alpha}\) (which is equivalent to \(s_+ = 1\)), we find 
    \[   \sum_{n = 2}^\infty n(n-1) a_n {z''}^n   - \frac{\sqrt{1 - \alpha}(1 + \sqrt{1 -\alpha})}{(1 - z'') z''} \left(1 + \sum_{n = 1}^\infty a_n {z''}^n \right) = 0 . 
    \] 
    This results in a contradiction near \(z'' = 0\) due to the presence of \(\frac{1}{z''}\) in the second term. Thus no such \( h_-\) exists and we must have \( c_{s_+} \neq 0 \). In particular, \(c_{s_+} \neq 0 \) so \(\psi _-\) still retains a logarithmic singularity at \( 0\) implying that any smooth-near-0 solution must be a scalar multiple of \(h_3\) or \( z''h_3\). 
  \end{itemize} 

  \item If \(\mathbb{R}{\rm e\,} \lambda > 1 - \sqrt{1-\alpha}\) then the fundamental system of solution to \eqref{eq:3.Eigenequation for L'α HypGeom in z'+1 monomial} is given by 
  \[  \psi _+(z'') := h_4(z'') , \quad \psi _-(z'') = {z''}^{s_-} h_3(z'')  + c_{-s_-} \psi _+(z'')\log  z''.  
  \] 
  Thus any smooth-near-0 solution to \eqref{eq:3.Eigenequation for L'α HypGeom in z'+1 monomial} is a linear combination of \(\psi _+\) and \(\psi _-\). If it smooth at \( 0\) then the logarithmic singularity cannot be present implying \( c_{-s_-} = 0 \implies -s_- \neq 0\). Moreover, \(-s_- \notin \mathbb{Z}_+\) thus \( c_{-s_-} = 0 \) and \(\mathbb{R}{\rm e\,}s_- \leq \mathbb{R}{\rm e\,} s_+= 0 \) so \(\mathbb{R}{\rm e\,} s_- < 0\) implying the smooth solution is a scalar multiple of \(h_4\) and thus analytic at \(0\). 
 \end{itemize} 
  In particular, we have shown that any smooth-near-1 solution to \eqref{eq:3.Eigenequation for L'α HypGeom in z' monomial} must be a scalar multiple of either \( h_3, z''h_3\) or \(h_4\). 
 


 \begin{proof}[Proof of Proposition \ref{prop:3.Mode stability of Uα∞κ}]
  For $\alpha\in (0,1]$, as per the above Frobenius analysis, we have shown that for any \( \lambda \in \mathbb{C} : \mathbb{R}{\rm e\,} \lambda > -1\) and any \( C^\infty[-1,1]\) solution \(\psi\) to \eqref{eq:3.Eigenequation for L'α HypGeom in z' monomial}, \(\psi(z')\) must also be analytic at \( \{0,1\}\). We will show that \( \psi\) must fail to be analytic at \( z' = 0\) unless \( \lambda = 0\) or \(1\). Recall that Frobenius analysis at \( z'_0 = 1\) showed that for $\alpha \in (0,1]$ and $\lambda \neq 0,1$, \(\tilde{\psi}(z'') := \psi(1-z'')\) must be a scalar multiple of either \( h_3, z'' h_3\) or \( h_4\).  Without loss of generality, let us fix 
  \[  \tilde{\psi}(z'') = h_4(z'') = F \left(\begin{matrix} \lambda - 1 , \lambda\\ \lambda +\sqrt{ 1- \alpha}  \end{matrix}\quad ; z''\right) . 
  \] 
  
  We claim that it fails to be analytic at \(z'' =1\) which then contradicts the results obtained via Frobenius analysis at \( z'_{0 } = 0\). 
  Since $h_4$ is a power series centered at 0 that satisfies the hypergeometric equation \eqref{eq:3.Eigenequation for L'α HypGeom in z'+1 monomial}, if it remains analytic at $1$, then its radius of convergence must exceed 1, because any solution of \eqref{eq:3.Eigenequation for L'α HypGeom in z'+1 monomial} is analytic everywhere except possibly at the regular singular points 0, 1, and $\infty.$ 
 However, applying the ratio tests on the coefficients of the power series of \(h_4\) shows that 
  \[  \lim_{n \to \infty}\frac{(h_4)_{n + 1}}{(h_4)_n}  = \lim_{n \to \infty} \frac{\frac{\Gamma(\lambda - 1 + n + 1) \Gamma(\lambda + n + 1)}{\Gamma( \lambda + \sqrt{ 1- \alpha} + n + 1) \Gamma(n + 1)}}{\frac{\Gamma(\lambda - 1 + n) \Gamma(\lambda + n )}{\Gamma( \lambda +  \sqrt{ 1- \alpha} + n ) \Gamma(n )}} = \lim_{n \to \infty} \frac{(\lambda - 1 + n) (\lambda +n)}{(\lambda + \sqrt{1 - \alpha} +n)n} \approx  \lim_{n \to \infty} \frac{n^2}{n^2} = 1
  \] 
  unless \( \lambda - 1 \) or \(\lambda \) equals \(0\) in which case the radius of convergence is infinite. The same argument works in the case \( \tilde{\psi}(z'') = h_3(z'')\) or \( \tilde{\psi}(z'') = z'' h_3(z'')\) and thus our conclusion is reached. 
 \end{proof}  
 
 \section{Mode Stability of the generalized self-similar blow-up of \texorpdfstring{\(\eqref{eq:1.■u=(∂ₜu)²−(∂ₓu)²}\)}{\ref{eq:1.■u=(∂ₜu)²−(∂ₓu)²}}.}\label{sec:Mode stability of second}
 In this section, akin to the previous section, we will study the linear stability of the smooth generalized self-similar blow-up solutions for \eqref{eq:1.■u=(∂ₜu)²−(∂ₓu)²}. Specifically, we will study the mode stability of the linearised equation at \( U_{\alpha ,\beta, \kappa}\) defined by \eqref{eq:2.Uαβκ for (∂ₜu)²−(∂ₓu)²}. As the details are generalized identical to the previous section, we will only provide the calculation and the reader can refer to the previous section for motivation and explanation. 

 To that end, fix \( U = U_{\alpha,\beta,\kappa} + \eta(s,y)\) in the equation \eqref{eq:2.■u=(∂ₜu)²−(∂ₓu)² Transformation} to obtain the linearised equation for \(\eta\) 
 \begin{align}\label{eq:4.Linearised equation at Uαβκ}
     \begin{split} 
          &\partial_{ss} \eta + \partial_s \eta + 2  y \partial_{ys} \eta + 2y \partial_y \eta  + (y^2 - 1) \partial_{yy} \eta  \\
          &\qquad- 2(\partial_s U_{\alpha,\beta,\kappa}  + y \partial_y U_{\alpha,\beta,\kappa})(\partial_s \eta + y \partial_y \eta)  + 2 \partial_y U_{\alpha,\beta,\kappa} \partial_y \eta  = (\partial_s \eta + y \partial_y \eta)^2 - (\partial_y \eta)^2. 
     \end{split}
 \end{align}  
 Denoting \( \mathbf{q} = (q_1, q_2) = (\eta , \partial_s \eta + y \partial_y \eta)^\intercal\), we can rewrite \eqref{eq:4.Linearised equation at Uαβκ} as a first-order PDE system as 
 
 \begin{align}\label{eq:4.Linearised equation at Uαβκ PDE}
     \begin{split} 
         \partial_s\begin{pmatrix} q_1 \\q_2 \end{pmatrix} &= \begin{pmatrix} q_2 - y \partial_y q_1 \\ - q_2 - y \partial_y q_2  + \partial_{yy} q_1  + 2(\partial_s U_{\alpha,\beta,\kappa} + y \partial_y U_{\alpha,\beta,\kappa})q_2  -2 \partial_y U_{\alpha,\beta,\kappa} \partial_y q_1 \\ \end{pmatrix} + \begin{pmatrix}0 \\  q_2^2  - (\partial_y q_1)^2 \\ \end{pmatrix} \\
         &= : \mathbf{L}^2 _{\alpha,\beta}\begin{pmatrix}q_1 \\q_2 \end{pmatrix} + \mathbf{M} \begin{pmatrix}q_1 \\q_2 \end{pmatrix} 
     \end{split}
 \end{align} 
 where the potential is given by 
 \[  \begin{cases} \partial_s U_{\alpha,\beta ,\kappa} + y \partial_y U_{\alpha,\beta,\kappa} = \alpha + y \partial_y \tilde{U}_{\alpha,\beta} , \\
\partial_y U_{\alpha,\beta,\kappa} = \partial_y \tilde{U}_{\alpha,\beta}.   \end{cases}  
 \] 
 In accordance with the Remark \ref{rmk:2.Range of stability for ∂ₓu²}, we will restrict to \(\alpha = 1, \beta \in \mathbb{R}_+\) in which case the potential simplify as 
\begin{align}\label{eq:4.∂ₛU+y∂_yU,1,β}
    \begin{split} 
        (\partial_s + y \partial_y ) U_{1,\beta,\kappa} &= 1 - \frac{y(1 - \beta)}{1 + \beta  + y(1 - \beta)}  = \frac{1 + \beta}{1 + \beta + y (1 - \beta)}, \\
        \partial_y U_{1,\beta,\kappa} &= - \frac{1 - \beta}{1 + \beta + y(1 - \beta)}. 
    \end{split}
\end{align} 
\begin{remark}\label{rmk:4.Special case of L²11κ}
   For \( \alpha >0,  \beta \in \mathbb{R}_+\), note that the potential is independent of the rescaled-time variable \(s\) and $\kappa$. Moreover, if \(\beta = 1\) then since \( \partial_y U_{1,1} = 0\) so 
   \[  \mathbf{L}^2_{1,1} = \mathbf{L}_{1} \equiv \mathbf{L}^1_{1,\infty}
   \] 
   for which the mode stability has already been studied in the previous section.  
\end{remark}
In light of the above remark, we restrict our attention to the case of \( \beta \neq 1\). 

\subsection{Mode stability of \texorpdfstring{\(\mathbf{L}^2_{1,\beta}\)}{L²1βκ} for \texorpdfstring{\(\beta \neq 1\)}{β≠1}.} 
As in the previous section, studying the mode stability of \( U_{1,\beta,\kappa}\) reduces to studying the eigenvalues of the linearised operator \(\mathbf{L}^2_{\alpha,\beta,\kappa}\) for which we then introduce the following eigen-equation 
\begin{definition}[Eigenvalue for \(\mathbf{L}^2_{1, \beta}\)]\label{def:4.Eigenvalue for L²1βκ}
  \(\lambda \in \mathbb{C}\) is called an \textbf{eigenvalue} of \(\mathbf{L}^2_{1,\beta}\) if there exists \( 0 \neq \phi \in C^\infty[-1,1]\) solving 
  \begin{align}\label{eq:4.Eigenequation for L²1βκ}
      \begin{split} 
              & \left(\lambda ^2 + \lambda - \frac{2\lambda (1 + \beta)}{1 + \beta + y (1 - \beta)} \right) \phi + \left(2\lambda y  + 2y-\frac{2(1 + \beta)y}{1 + \beta + y (1 - \beta)}   -   \frac{2(1 - \beta)}{1 + \beta + y(1 - \beta)}\right) \phi'  \\
              &\qquad\qquad\qquad\qquad \qquad\qquad\qquad\qquad \qquad\qquad\qquad\qquad \qquad\qquad\qquad\qquad \qquad + (y^2 - 1) \phi''   = 0. 
      \end{split}
  \end{align}  
  We say \(\lambda\) is \textbf{stable} if its real-part is negative, and \textbf{unstable} otherwise. Moreover, we say it is \textbf{mode-stable} if it is either stable or equal to \(0\) or \(1\). 
\end{definition}

The equation \eqref{eq:4.Eigenequation for L²1βκ} can be rewritten as a Heun-type ODE
\begin{equation}\label{eq:4.Eigenequation for L²1βκ Heun}
     \phi '' +  \left(\frac{\lambda}{y + 1} + \frac{\lambda}{y-1} + \frac{  2  }{y + \frac{1 + \beta}{1 - \beta}   } \right)\phi' + \frac{1}{(y - 1)(y+1)(y + \frac{1 + \beta}{1 - \beta} )}\left(\lambda(\lambda + 1)y + (\lambda ^2 - \lambda)\frac{1 + \beta}{1 - \beta} \right)  \phi = 0. 
\end{equation}
as it can be realised as the standard Heun-type ODE \eqref{eq:3.Standard Heun type ODE} with the change of variable \( z = \frac{y + 1}{2}\). Note that \eqref{eq:4.Eigenequation for L²1βκ Heun} has singularities at \( \pm 1,  - \frac{1 + \beta}{1 - \beta}\) and \( \infty\). We can bypass the difficulty of solving the Heun-type ODE by introducing the Lorentz transformation to convert \eqref{eq:4.Eigenequation for L²1βκ Heun} into a hypergeometric ODE. 

To that end, we introduce the Lorentz transformation and its inverse 
\begin{equation}\label{eq:4.Lorentz Transformation}
  \begin{cases} t' = \frac{t - \gamma x }{\sqrt{ 1 - \gamma ^2}}, \\ x' = \frac{x - \gamma t}{\sqrt{ 1- \gamma ^2}} \end{cases}     \quad \text{and} \quad \begin{cases} t = \frac{t ' + \gamma x'}{\sqrt{ 1- \gamma ^2}}, \\ x = \frac{x' + \gamma t'}{\sqrt{1 - \gamma ^2}} \end{cases}  \quad \text{where}\quad\gamma \in (-1,1). 
\end{equation}
As such, we have the following proposition mimicking the devolvement in the previous section. 
\begin{proposition}\label{prop:4.Lorentz Transformation summary}
  The Lorentz transformation transforms the following quantities in the following way: For fixed \((T',x') \in \mathbb{R}_+\times \mathbb{R} : \)
   \begin{align}\label{eq:4.Lorentz Transformation Table}
       \begin{split} 
           &\begin{cases} (t,x)   \\ (s,y) \\ \left(s' - \log  \frac{1 - \gamma y'}{\sqrt{ 1- \gamma ^2}}, \frac{y' - \gamma}{1 - \gamma y'}\right) \\ \partial_{tt} u - \partial_{xx} u = (\partial_t u )^2 -( \partial_x u)^2  \\ \eqref{eq:2.■u=(∂ₜu)²−(∂ₓu)² Transformation} \\
      U_{1,\beta,\kappa}   
    \end{cases}   
     \quad\textup{gets transformed into}\\
      &\begin{cases} (t',x')  &  \textup{(Coordinates)}, \\ (s',y') := (- \log  (T' - t') , \frac{x' - x_0 }{T' - t'}) &  \textup{(Self-similar variables),} \\ 
        (s',y' ) & \textup{(Inverse self-similar variables),} \\ 
        \partial_{t't'} v - \partial_{x'x'} v = (\partial_{t'}v)^2 - (\partial_{x'}v)^2& \textup{(PDE)}, \\
         \partial_{s's'} V + \partial_{s'} V + 2y' \partial_{s'y'} V + 2y' \partial_{y'} V + ({y'}^2 - 1)\partial_{y'y'} V \\
         = ( \partial_{s'} V+ y' \partial_{y'} V)^2 - (\partial_{y'}V)^2  & \textup{(PDE in self-similar variables)}, \\
        V_{1,\beta,\kappa} := s'  - \log \frac{1}{\sqrt{1 - \gamma ^2}} \left( 1 + \beta - \gamma(1 - \beta) \right)\\
          \qquad\qquad- \log \frac{1}{\sqrt{1 - \gamma ^2}} \left( (1 - \beta - \gamma(1 + \beta))y' \right) + \kappa  & \textup{(Self-similar blow-up family)}, \\
     \end{cases}  
       \end{split}
   \end{align} 
   where we have used 
  \begin{align*}
      &\frac{1}{1 - \gamma ^2} (\partial_{t'}v - \gamma \partial_x' v)^2  - \frac{1}{1 - \gamma ^2}  ( \partial_{x'} v - \gamma \partial_{t'} v)^2 \\
      &   = \frac{1}{1 - \gamma ^2} ( \partial_{t'} v  - \gamma \partial_{x'}v  + \partial_{x'}v  - \gamma \partial_{t'}v)( \partial_{t'} v - \gamma \partial_{x'} v - \partial_{x'}v  +\gamma \partial_{t'}v ) \\
      &= \frac{1}{1 - \gamma ^2} ( (1 - \gamma) (\partial_{t'}v + \partial_{x'}v))((1 + \gamma)(\partial_{t'}v - \partial_{x'}v))  \\
      &= (\partial_{t'}v + \partial_{x'} v)(\partial_{t'} v - \partial_{x'}v) = (\partial_{t'}v)^2 - (\partial_{x'}v)^2. 
  \end{align*} 
  Thus to make \(V_{1,\beta,\kappa}\) independent of the spatial variable we fix \( \gamma := \frac{1 - \beta}{1 + \beta}\) which is in \( (-1,1)\) iff \( \beta > 0 \). Then we have the following transformations   
  \begin{align}\label{eq:4.Lorentz Transformation Table at γ=γ(β)} 
      \begin{split} 
          &\begin{cases}  U_{1,\beta,\kappa} \\ 
          \eqref{eq:4.Linearised equation at Uαβκ} \\
        \mathbf{L}^2_{1,\beta}   \end{cases}  
         \textup{gets transformed into}  \\
         &\begin{cases} V_{1,\beta,\kappa} := s' - \log  2 \sqrt{ \beta} + \kappa & \textup{(Self-similar blow-up family)}, \\
          \partial_{s's'} \xi - \partial_{s'} \xi + 2y' \partial_{y's'} \xi + ({y'}^2 - 1)\partial_{y'y'}\xi = 0 & \textup{(Linearisation at } V_{1,\beta,\kappa}) , \\
        \mathbf{L}''_{1}\begin{pmatrix}r_1 \\r_2 \\ \end{pmatrix} := \begin{pmatrix} r_2 - y' \partial_{y'} r_1 \\ r_2 - y' \partial_{y'} r_2  + \partial_{y'y'} r_1 \\ \end{pmatrix} & \textup{(Linearised operator).} \end{cases}  
      \end{split}
  \end{align} 
\end{proposition}
\begin{proof} 
  Follows from direct calculation as done in the previous section. 
\end{proof}  
\begin{remark}\label{rmk:4.Reduction to previous case}
  Recall that we had shown 
   \[ \mathbf{L}'_\alpha\begin{pmatrix}r_1 \\r_2 \\ \end{pmatrix} := \begin{pmatrix} r_2 - y' \partial_{y'} r_1 \\ r_2 - y' \partial_{y'} r_2  + \partial_{y'y'} r_1  - 2 \sqrt{ 1 - \alpha} \partial_{y'} r_1\\ \end{pmatrix} \]
  is mode stable for all \(0<\alpha\le 1.\) Note that
  \[  \mathbf{L}''_{1}  =\mathbf{L}'_{1} 
  \]  
  and hence the mode-stability of \(\mathbf{L}''_{1}\) is reduced to showing the mode-stability of \(\mathbf{L}'_{1}\). 
\end{remark}
In light of Remark \ref{rmk:4.Special case of L²11κ} and Remark \ref{rmk:4.Reduction to previous case} and Proposition \ref{prop:3.Mode stability of Uα∞κ}, we obtain 
\begin{proposition}\label{prop:4.Mode stability of U1βκ}
   For \(\beta > 0\), the operator $\mathbf{L}^2_{1,\beta}$ admits no eigenvalues with $\mathrm{Re}\,\lambda > -1$ except for $\lambda = 0$ and $\lambda = 1$. As a result, \(U_{1,\beta,\kappa}\) is mode-stable for all \(\beta > 0, \kappa \in \mathbb{R}\).
\end{proposition}

\section{Spectral Analysis of \texorpdfstring{\(\mathbf{L}_\alpha\)}{Lα}}\label{sec:Spectral Analysis}
In this section, we study the essential spectrum of \(\mathbf{L}_\alpha := \mathbf{L}^1_{\alpha,\infty} \). Here we aim to write \( \mathbf{L}_\alpha\) as a sum of a maximally dissipative operator and a compact perturbation. To that end, we present the functional setup. 

\subsection{Functional Setup}
Per \eqref{eq:3.Linearisation around Uαβκ PDE}, the linearised equation can be written as 
\begin{equation}\label{eq:5.Linearised operator Lα}
  \partial_s \mathbf{q} = \mathbf{\tilde{L}}_\alpha \mathbf{q} := \begin{pmatrix} q_2 - y \partial_y q_1 \\ -q_{2} - y \partial_y q_2 + \partial_{yy} q_1 +\frac{2\alpha}{1+  \sqrt{ 1- \alpha} y  } q_2  \end{pmatrix}
\end{equation} 
We now further decompose the operator into the \textit{free wave operator} and the \textit{potential} respectively: 
\[ \mathbf{\tilde{L}}_\alpha \mathbf{q} := \mathbf{\tilde{L}} _{\square} \mathbf{q} + \mathbf{V}_\alpha \mathbf{q} := \begin{pmatrix} q_2 - y \partial_y q_1  - q_1(-1) \\ -q_2 - y \partial_y q_2  + \partial_{yy} q_1 \end{pmatrix} + \begin{pmatrix} q_1(-1) \\   \frac{2\alpha }{1 + \sqrt{1 - \alpha}y} q_2 \end{pmatrix}. 
\] 
We recursively define the following sesquilinear forms on \( C^{k+1}[-1,1] \times  C^k[-1,1] :\)
\begin{align*}
     \langle  \mathbf{f},\mathbf{g} \rangle _0 &:= f_1(-1) \overline{g_1(-1)} + \langle  f_1, g_1 \rangle_{\dot{H}^1(-1,1)} + \langle  f_2, g_2 \rangle_{L^2(-1,1)}, \\
     \langle  \mathbf{f}, \mathbf{g} \rangle_k &:= \langle \mathbf{f},\mathbf{g} \rangle_0 + \langle  f_1, g_1 \rangle_{\dot{H}^{k+1}(-1,1)} + \langle  f_2, g_2 \rangle_{\dot{H}^k(-1,1)},
\end{align*} 
and similarly, we define the induced norms on \( C^{k+1}[-1,1] \times C^k[-1,1]\) by 
\[  \lVert  \mathbf{f} \rVert_{k}^2 := \langle  \mathbf{f},\mathbf{f} \rangle_k.  
\] 
In fact, the \(\lVert\,\cdot\, \rVert_{k}\) norm is equivalent to the norm on the Sobolev space \( H^{k+1} \times  H^k\). 
\begin{lemma}\label{lem:5.k-norm = Hk}
  For \( k  \in \mathbb{N}\) and \( \mathbf{f} \in C^{k+1} [ -1,1] \times  C^k[-1,1]\), we have 
  \[  \lVert  \mathbf{f} \rVert_{k} \approx \lVert  f_1 \rVert_{H^{k+1}}  + \lVert  f_2 \rVert_{H^k}.
  \] 
\end{lemma}
\begin{proof} 
  As per the definition of \(\lVert\,\cdot\, \rVert_{k}\) and \(\langle\,,\, \rangle_k\), it suffices to show 
  \[  \lVert  \mathbf{f} \rVert_{0}^2 \geq \langle  f_1, f_1 \rangle_{L^2} , \quad \forall  \mathbf{f} \in C^1[-1,1]\times C^0[-1,1]. 
  \] 
  Indeed, we have 
  \begin{align*}
     \lVert  f_1 \rVert_{L^2} &=\left\lVert  f_1(-1) + \int_{-1}^y \partial_y f_1 \right\rVert_{L^2} \lesssim   \lvert f_1(-1) \rvert + \int_{-1}^y \lVert  \partial_y f_1 \rVert_{L^2} \lesssim   \lvert f_1(-1) \rvert + \lVert  f_1 \rVert_{\dot{H}^1} \leq \lVert  \mathbf{f} \rVert_{0}. \qedhere 
  \end{align*} 
\end{proof}  
\subsection{Generation of the semigroup.}
For \(\alpha > 0\), fix \( \mathcal{D}(\mathbf{\tilde{L}}_\alpha) = \mathcal{D}(\mathbf{\tilde{L}}) = (C^\infty[-1,1])^2\) and denote \(\mathcal{H}^k := H^{k+1} (-1,1) \times  H^k(-1,1)\). Then the free wave operator \( \mathbf{\tilde{L}}_\square : \mathcal{D}(\mathbf{\tilde{L}}_\square) \subset \mathcal{H}^k \to \mathcal{H}^k \) is well-defined for \( k \in \mathbb{N}\). 

\begin{lemma}[Dissipativity of \(\mathbf{\tilde{L}}_\square\)]\label{lem:5.Dissipativity of Wave}
  For \( k \in \mathbb{N}\) and \(\mathbf{f} \in C^{k+1} [-1,1] \times C^k[-1,1] : \) 
  \[  \mathbb{R}{\rm e\,} \langle  \mathbf{\tilde{L}}_\square\mathbf{f},\mathbf{f} \rangle _k \leq -\frac{1}{2} \lVert  \mathbf{f} \rVert_{k}^2. 
  \] 
\end{lemma}
\begin{proof} 
  For \(k = 0\), we have 
  \begin{align*}
      &\mathbb{R}{\rm e\,} \langle  \mathbf{\tilde{L}}_\square \mathbf{f},\mathbf{f} \rangle_0 \\
      &= \mathbb{R}{\rm e\,} \left((f_2 - y \partial_y f_1 - f_1(-1))\bar{f}_1\right)|_{y=-1}  + \mathbb{R}{\rm e\,}\langle  f_2 - y \partial_y f_1  , f_1 \rangle_{\dot{H}^1} + \mathbb{R}{\rm e\,}\langle   \partial_{yy} f_1  -f_2 - y \partial_y f_2, f_2 \rangle_{L^2}\\ 
      & =  -  \lvert f_1(-1) \rvert^2  + \frac{1}{2}  \lvert  (f_2 - y \partial_y f_1)(-1) \rvert^2 + \frac{1}{2}  \lvert  f_1(-1) \rvert^2  + \mathbb{R}{\rm e\,} \langle  f_2, f_1 \rangle_{\dot{H}^1} - \lVert  \partial_y f_1 \rVert_{L^2}^2 \\
      &\qquad\qquad - \frac{1}{2} \mathbb{R}{\rm e\,} \langle  y, \partial_y  \lvert  \partial_y f_1 \rvert^2 \rangle  + \mathbb{R}{\rm e\,} \langle  \partial_{yy} f_1, f_2   \rangle - \lVert  f_2 \rVert_{L^2}^2 - \frac{1}{2} \mathbb{R}{\rm e\,} \langle  y, \partial_y  \lvert f_2 \rvert^2 \rangle \\
      &= - \frac{1}{2}  \lvert f_1(-1) \rvert^2  + \frac{1}{2}  \lvert  (f_2 + \partial_y f_1)(-1) \rvert^2  - \lVert  \partial_y f_1 \rVert_{L^2}^2 - \frac{1}{2} (  \lvert  \partial_y f_1(\pm 1) \rvert^2 +  \lvert  f_2(\pm 1) \rvert^2   )  \\
      &\qquad\qquad + \frac{1}{2} \lVert  \partial_y f_1 \rVert_{L^2}^2  + \mathbb{R}{\rm e\,} (\partial_y f_1 f_2)|_{-1}^1- \lVert  f_2 \rVert_{L^2}^2    + \frac{1}{2} \lVert  f_2 \rVert_{L^2}^2 \\
      &\leq -\frac{1}{2}  \lvert f_1(-1) \rvert^2 - \frac{1}{2} \lVert  \partial_y f_1 \rVert_{L^2}^2  - \frac{1}{2} \lVert  f_2 \rVert_{L^2}^2 = - \frac{1}{2} \lVert  \mathbf{f} \rVert_{0}^2. 
  \end{align*}  
  For \(k \geq 1\), we have 
  \begin{align*}
      & \mathbb{R}{\rm e\,} \langle  \mathbf{\tilde{L}}_\square \mathbf{f}, \mathbf{f} \rangle_k  = \mathbb{R}{\rm e\,} \langle  f_2 - y \partial_y f_1, f_1 \rangle_{\dot{H}^{k+1}} + \mathbb{R}{\rm e\,} \langle  \partial_{yy} f_{1} - f_2 - y \partial_y f_2, f_2 \rangle_{\dot{H}^k} + \mathbb{R}{\rm e\,} \langle  \mathbf{\tilde{L}}_\square \mathbf{f}, \mathbf{f} \rangle_0 \\
      & = \mathbb{R}{\rm e\,} \langle  f_2, f_1 \rangle_{\dot{H}^{k + 1}}  - (k + 1) \lVert  \partial^{k + 1} f_1 \rVert_{L^2}^2 - \frac{1}{2} \mathbb{R}{\rm e\,} \langle  y, \partial_y  \lvert  \partial^{k + 1} f_1 \rvert^2 \rangle  + \mathbb{R}{\rm e\,} \langle  \partial_{yy} f_1 , f_2 \rangle_{\dot{H}^k} \\
      &\qquad\qquad - \lVert  \partial^k f_2 \rVert_{L^2}^2 - k \lVert  \partial^k f_2 \rVert_{L^2}^2  - \frac{1}{2} \mathbb{R}{\rm e\,} \langle  y, \partial_y  \lvert  \partial^k f_2 \rvert^2 \rangle + \mathbb{R}{\rm e\,} \langle  \mathbf{\tilde{L}}_\square \mathbf{f},\mathbf{f} \rangle_0 \\
      &= \mathbb{R}{\rm e\,} ( \partial^{k + 1} f_1 \partial^k \bar{f}_2)|_{-1}^1 - (k + 1  - \frac{1}{2})( \lVert  \partial^{k+ 1} f_1 \rVert_{L^2}^2 + \lVert  \partial^k f_2 \rVert_{L^2}^2) \\
      &\qquad\qquad- \frac{1}{2}(  \lvert  \partial^{ k + 1} f_{1}(\pm 1) \rvert^2 +  \lvert  \partial^k f_2(\pm 1) \rvert^2) + \mathbb{R}{\rm e\,} \langle  \mathbf{\tilde{L}}_\square \mathbf{f},\mathbf{f} \rangle_0  \leq -\frac{1}{2} \lVert  \mathbf{f} \rVert_{k}^2. \qedhere
  \end{align*} 
\end{proof}  
We will now show that \( \mathbf{\tilde{L}_\square}\) has dense range in \(\mathcal{H}^k\) for all \( k \in \mathbb{N}\). 
\begin{lemma}[Dense range of \(\mathbf{\tilde{L}}_\square\)]\label{lem:5.Dense range of Wave}
  For \( k \in \mathbb{N}\), \(\mathbf{g} \in \mathcal{H}^k\) and \(\varepsilon > 0\), there exists an \(\mathbf{f} \in \mathcal{H}^k : \) 
  \[  \lVert  \mathbf{\tilde{L}}_\square \mathbf{f} - \mathbf{g} \rVert_{\mathcal{H}^k} \leq \varepsilon.  
  \] 
\end{lemma}
\begin{proof} 
  Since smooth functions are dense in \(\mathcal{H}^k\), it suffices to assume \(\mathbf{g} \in (C^\infty[-1,1])^2\). We now show that there exists a \(\mathbf{f} \in (C^\infty[-1,1])^2 \) such that 
  \[  \mathbf{\tilde{L}}_\square\mathbf{f} = \mathbf{g}.  
  \] 
  In particular, we show 
  \[  \begin{cases}  f_2 - y \partial_y f_1 - f_1(-1)  = g_1 , \\ 
  \partial_{yy} f_1 - f_2 - y \partial_y f_2 = g_2.  \end{cases}   
  \] 
  (Here by \(g_1\) and \(g_2\), we mean a representative from the equivalence class denoted by \( g_1\) and \(g_2\). Moreover, establishing the pointwise equality is allowed as we are only interested in the \(L^p\)-norms). Plugging the first equation in the second, we obtain the following second-order ODE: 
  \[ (1 - y^2) \partial_{yy} f_1 -2 y \partial_y f_1      =g_1 - y \partial_y g_1 + g_2 + f_1(-1)
  \] 
  This can be written as 
  \[  \partial_y \left( (1 - y^2) \partial_y  f_1\right)  =g_1 - y \partial_y g_1 + g_2 + f_1(-1) = : F(y) + f_1(-1)
  \] 
  and thus the solution is 
  \[  f_1(y) = f_1(-1)  + \int_{-1}^y \frac{\int_{-1}^{y'} F(y') + f_1(-1) }{1 - {y'}^2} \mathrm{\,d} y' . 
  \] 
  It remains to study the smoothness properties of \(f_1\) and \(f_2\). Since \( f_2 = y \partial_y f_1 + f_1(-1) + g_1\) where \(g_1\) is smooth so it suffices to show that \( f_1 \) is smooth. As \(g_1,g_2\) are smooth so then is \(F\) and thus the above formula suggests that \( f\) is smooth atleast on \((-1,1)\). At \( y = -1\), we can express the integrand as 
  \[  \frac{1}{1 - y'} \frac{1}{1 + y'} \int_{-1}^{y'} (F(y') + f_1(-1)   )
  \] 
  which is smooth at \(-1\) as it is a product of functions that are smooth near \(-1\) (recall that \( \frac{1}{x} \int_{0}^x \phi \in C^k[0,1]\) if \(\phi \in C^{k-1}[0,1]\).) For \(y = 1\), we express the integrand similarly: 
  \[  \frac{1}{1 + y'} \frac{1}{1 - y'}  \int_{-1}^{y'}\left(F + f_1(-1)\right) = \frac{1}{1 + y'} \frac{1}{1 - y'} \left( \int_{-1}^1 (F + f_1(-1))- \int_{y'}^1 (F + f_1(-1)) \right) .
  \] 
  This is smooth iff \( \int_{-1} ^1 \left(F + f_1(-1)\right) = 0\). Thus fixing 
  \[ f_1(-1) = \frac{1}{2} \int_{-1}^1 F \]
  implies that \( f_1 \in C^\infty[-1,1]\) and thus \(\mathbf{f} \in (C^\infty[-1,1])^2\). 
\end{proof}  

\begin{proposition}\label{prop:5.Semigroup of Wave}
  For \( \in \mathbb{N}\), the operator \( \mathbf{\tilde{L}}_\square : D(\mathbf{\tilde{L}}_\square) \subset \mathcal{H}^k \to \mathcal{H}^k\) is a densely-defined closable operator whose closure \(\mathbf{L}_\square\) generates a strongly continuous one-parameter semigroup \( \mathbf{S} : \mathbb{R}_+ \to \mathcal{L}(\mathcal{H}^k)\) which satisfies 
  \[  \lVert  \mathbf{S}(s) \rVert_{\mathcal{L}(\mathcal{H}^k)}  \lesssim  e^{ -\frac{1}{2}s} , \quad \forall  s \in \mathbb{R}_+. 
  \] 
\end{proposition}
\begin{proof} 
  Since \( \mathcal{D}(\mathbf{\tilde{L}}_\square) = (C^\infty[-1,1])^2\) so it is densely-defined in \( \mathcal{H}^k\). Moreover, the previous lemmas show that it is a dissipative operator with dense range for all \( k \in \mathbb{N}\). Therefore, the conclusion follows from the Lumer-Phillips Theorem. 
\end{proof}  
Next we deal with the potential operator \( \mathbf{V}_{\alpha}\). It turns out that this potential operator is not a compact operator, not even relatively compact with respect to \(\mathbf{L}_\square\). Nevertheless, we can prove that 
\begin{proposition}\label{prop:5.Semigroup Sα}
  For all \(\alpha \in (0,1], k \in \mathbb{N}\) the operator \( \mathbf{V}_\alpha\) is bounded on \( \mathcal{H}^k\). As such, the operator \(\mathbf{L}_\alpha = \mathbf{L}_\square + \mathbf{V}_\alpha\) generates a strongly continuous one-parameter semigroup \( \mathbf{S}_\alpha : \mathbb{R}_+ \to \mathcal{L}(\mathcal{H}^k)\) which satisfies 
  \[  \lVert  \mathbf{S}_\alpha(s) \rVert_{\mathcal{L}(\mathcal{H}^k)} \lesssim  e^{ \left(\lVert  \mathbf{V}_{\alpha} \rVert_{\mathcal{L}(\mathcal{H}^k)} - \frac{1}{2}\right)s}, \quad \forall  s \in \mathbb{R}_+. 
  \] 
\end{proposition}
\begin{proof} 
  Note that \(\mathbf{V}_\alpha\) is bounded on \(\mathcal{H}^k\) since 
  \[  \lVert  \mathbf{V}_\alpha \mathbf{f}\rVert_{\mathcal{H}^k}   = {\lVert  f_1(-1) \rVert_{H^{k+1}} } +  \left\lVert  \frac{2\alpha}{1+ \sqrt{1-\alpha}y} f_2 \right\rVert_{H^k} \lesssim  _\alpha \lVert  \mathbf{f} \rVert_{\mathcal{H}^k}. 
  \] 
  As such, by the bounded perturbation theorem, \(\mathbf{L}_\alpha\) generates a strongly continuous one-parameter semigroup. 
\end{proof}  
\subsection{A new decomposition with a compact perturbation}
Since \( \mathbf{V}_\alpha\) is merely a bounded perturbation without any compactness compared to the free wave operator, we can not obtain any information for the essential spectrum of \(\mathbf{L}_\alpha\) and hence no uniform growth for the semigroup \( \mathbf{S}_\alpha\). To overcome this difficulty, we introduce a new decomposition of the linearised operator \(\mathbf{L}_\alpha\) as a sum of a new dissipative operator plus a compact perturbation. More precisely, we recall that 
\[  \mathbf{L}_\alpha = \begin{pmatrix}  - y \partial_y  & 1 \\ \partial_{yy}  &  \frac{2 \alpha }{1 + \sqrt{1-\alpha} y} - 1 - y \partial_y  \end{pmatrix}  . 
\] 
Here \( \mathbf{L}_\alpha\) is not necessarily dissipative: 
\begin{align*}
    \mathbb{R}{\rm e\,} \langle  \mathbf{L}_\alpha \mathbf{f} ,& \mathbf{f} \rangle_{\dot{H}^1 \times  L^2}  \leq \mathbb{R}{\rm e\,}\langle  f_2 - y \partial_y f_1 , f_1 \rangle_{\dot{H}^1} +\mathbb{R}{\rm e\,} \left\langle \partial_{yy} f_1  - f_2 - y \partial_y f_2  +  \frac{2\alpha    }{1+ \sqrt{1-\alpha} y }  f_2, f_2  \right\rangle_{L^2} \\
    &\leq \mathbb{R}{\rm e\,} \langle  f_2, f_1 \rangle_{\dot{H}^1}- \mathbb{R}{\rm e\,}\langle  \partial_y f_1 + y \partial_{yy} f_1, \partial_y  f_1 \rangle_{L^2} + \mathbb{R}{\rm e\,} \langle  \partial_{yy} f_1, f_2 \rangle_{L^2}  \\
    &\qquad\qquad -\lVert  f_2 \rVert_{L^2}^2 - \frac{1}{2} \mathbb{R}{\rm e\,}\langle  y , \partial_y  \lvert f_2 \rvert^2 \rangle_{L^2}  + \mathbb{R}{\rm e\,} \left\langle  \frac{2\alpha    }{1+ \sqrt{1-\alpha} y }  f_2, f_2\right\rangle_{L^2} \\
    &\leq \mathbb{R}{\rm e\,} \langle  f_2, f_1 \rangle_{\dot{H}^1} - \lVert  f_1 \rVert_{\dot{H}^1}^2 -\frac{1}{2} \mathbb{R}{\rm e\,} \langle  y , \partial_y  (\lvert \partial_y f_1 \rvert^2  +  \lvert f_2 \rvert^2)\rangle_{L^2} + \mathbb{R}{\rm e\,} (\partial_y f_1 f_2)_{-1}^1 - \mathbb{R}{\rm e\,} \langle  f_1, f_2 \rangle_{\dot{H}^1} \\
    &\qquad\qquad - \lVert  f_2 \rVert_{L^2}^2    + \mathbb{R}{\rm e\,} \left\langle  \frac{2\alpha    }{1+ \sqrt{1-\alpha} y }  f_2, f_2\right\rangle_{L^2} \\ 
    &\leq - \lVert  f_1 \rVert_{\dot{H}^1}^2 - \frac{1}{2}  (  (\lvert  \partial_y f_1 \rvert^2  +  \lvert f_2 \rvert^2)(\pm 1) - \lVert  f_1 \rVert_{\dot{H}^1}^2 - \lVert  f_2 \rVert_{L^2}^2) + \mathbb{R}{\rm e\,} ( \partial_y f_1 f_2)_{-1}^1  \\
    &\qquad\qquad  - \lVert  f_2 \rVert_{L^2}^2  + \mathbb{R}{\rm e\,} \left\langle   \frac{2\alpha    }{1+ \sqrt{1-\alpha} y } f_2, f_2 \right\rangle_{L^2} \\
    &\leq - \frac{1}{2} \lVert  f_1 \rVert_{\dot{H}^1}^2  -\frac{1}{2}\lVert  f_2 \rVert_{L^2}^2  + \mathbb{R}{\rm e\,} \left\langle   \frac{2\alpha    }{1+ \sqrt{1-\alpha} y }  f_2, f_2 \right\rangle_{L^2} \\ &\leq -\frac{1}{2} \lVert  f_1 \rVert_{\dot{H}^1}^2 + \left(  - \frac{1}{2} +  \frac{2\alpha    }{1+ \sqrt{1-\alpha} y } \right) \lVert  f_2 \rVert_{L^2}^2. 
\end{align*} 
To gain more dissipation, we need to consider \( \mathbf{L}_\alpha\) in a Sobolev space with higher regularity, i.e. \(\mathcal{H}^k := H^{k + 1} \times  H^k\) for large enough \(k\). 

\begin{lemma}[Commutator with derivatives]\label{lem:5.Commutator with derivatives}
  For \(\alpha \in (0,1]\) and \(k \in \mathbb{N}\), we have 
  \[   \partial^k \mathbf{L}_\alpha = \mathbf{L}_{\alpha,k} \partial^k + \mathbf{L}'_{\alpha,k}
  \] 
  where 
  \[  \mathbf{L} _{\alpha,k} = \begin{pmatrix} -k - y \partial_y &1  \\ \partial_{yy}&  \frac{2 \alpha }{1 + \sqrt{1-\alpha} y} -1-k -y \partial_y  \end{pmatrix} 
  \] 
  and \( \mathbf{L}'_{\alpha,k}\) satisfies the pointwise bound 
  \[   \lvert  \mathbf{L}'_{\alpha,k} \mathbf{f} \rvert  \lesssim _\alpha \begin{pmatrix} 0 \\ \sum_{j=0}^{k-1}  \lvert  \partial^j f_1 \rvert  \end{pmatrix}  . 
  \] 
\end{lemma}
\begin{proof} 
  Direct computation yields the following formulas: 
  \[  [ \partial^k , y \partial_y] = k \partial^k , \quad \left[ \partial^k , \frac{2\alpha }{1 + \sqrt{1-\alpha} y}\right]  =\frac{2\alpha }{1 + \sqrt{1-\alpha} y} \partial^k +  \sum_{j =1}^{k} C^{j}_k \partial^{j} \frac{2\alpha }{1 + \sqrt{1-\alpha} y} \partial^{k-j} 
  \] 
  and the pointwise bound follows from 
 \[  \left\lVert  \partial^j \frac{2\alpha }{1 + \sqrt{1-\alpha} y} \right\rVert_{L^\infty(-1,1)}  \lesssim  _{\alpha ,j}\frac{\alpha (1 - \alpha)^{\frac{j-1}{2}}}{(1 - \sqrt{1 - \alpha})^j}. \qedhere 
 \] 
\end{proof}  
For \( k\geq 0\), we introduce the inner product 
\[  \langle  f,g \rangle_{H^k} : = \sum_{j=0} ^k \langle  f, g \rangle_{\dot{H}^j}. 
\] 
This is equivalent to the following representation which we will fixate from now onwards:
\[  \langle  f,g \rangle _{H^k} := \langle  f, g \rangle_{\dot{H}^k} + \langle  f,g \rangle_{L^2}. 
\] 
For tuples of functions \( \mathbf{f} = (f_1, f_2)\), we introduce the corresponding inner product on \( \mathcal{H}^k\) as 
\[  \langle  \mathbf{f}, \mathbf{g} \rangle_{\mathcal{H}^k} := \langle  f_1, g_1 \rangle_{H^{k + 1}}  + \langle  f_2, g_2 \rangle_{H^k}. 
\] 
To control \( \mathbf{L}'_{\alpha,k}\), we prove the following subcoercivity estimate: 
\begin{lemma}\label{lem:5.Subcoercity}
  For \( k \in \mathbb{N}\), there exists positive \(\varepsilon _n \to 0\) and positive \(c_n\) and subspaces \( \{\Pi _i\}_{i \leq n} \subset H^{k+1}(-1,1) :\) 
  \begin{equation}\label{eq:5.Subcoercity estimate for Hk}
    \varepsilon _n \langle f ,f \rangle_{H^{k+1}} \geq \lVert  f \rVert_{H^{k}}^2 - c_n \sum_{i = 1}^n  \lvert \langle  f, \Pi _i \rangle_{L^2}  \rvert^2, \quad \forall  f \in H^{k+1}(-1,1) , \forall  n \in \mathbb{Z}_+.
  \end{equation} 
\end{lemma}
\begin{proof} 
  We first construct the subspaces \( \{\Pi _{i}\}_{i \leq n}\). To that end, consider an element \(T \in L^2(-1,1)\). This defines a anti-linear functional on \( (H^{k+1} , \langle\,,\, \rangle_{H^{k+1}})\) via the mapping \( h \mapsto \langle  T, h \rangle_{L^2}\). As such, via Riesz representation theorem, we can find a unique element \( l(T) \in H^{k+1}\) satisfying 
  \[ \langle  l(T), h \rangle_{H^{k+1}} = \langle  T, h \rangle_{L^2}, \quad \forall  h \in H^{k+1}(-1,1).  
  \] 
  In particular, we have constructed a mapping \( l : L^2(-1,1) \to H^{k+1}(-1,1)\). If we compose it with the imbedding \( \iota : H^{k+1}(-1,1) \to L^2(-1,1)\), then \( \iota \circ l \) is compact (via Sobolev Embedding theorem) and self-adjoint as 
  \begin{align*}
      \langle T_1, \iota \circ l(T_2) \rangle _{L^2} &= \langle  T_1, l(T_2) \rangle_{L^2} = \langle l(T_1), l(T_2)  \rangle_{H^{k+1}} = \overline{\langle  l(T_2), l(T_1) \rangle_{H^{k+1}}} \\
      &= \overline{\langle T_2, l(T_1) \rangle_{L^2}} = \langle  \iota \circ l(T_1) , T_2 \rangle_{L^2}. 
  \end{align*} 
  We also find that it is positive-definite since 
  \[  \langle i \circ l(T_1), T_1 \rangle_{L^2} = \langle  l(T_1),l(T_1) \rangle _{H^{k+1}} \geq 0 . 
  \] 
  Consequently, \( \iota \circ l \in \mathcal{L}(L^2(-1,1))\) is a positive-definite compact operator and so it admits a spectral decomposition characterised by the positive eigenvalues \( \lambda _n\) and an orthonormal basis \( \{\Pi _{n,i}\}_{i \leq I(n)}\) such that \(0 <  \lambda _n \to 0.\) With this, we have constructed the relevant collection of subspaces and we set \( \{\Pi _i\}_{i \leq n} := \{\Pi _{m,j}\}_{j \leq I(m), m\le n}. \)

  We now turn our attention towards establishing the inequality. To that end, first consider the minimisation problem 
  \begin{align*}
      I_n :&= \inf_{\Psi \in \mathcal{A}_n}\, \langle  \Psi,\Psi \rangle_{H^{k+1}} ,\\
       \mathcal{A}_n :&= \left\{ \Psi \in H^{k+1}(-1,1) : \lVert  \Psi \rVert_{L^2} = 1 \text{ and } \forall j \in [1,n] . \forall  i \in [1,I(j)] :  \langle  \Psi , \Pi _{j,i} \rangle_{L^2} = 0\right\}. 
  \end{align*} 
  The infimum is attained on some \( \Psi ^\star \in \mathcal{A}_n\) as this problem can be seen as the projection of \(0\) onto the closed convex set \( \mathcal{A}_n\) in \( (H^{k+1}, \langle\,,\, \rangle_{H^{k+1}})\) which always exist as per Hilbert space theory. Moreover, writing the Lagrangian of this minimisation problem and considering the first variation along any direction \( h \in H^{k+1}(-1,1)\), we have 
  \[  \langle  \Psi ^\star , h \rangle_{H^{k+1}} = \sum_{\substack{ j \leq n \\ i \leq I(j)}}   \beta _{j,i} \langle   \Pi _{j,i} ,h\rangle_{L^2} + \beta \langle  \Psi ^\star, h \rangle_{L^2}, \quad \forall  h \in H^{k+1}(-1,1). 
  \] 
  Fixing \( h = \Pi _{j,i}\) gives us \(\forall  j \leq n , \forall  i \leq I(j). \)
  \[ 0 =\lambda _n ^{-1} \langle  \Psi ^\star , \Pi _{j,i} \rangle_{L^2} =  \lambda _n ^{-1} \langle  \Psi ^\star , l(\Pi _{j,i}) \rangle_{H^{k+1}} =  \langle \Psi ^\star , \Pi _{j,i} \rangle _{H^{k+1}} = \beta _{j,i} + \beta \langle  \Psi ^\star , \Pi _{j,i} \rangle _{L^2} = \beta _{j,i}. 
  \] 
  Thus we have 
  \[  \langle  \Psi ^\star , h \rangle_{H^{k+1}} = \beta \langle  \Psi ^\star , h \rangle_{L^2}  = \beta\langle  l(\Psi ^\star ), h \rangle_{H^{k+1}}  , \quad \forall  h \in H^{k+1}(-1,1) . 
  \] 
  Thus \((\beta,\Psi ^\star )\) is an eigenpair for \( ( l, (H^{k+1}, \langle\,,\, \rangle_{H^{k+1}}))\) and thus even for \( (l, L^2)\). Since \( \Psi ^\star \in \mathcal{A}_n\) so 
  \[  \beta ^{-1} = \langle  \Psi ^\star, l(\Psi ^\star)  \rangle _{L^2} = \sum_{\substack{j \geq n + 1 \\ i \leq I(j) }} \lambda _{j }  \lvert  \langle  \Psi ^\star , \Pi _{j,i} \rangle_{L^2} \rvert^2 \leq \lambda _{n + 1} \lVert  \Psi ^\star \rVert_{L^2}^2 = \lambda _{n + 1}. 
  \]
  As such, we obtain 
  \[  I_n = \langle  \Psi ^\star , \Psi ^\star \rangle_{H^{k+1}} = \beta \langle  \Psi ^\star , \Psi ^\star \rangle_{L^2} = \beta \geq \lambda _{n + 1} ^{-1}. 
  \]  
  In other words, 
  \begin{equation}\label{eq:5.Control on L2 via Hk}
    \langle  \Psi , \Psi \rangle_{L^2} \leq \lambda _{n +1}\langle  \Psi, \Psi \rangle_{H^{k+1}}, \quad \forall  \Psi \in \mathcal{A}_n. 
  \end{equation}
  Finally, fix \(\varepsilon > 0 , k \in \mathbb{N}\) and for any \( f \in \mathcal{A}_n\), by Gagliardo-Nirenberg and Young's inequality and \eqref{eq:5.Control on L2 via Hk}
  \begin{align*}
      \lVert  f \rVert_{H^{k}}^2 \leq \varepsilon \lVert  f \rVert_{\dot{H}^{k+1}}^2 + c_{\varepsilon, k} \lVert  f \rVert_{L^2}^2 \leq \varepsilon \langle f,f \rangle_{H^{k+1}}   + c_{\varepsilon,k} \lambda _{n + 1} \langle  f,f \rangle_{H^{k+1}} \leq (\varepsilon + c_{\varepsilon, k} \lambda _{n + 1}) \langle f,f \rangle_{H^{k+1}} . 
  \end{align*} 
  Letting \( n\) large enough so that \( c_{\varepsilon,k} \lambda _{n + 1} \leq \varepsilon\), we obtain 
  \[\lVert  f \rVert_{H^k}^2 \leq 2 \varepsilon \langle  f,f \rangle_{H^{k+1}}, \quad \forall  f \in H^{k+1} : \langle  f, \Pi _i \rangle_{L^2} = 0 , \forall  i \leq n.\]
    For a general \(f\), we apply the above inequality with \( f - \sum_{i \leq n} \langle  f, \Pi _i \rangle_{L^2} \Pi _i \). 
\end{proof}  
We will now be able to show the dissipativity of \(\mathbf{L}_\alpha \) up to a compact perturbation. 

\begin{proposition}[Maximal Dissipativity of \(\mathbf{\hat{L}}_\alpha \)]\label{prop:5.Maximal Dissipativity of L̂α}
  For \(\alpha \in (0,1], k \geq 5\) and \( 0 < \varepsilon \leq \varepsilon _0\) sufficiently small, there exists \( (\mathbf{\Pi} _{\alpha,i})_{i \leq N} \in \mathcal{H}^k \) such that for the finite rank projection operator 
  \begin{equation}\label{eq:5.Finite Projection Operator P̂α}
      \mathbf{\hat{P}}_\alpha = \sum_{ i \leq N}  \langle  - , \mathbf{\Pi}_{\alpha,i} \rangle _k  \mathbf{\Pi}_{\alpha,i}
  \end{equation} 
  the modified operator \( \mathbf{\hat{L}}_\alpha := \mathbf{L}_\alpha - \mathbf{\hat{P}}_\alpha\) is dissipative: 
  \begin{equation}\label{eq:5.Dissipative Inequality P̂α}
    \forall \mathbf{q} \in \mathcal{D}(\mathbf{L}_\alpha ) : \mathbb{R}{\rm e\,} \langle  - \mathbf{\hat{L}}_\alpha \mathbf{q},\mathbf{q} \rangle_{\mathcal{H}^k} \geq (k-\frac{3}{2} - 2 \sqrt{ 1-\alpha} - \varepsilon) \lVert  \mathbf{q} \rVert_{\mathcal{H}^k}^2 \geq ( \frac{3}{2}-\varepsilon) \lVert \mathbf{q} \rVert_{\mathcal{H}^k}^2 \geq \lVert  \mathbf{q} \rVert_{\mathcal{H}^k}^2.  
  \end{equation}
  and is maximal: 
  \[  \lambda - \mathbf{\hat{L}}_\alpha \text{ is surjective for some } \lambda > 0. 
  \] 
\end{proposition}
\begin{proof} 
  We first prove the dissipative estimate for smooth functions. Now for \(\mathbf{q} \in (C^\infty[-1,1])^2\) using Lemma \ref{lem:5.Commutator with derivatives}, we have 
  \begin{align*}
      &\langle  - \mathbf{L}_\alpha \mathbf{q}, \mathbf{q} \rangle_{\mathcal{H}^k}  = \langle  - \partial^{k+1}(\mathbf{L}_\alpha \mathbf{q})_1, \partial^{k+1}q_1 \rangle_{L^2} +  \langle  - \partial^{k}(\mathbf{L}_\alpha \mathbf{q})_2, \partial^k q_2 \rangle_{L^2} \\
      &\qquad\qquad +\langle  - (\mathbf{L}_\alpha \mathbf{q})_1, q_1 \rangle_{L^2} +  \langle  - (\mathbf{L}_\alpha \mathbf{q})_2, q_2 \rangle_{L^2} \\
      &= \langle -  (\mathbf{L}_{\alpha,k + 1} \partial^{k+1} \mathbf{q})_1 , \partial^{k+1} q_1 \rangle_{L^2} + \langle  - (\mathbf{L}'_{\alpha,k+1} \mathbf{q})_1,\partial^{k+1} q_1 \rangle_{L^2} \\
      &\qquad\qquad + \langle - ( \mathbf{L}_{\alpha,k } \partial^{k} \mathbf{q})_2 , \partial^{k} q_2 \rangle_{L^2} + \langle  - (\mathbf{L}'_{\alpha,k} \mathbf{q})_2, q_2 \rangle_{L^2}   + \langle - \mathbf{L}_{\alpha}\mathbf{q}, \mathbf{q} \rangle_{L^2 \times L^2}
  \end{align*} 
  For the terms involving \( \mathbf{L}_{\alpha,j}\) we have the following equalities: 
  \begin{align*}
           \langle  -  (\mathbf{L}_{\alpha,k + 1} \partial^{k+1} \mathbf{q} )_1, \partial^{k + 1} q_1 \rangle _{L^2} &= \langle  y \partial^{k + 2} q_1 + (k+1) \partial^{k + 1} q_1 - \partial^{k + 1} q_2 ,\partial^{k + 1} q_1 \rangle_{L^2} \\
      &=\frac{1}{2}\langle  y, \partial_y  \lvert  \partial^{k+1} q_1 \rvert^2  \rangle_{L^2 } + (k+1)\lVert  \partial^{k + 1}q_1 \rVert_{L^2}^2 - \langle  q_2, q_1 \rangle_{\dot{H}^{k+1}}, \\ 
      &=\frac{1}{2}  \lvert  \partial^{k+1} q_1 \rvert^2 (\pm 1) - \frac{1}{2}\lVert    q_1 \rVert_{\dot{H}^{k+1}}^2+ (k+1)\lVert    q_1 \rVert_{\dot{H}^{k+1}}^2 - \langle  q_2, q_1 \rangle_{\dot{H}^{k+1}} \\
      \langle  - (\mathbf{L}_{\alpha, k} \partial^{k}\mathbf{q})_2,\partial^k q_2  \rangle_{L^2}&= \left\langle  -\partial_{yy} \partial^k q_1 - (  - 1 - k  + \frac{2\alpha}{1 + \sqrt{ 1- \alpha }y}) \partial^k q_2 + y \partial^{k + 1}q_2, \partial^k q_2  \right\rangle\\
      &= \langle  - \partial^{ k + 2}q_1, \partial^k q_2 \rangle_{L^2} + (1 + k) \lVert \partial^k q_2 \rVert_{L^2}^2 \\
      &\qquad\qquad+   \langle \frac{2\alpha}{1+ \sqrt{1 -\alpha}y},  \lvert  \partial^k q_2 \rvert^2 \rangle_{L^2} + \frac{1}{2}\langle  y , \partial_y  \lvert \partial^k q_2 \rvert^2  \rangle _{L^2}\\
      &\geq - (\partial^{k + 1}q_1 \overline{\partial^k q_2})_{-1}^1 + \frac{1}{2}  \lvert \partial^k q_2 \rvert^2(\pm 1) - \frac{1}{2} \lVert  q_2 \rVert_{\dot{H}^k}^2+ \langle  q_1, q_2 \rangle_{\dot{H}^{k+1}}\\
      &\qquad\qquad +  (1 + k- \frac{2\alpha}{1-\sqrt{1 - \alpha}}  ) \lVert  q_2 \rVert_{\dot{H}^k}^2 . 
  \end{align*}  
  Summing these up, we get 
 \begin{align*}
       - \langle  (\mathbf{L}_{\alpha,k+1}\partial^{k+1}\mathbf{q})_1, \partial^{k+1}q_1  \rangle_{L^2} &- \langle  (\mathbf{L}_{\alpha,k} \mathbf{q})_2, \partial^k q_2 \rangle_{L^2} \\
       &\geq  (k+\frac{1}{2}) \lVert  q_1 \rVert_{\dot{H}^{k+1}}^2 + (k+ \frac{1}{2} - \frac{2\alpha }{1 - \sqrt{1 - \alpha}} )\lVert  q_2 \rVert_{\dot{H}^k}^2\\
       &\geq (k + \frac{1}{2}) \lVert  q_1 \rVert_{\dot{H}^{k+1}}^2 + ( k - \frac{3}{2} - 2 \sqrt{1 - \alpha} ) \lVert  q_2 \rVert_{\dot{H}^k}^2\\
       &\geq (k-\frac{3}{2} -2 \sqrt{ 1-\alpha}) \left(\lVert  q_1 \rVert_{\dot{H}^{k+1}}^2 + \lVert  q_2 \rVert_{\dot{H}^k}^2 \right). 
 \end{align*} 
 For the remaining terms, we have via Young's inequality and Lemma \ref{lem:5.Commutator with derivatives}, 
 \begin{align*}
      \lvert  \langle  \mathbf{L}'_{\alpha,k+1} \mathbf{q}, \mathbf{q} \rangle_{\dot{H}^{k+1} \times \dot{H}^k} \rvert +  \lvert  \langle  \mathbf{L} \mathbf{q}, \mathbf{q} \rangle_{L^2 \times  L^2} \rvert \leq \frac{\varepsilon}{2} \lVert  q_2 \rVert_{\dot{H}^k}^2  + c_{\varepsilon,k,\alpha} \lVert  \mathbf{q} \rVert_{\mathcal{H}^{k-1}}^2. 
 \end{align*} 
 Finally, using Lemma \ref{lem:5.Subcoercity}, we can write 
 \begin{align*}
     \mathbb{R}{\rm e\,} \langle  - \mathbf{L}_\alpha\mathbf{q}, \mathbf{q} \rangle_{\mathcal{H}^k} &\geq  (k - \frac{3}{2} -  2 \sqrt{ 1 -\alpha} - \frac{\varepsilon}{2}) \lVert  \mathbf{q} \rVert_{\mathcal{H}^k}^2 - c_{\varepsilon,k,\alpha} \lVert  \mathbf{q} \rVert_{\mathcal{H}^{k-1}}^2 \\
     &\geq  (k - \frac{3}{2} -   2 \sqrt{ 1 -\alpha} - \frac{\varepsilon}{2}) \lVert  \mathbf{q} \rVert_{\mathcal{H}^k}^2 - \frac{\varepsilon}{2} \lVert  \mathbf{q} \rVert_{\mathcal{H}^k}^2 \\
     &\qquad\qquad- c'_{\varepsilon,k,\alpha} \left(\sum_{ i \leq n_{1}}  \lvert  \langle  q_1, \Pi _{i}^{(1)} \rangle_{L^2} \rvert ^2 + \sum_{i \leq n_2}  \lvert  \langle  q_2, \Pi ^{(2)}_i \rangle_{L^2} \rvert^2  \right)
 \end{align*} 
 where \(\Pi ^{(j)}_i \in H^{k+2- j}\). Via Riesz theorem, the result is complete. 

 Now we claim that the inequality holds for all \( \mathbf{q} \in  \mathcal{D}(\mathbf{L}_\alpha) \). Fix \(\mathbf{q} \in \mathcal{D}(\mathbf{L}_\alpha).\) Recall that 
 \[  \mathbf{L}_\alpha \mathbf{q} 
 := \mathbf{L}_\square  \mathbf{q} + \mathbf{V}_\alpha \mathbf{q} 
 \] 
and \(\mathcal{D}(\mathbf{L}_\alpha) = \mathcal{D}(\mathbf{L}_\square) \subset \mathcal{H}^k\) where \( \mathbf{L}_\square \) is the closure of \(\tilde{\mathbf{L}}_\square : \mathcal{D}(\tilde{\mathbf{L}}_\square) \subset \mathcal{H}^k \to \mathcal{H}^k\). Thus \(\mathbf{q} \in \mathcal{D}(\mathbf{L}_\square)\) and hence \( \mathbf{L}_\square\mathbf{q}, \mathbf{L}_\alpha \mathbf{q} \in \mathcal{H}^k\). As \( (C^\infty[-1,1])^2 \subset \mathcal{H}^k\) is dense, we can fix \( \mathbf{f}_n \in (C^\infty[-1,1])^2 : \) 
\[  \mathbf{f}_n \to - \mathbf{L}_\square \mathbf{q} .  
\] 
As argued in Lemma \ref{lem:5.Dense range of Wave}, \(\mathbf{L}_\square : (C^\infty[-1,1])^2 \to (C^\infty[-1,1])^2\) is bijective; thus \(\exists ! \mathbf{q}_n \in (C^\infty[-1,1])^2 : \) 
\[   \mathbf{f} _n = - \mathbf{L}_\square\mathbf{q}_n. 
\] 
We claim that \( \mathbf{q}_n \xrightarrow{\langle\,,\, \rangle_k} \mathbf{q}.  \) By Lemma \ref{lem:5.Dissipativity of Wave}, we have 
\[ \lVert  \mathbf{q}_n -\mathbf{q}_m \rVert_{k}^2 \lesssim  \mathbb{R}{\rm e\,} \langle  - \mathbf{L}_\square(\mathbf{q}_n - \mathbf{q}_m ), \mathbf{q}_n - \mathbf{q}_m \rangle _k  \lesssim \lVert  \mathbf{L}_\square\mathbf{q}_n - \mathbf{L}_\square \mathbf{q}_m  \rVert_{k} \lVert  \mathbf{q}_n - \mathbf{q}_m \rVert_{k} \lesssim  \lVert  \mathbf{f}_n - \mathbf{f}_m \rVert_{k}\lVert  \mathbf{q}_n - \mathbf{q}_m \rVert_{k}. 
\] 
As \(\mathbf{f}_n\) is convergent so it is Cauchy-convergent and thus so is \( \mathbf{q}_n\). In particular, \(\mathbf{q}_n \xrightarrow{\langle\,,\, \rangle_k} \mathbf{q}' \in \mathcal{H}^k \). As such \( [ \mathbf{q}_n , -\mathbf{L}_\square\mathbf{q}_n] \to [ \mathbf{q}', - \mathbf{L}_\square \mathbf{q}]\) and since \(\mathbf{L}_\square\) is closed so we have \( \mathbf{q}' \in \mathcal{D}(\mathbf{L}_\square)\) and \( \mathbf{L}_\square\mathbf{q}' = \mathbf{L}_\square\mathbf{q}\). By Proposition \ref{prop:5.Semigroup of Wave}, we have \(\lambda \in \rho(\mathbf{L}_\square) , \forall \lambda :  \mathbb{R}{\rm e\,}\lambda > - \frac{1}{2}\) thus \( \mathbf{L}_\square\) is invertible implying \( \mathbf{q}' = \mathbf{q}\). Finally, as \( \mathbf{\hat{L}}_\alpha = \mathbf{L}_\alpha - \mathbf{\hat{P}}_\alpha = \mathbf{L}_\square + \mathbf{V}_{\alpha}  - \mathbf{\hat{P}}_\alpha \) where \(\mathbf{V}_{\alpha}\) and \( \mathbf{\hat{P}}_\alpha\) are bounded operators so \( \mathbf{\hat{L}}_\alpha \mathbf{q}_n \to \mathbf{\hat{L}}_\alpha \mathbf{q}\). As such, the dissipative inequality extends to all \( \mathbf{q} \in \mathcal{D}(\mathbf{L}_\alpha)\). 

It remains to prove \( \mathbf{\hat{L}}_\alpha\) is maximal. Note that \( \mathbf{L}_\alpha\) generates a strongly continuous semigroup by Proposition \ref{prop:5.Semigroup Sα} and so by bounded perturbation theorem so does \( \mathbf{\hat{L}}_\alpha\) as \( \mathbf{\hat{L}}_\alpha = \mathbf{\hat{L}}_\alpha - \mathbf{\hat{P}}_\alpha\) where \( \mathbf{\hat{P}}_\alpha\) is a bounded perturbation. So by Hille-Yosida theorem, there exists a \(\lambda\) large enough such that \( \lambda - \mathbf{\hat{L}}_\alpha\) is invertible and hence surjective. 
\end{proof}  
\subsection{Spectral Analysis of the generator.}
Now we can give a sufficiently detailed description of the spectrum of \(\mathbf{L}_\alpha\). 

\begin{proposition}\label{prop:5.Spectrum of Lα}
  For \(\alpha \in (0,1], k \geq 5\), we have 
  \begin{equation}\label{eq:5.Spectrum of Lα}
    \sigma(\mathbf{L}_\alpha) \subset \{z \in \mathbb{C} : \mathbb{R}{\rm e\,} z\leq - 1\} \cup \{0,1\}
  \end{equation}
  and \( \{0,1\} \subset \sigma _p(\mathbf{L}_\alpha)\). Moreover, the geometric eigenspaces of eigenvalues \(1\) and \(0\) are spanned by the following functions respectively:  
  \[  \mathbf{f}_{1,\alpha}(y): = \begin{pmatrix}   \frac{\alpha }{1 +  \sqrt{1 - \alpha} y}\\ \frac{\alpha }{(1 +  \sqrt{1 - \alpha} y)^2}\end{pmatrix}, \quad \mathbf{f}_{0,\alpha} := \begin{pmatrix}1 \\0 \end{pmatrix}.  
  \] 
  Moreover for \(\alpha \in (0,1)\), there exists generalised eigenfunctions \(\mathbf{g}_{0,\alpha}\) for the eigenvalue \(0\) 
  \[  \mathbf{g}_{0,\alpha}  = \begin{pmatrix} - \sqrt{ 1- \alpha}\ln ( 1 +  \sqrt{ 1- \alpha} y) + \frac{ \alpha y}{2(1 +  \sqrt{ 1- \alpha} y)} \\ \sqrt{ 1- \alpha} -  \frac{(1- \alpha) y }{ 1 + \sqrt{1 - \alpha} y} + \frac{\alpha y}{ 2(1 + \sqrt{ 1- \alpha}y)^2} \end{pmatrix}
  \]  
  which satisfy 
  \[  \mathbf{L}_\alpha  \mathbf{g}_{0,\alpha} =( 1- \alpha) \mathbf{f}_{0,\alpha}, \quad \mathbf{L}_\alpha^2 \mathbf{g}_{0,\alpha} = 0,
  \]  
  and for \( \alpha = 1,\) \(\mathbf{g}_{0,1}\) is reduced to an eigenfunction for eigenvalue \( 0\). 
\end{proposition}
\begin{proof} 
  By maximal dissipativity of \(\mathbf{L}_\alpha\), as shown in Proposition \ref{prop:5.Maximal Dissipativity of L̂α}, we have \( \sigma(\mathbf{\hat{L}}_\alpha) \subset \{z \in \mathbb{C} : \mathbb{R}{\rm e\,} z \leq - 1\}\) for \( k \geq 5\). Moreover, \( \mathbf{L}_\alpha = \mathbf{\hat{L}}_\alpha + \mathbf{\hat{P}}_\alpha \) where \( \mathbf{\hat{P}}_\alpha\) is finite-rank and thus compact so the essential spectrum of \( \mathbf{L}_\alpha\) remains unperturbed. Finally, the results in section 3 relating to the mode-stability of \(\mathbf{L}_\alpha\) shows that all the eigenvalues of \( \mathbf{L}_\alpha\) must belong to the set \( \{0,1\}\). 

  Moreover, a direct computation shows that \(\mathbf{f}_{0,\alpha}\) and \( \mathbf{f}_{1,\alpha}\) are the eigenfunctions of \( 0\) and \(1\) respectively, and \( \mathbf{g}_{0,\alpha}\) is a generalised eigenfunction for \(0\). Moreover, Frobenius analysis shows that the local smooth solutions around \(1\) are always one-dimensional and thus the eigenspaces are one-dimensional. \qedhere 

\end{proof}  

Now we define the usual Riesz projections associated to the eigenvalues \(0\) and \(1\) of \(\mathbf{L}_\alpha\). We set 
\[  \mathbf{P}_{0,\alpha}  := \frac{1}{2\pi i } \int_{ \gamma _0} \mathbf{R}_{\mathbf{L}_{\alpha,\infty,\lambda}} , \quad \mathbf{P}_{1,\alpha}  := \frac{1}{2\pi i } \int_{ \gamma _1} \mathbf{R}_{\mathbf{L}_{\alpha,\infty,\lambda}} 
\] 
where we define the curves \( \gamma _0 , \gamma _1 : [0,1 ] \to \mathbb{C}\) by 
\[ \gamma _0(s) : = \frac{1}{4}e^{2\pi i s} , \quad  \gamma _1(s) := 1 + \frac{1}{4}e^{2\pi i s}. 
\] 
We will now show that the algebraic multiplicity of \(1\) equals \(0\) whereas the algebraic multiplicity of \(0\) is strictly larger than its geometric multiplicity. As such, the generalised \(0\)-mode would induce a polynomial growth in time. 

\begin{lemma}\label{lem:5.Ranks of Projections}
  The projections \( \mathbf{P}_{0,\alpha}\) and \( \mathbf{P}_{1,\alpha}\) have rank \(2\) and \(1\) respectively. 
\end{lemma}
\begin{proof} 
  Note that it suffices to fix \(\alpha := 1\). Indeed, these projections are finite-rank since if the rank was infinite then the respective eigenvalues would have infinite algebraic multiplicity and thus belong to the essential spectrum of \(\mathbf{L}_\alpha\). But \(\sigma _e(\mathbf{L}_\alpha) =  \sigma _e(\mathbf{\hat{L}}_\alpha) \subset \left\{z \in \mathbb{R}{\rm e\,} z \leq -1\right\} \), by maximal dissipativity of \(1 + \mathbf{\hat{L}}_\alpha\). Now since \(\alpha \mapsto \mathbf{P}_{i,\alpha}\) is a continuous map so by Lemma 4.10 in Kato 1.4.6, the rank of \( \mathbf{P}_{i,\alpha}\) remains constant as \(\alpha\) is varied.

  Now for \(\mathbf{P}_{0,\alpha}\), note that Theorem 6.17 
  implies that \( \mathbf{L}_\alpha|_{\mathrm{ran\,} \mathbf{P}_{0,\alpha}}\) is a bounded operator and per the description of the spectrum of \(\mathbf{L}_\alpha\) given in Proposition \ref{prop:5.Spectrum of Lα}, we have 
  \[  \sigma(\mathbf{L}_\alpha|_{\mathrm{ran\,} \mathbf{P}_{0,\alpha}})  = \{ z :  \lvert z \rvert < \tfrac{1}{4} \text{ and } z \in \sigma(\mathbf{L}_\alpha) \} = \{0\}. 
  \] 
  As \(\mathbf{P}_{0,\alpha}\) is a finite-rank operator with spectrum equalling \(0\) so \( \mathbf{L}_\alpha|_{\mathrm{ran\,} \mathbf{P}_{0,\alpha}}\) is a nilpotent operator and thus there exists a minimal \( n \in \mathbb{N} \) such that \(( \mathbf{L}_\alpha|_{\mathrm{ran\,} \mathbf{P}_{0,\alpha}})^n = 0 \). As the algebraic multiplicity of the eigenvalue \( 0\) equals \(n\), it suffices to calculate this \(n\) to determine the rank of \(\mathbf{P}_{0,\alpha}\). 

  Recall that 
  \[  \mathrm{ran\,} \mathbf{P}_{0,\alpha} = \left\{ \mathbf{x} \in \mathcal{H}^k : \lim_{n \to \infty} \lVert  (\mathbf{L}_\alpha)^n \mathbf{x}  \rVert_{k}^{\frac{1}{n}} = 0  \right\}. 
  \] 
  Thus the eigenspaces spanned by the generalised eigenvectors are contained in \(\mathrm{ran\,} \mathbf{P}_{0,\alpha}\) so \(\mathbf{f}_{0,\alpha}, \mathbf{g}_{0,\alpha} \subset \mathrm{ran\,} \mathbf{P}_{0,\alpha}\) and  
  \[   (\mathbf{L}_\alpha|_{\mathrm{ran\,} \mathbf{P}_{0,\alpha}})^2 \mathbf{g}_{0,\alpha}  = (\mathbf{L}_\alpha )^2 \mathbf{g}_{0,\alpha}  = 0 . 
  \] 
  This implies \(n \geq 2\). If \(n \geq 3\) then there exists \(  \mathbf{v} \in \mathrm{ran\,} \mathbf{P}_{0,\alpha} : (\mathbf{L}_\alpha|_{\mathrm{ran\,} \mathbf{P}_{0,\alpha}}) \mathbf{v} \in \ker   (\mathbf{L}_\alpha|_{\mathrm{ran\,} \mathbf{P}_{0,\alpha}})^2 \backslash \{0\} \). Since \(\ker  (\mathbf{L}_\alpha)^2\) is spanned by \( \mathbf{f}_{0,\alpha}\) and \(\mathbf{g}_{0,\alpha}\), we can write 
  \[   (\mathbf{L}_\alpha|_{\mathrm{ran\,} \mathbf{P}_{0,\alpha}}) \mathbf{v} = c_1 \mathbf{f}_{0,\alpha} + c_2 \mathbf{g}_{0,\alpha} = c_1 \mathbf{L}_\alpha \mathbf{g}_{0,\alpha} + c_2 \mathbf{g}_{0,\alpha}. 
  \] 
  As \(\mathbf{v} \in \mathrm{ran\,}\mathbf{P}_{0,\alpha}\) we have \( \mathbf{L}_\alpha (\mathbf{v} - c_1 \mathbf{g}_{0,\alpha})= c_2 \mathbf{g}_{0,\alpha}\). This is only possible if \(c_2 = 0\) per Lemma \ref{lem:6.Non-existence of other generalised eigenfunctions} implying \( \mathbf{v} - c_1 \mathbf{g}_{0,\alpha} \in \ker  \mathbf{L}_\alpha\) and so \(  \mathbf{v} = c_3 \mathbf{f}_{0,\alpha} + c_1 \mathbf{g}_{0,\alpha}\) contradicting \((\mathbf{L}_\alpha|_{\mathrm{ran\,} \mathbf{P}_{0,\alpha}})^2 \mathbf{v} \neq 0 . \) A similar argument shows that \( \mathbf{P}_{1,\alpha}\) has rank \(1\). 
\end{proof}  
\subsection{The resolvent estimate}
For some \( 0 < w_0 < 1\), we define the rectangle
\[  R_{m,n} := \left\{ z \in \mathbb{C} : \mathbb{R}{\rm e\,} z \in [ - w_0 , m], \mathrm{Im\,} z \in [-n,n]\right\}
\] 
and 
\[  R'_{m,n} := \{z \in \mathbb{C} : \mathbb{R}{\rm e\,} z \geq - w_0 \}  \backslash R_{m,n}. 
\] 
The following estimate confines the possible unstable eigenvalues of \( \mathbf{L}_\alpha\) to the compact rectangle \(R_{m,n}\). 

\begin{proposition}\label{prop:5.Uniform bound on the resolvent outside R}
  Let \(  k \geq 5, 0 < w_0 < 1,\) and \( I \subset (0,1]\) be a compact region. Then there exists \(m,n\) such that \( R_{m,n}' \subset \rho(\mathbf{L}_\alpha)\) and 
  \begin{equation}\label{eq:5.Uniform bound on the resolvent outside R}
    \sup_{\alpha \in I} \sup_{\lambda \in R_{m,n}'}\, \lVert  \mathbf{R}_{\mathbf{L}_{\alpha}} \rVert_{\mathcal{L}(\mathcal{H}^k)} \leq  C
  \end{equation}
  where \(C\) is a uniform constant. 
\end{proposition}
\begin{proof} 
   
   Since \( \sigma(\mathbf{L}_\alpha) \subset \{ \mathbb{R}{\rm e\,} z \geq - 1\} \cup \{0,1\}\) 
   by Proposition \ref{prop:5.Spectrum of Lα} we have \( R_{m,n}' \subset \{\mathbb{R}{\rm e\,} z  > -1\} \cap \{0,1\}^c \subset \rho(\mathbf{L}_\alpha)\). So to prove the resolvent estimate, we have to show 
  \begin{equation}\label{eq:5.Resolvent estimate reformulation in Hk}
     \lvert  \langle  (\lambda - \mathbf{L}_\alpha) \mathbf{q} , \mathbf{q} \rangle_{\dot{H}^{k +1} \times  \dot{H}^k} \rvert +  \lvert  \langle  (\lambda - \mathbf{L}_\alpha) \mathbf{q} , \mathbf{q} \rangle_{0} \rvert \gtrsim  \lVert  \mathbf{q} \rVert_{\mathcal{H}^k}^2, \quad \forall  \lambda \in R_{m,n}', \ \alpha \in I, \   \mathbf{q} \in (C^\infty[-1,1])^2
  \end{equation}
  since under an application of Cauchy-Schwartz, it is equivalent to the bound \( \lVert  (\lambda - \mathbf{L}_\alpha) \mathbf{q} \rVert_{\mathcal{H}^k} \gtrsim  \lVert  \mathbf{q} \rVert_{\mathcal{H}^k}\) satisfied by smooth \(\mathbf{q}\) which are dense in \( \mathcal{D}(\mathbf{L}_\alpha)\).

To prove \eqref{eq:5.Resolvent estimate reformulation in Hk}, we first subdivide \(R_{m,n}'\) into the outer half \(O_{m,n}'\), and the middle half \(M_{m,n}'\) as 
\[   O_{m,n}' :=  \{z \in \mathbb{C} : \mathbb{R}{\rm e\,} z \geq  -w_0  ,  \lvert \mathrm{Im\,}z \rvert \geq n\}, \quad M_{m,n}' := \{ z\in \mathbb{C} : \mathbb{R}{\rm e\,} z \geq m ,  \lvert \mathrm{Im\,}z \rvert \leq n\}. 
\] 
The first term in \eqref{eq:5.Resolvent estimate reformulation in Hk} can be expanded as 
\begin{align*}
    &\langle  (\lambda - \mathbf{L}_\alpha) \mathbf{q}, \mathbf{q} \rangle_{\dot{H}^{k+ 1} \times \dot{H}^k} = \langle  \partial^{k + 1}((\lambda  - \mathbf{L}_\alpha)\mathbf{q})_1, \partial^{ k + 1}\mathbf{q}_1 \rangle_{L^2} + \langle \partial^k (( \lambda - \mathbf{L}_\alpha) \mathbf{q})_2, \partial^k\mathbf{q}_2  \rangle _{L^2} \\
    &= \langle  \partial^{k + 1} ( \lambda q_1 + y q_1' - q_2) , \partial^{ k + 1} q_1 \rangle_{L^2}  + \langle  \partial^k ( \lambda q_2  - q_1'' - \Upsilon_\alpha(y)q_2 + q_2 + y q_2' ), \partial^k q_2 \rangle_{L^2} \\
    &= \lambda \lVert  q_1 \rVert_{\dot{H}^{k + 1}}^2 + \langle  y \partial^{ k + 1} q_1' + (k + 1) \partial^{ k + 1} q_1 , \partial^{k + 1} q_1  \rangle_{L^2}  - \langle  q_2, q_1 \rangle_{\dot{H}^{k+1}} \\
    &\qquad\qquad + \lambda \lVert  q_2 \rVert_{\dot{H}^k}^2 - \langle  q_1'', q_2 \rangle_{\dot{H}^k} - \langle  \Upsilon _\alpha(y) q_2, q_2 \rangle_{\dot{H}^k}  + \lVert  q_2 \rVert_{\dot{H}^k}^2 + \langle  y \partial^k q_2' +k \partial^k q_2, \partial^kq_2 \rangle_{L^2} \\
    &= ( \mathbb{R}{\rm e\,}\lambda +k+  \tfrac{1}{2}) \lVert  q_1 \rVert_{\dot{H}^{k+ 1}}^2 + \tfrac{1}{2}  (\lvert  \partial^{ k + 1} q_1 \rvert^{2} +  \lvert \partial^k q_2 \rvert^2)(\pm 1 )  +(\mathbb{R}{\rm e\,}\lambda +k+  \tfrac{1}{2}) \lVert q_2 \rVert_{\dot{H}^k}^2 \\
    &\qquad\qquad- \partial^{k + 1} q_1 \overline{\partial^k  q_2}|^{ + 1}_{-1} -  \langle  \Upsilon _\alpha(y) q_2, q_2  \rangle_{\dot{H}^k}   \\
    &\qquad\qquad + i \mathrm{Im\,} \left(  \lambda \lVert  q_1 \rVert_{\dot{H}^{k+1}}^2 + \lambda \lVert  q_2 \rVert_{\dot{H}^k}^2 +  \langle  y \partial^{k + 2}q_1, \partial^{k + 1}q_1 \rangle_{L^2} +\langle  y \partial^{k+1} q_2 , \partial^kq_2  \rangle_{L^2}  +   \langle  q_1, q_2 \rangle_{\dot{H}^{k+ 1}}\right)
\end{align*} 
where \(\Upsilon _\alpha(y) :=\frac{2 \alpha}{1+ \sqrt{ 1- \alpha}y }\) and we have used 
\[ \mathbb{R}{\rm e\,}  \langle  y f', f \rangle _{L^2} = \langle  y,  \tfrac{1}{2}(f' \bar{f} + f \bar{f}') \rangle = \langle y, \tfrac{1}{2} ( \lvert f \rvert^2)'  \rangle = \tfrac{1}{2}  \lvert f \rvert^2( 1) + \tfrac{1}{2}  \lvert f \rvert^2(-1) =: \tfrac{1}{2}  \lvert f \rvert^2(\pm 1). 
\] 
For the term with \( \Upsilon _\alpha\), we have 
\begin{align*}
       \langle  \Upsilon _\alpha q_2, q_2 \rangle_{\dot{H}^k} &= \langle  \Upsilon _\alpha \partial^k q_2, \partial^k q_2 \rangle_{L^2} + \sum_{j = 1}^k  C^k_j \langle  \partial^{j} \Upsilon _\alpha  \partial^{k-j} q_2 , \partial^k q_2 \rangle_{L^2} \\
       &\leq \frac{2 \alpha}{1 - \sqrt{ 1- \alpha}} \lVert  q_2 \rVert_{\dot{H}^k}^2 + M_{\alpha _0, \varepsilon'} \lVert  q_2 \rVert_{L^2}^2 + \varepsilon' \lVert  \partial^k q_2 \rVert_{L^2}^2, \quad \forall  \alpha \in B_{\varepsilon}(\alpha _0 )
\end{align*} 
where we have used Gagliardo-Nirenberg in bounded domain and \(\varepsilon'\) is sufficiently small. We have thus obtained 
\begin{align*}
     \lvert  \langle ( \lambda - \mathbf{L}_\alpha)\mathbf{q}, \mathbf{q} \rangle_{\dot{H}^{k+1} \times  \dot{H}^k} \rvert  &+ M \lVert  q_2 \rVert_{L^2}^2\\
      &\geq (\mathbb{R}{\rm e\,} \lambda + k  + 1) \lVert  q_1 \rVert_{\dot{H}^{k+1}}^2  + \left(\mathbb{R}{\rm e\,} \lambda + k  - \frac{3}{2} - 2 \sqrt{ 1- \alpha} - \varepsilon'\right) \lVert  q_2 \rVert_{\dot{H}^k}^2\\
     &\geq \left(\mathbb{R}{\rm e\,} \lambda + \frac{3}{2} - \varepsilon'\right) \lVert  \mathbf{q} \rVert_{\mathcal{H}^k}^2  ,  
\end{align*} 
for \(k \geq 5\). For \(k = 0\), we instead use the imaginary terms: 
\begin{align*}
     \lvert  \langle  y q_1'' , q_1' \rangle _{L^2}\rvert& \leq \lVert  q_1'' \rVert_{L^2} \lVert q_1' \rVert_{L^2} \leq \varepsilon' \lVert  q_1 \rVert_{\dot{H}^{k+1}}^2 + M_1 \lVert  q_1 \rVert_{\dot{H}^1}^2, \\
      \lvert  \langle  y q_2', q_2 \rangle_{L^2} \rvert &\leq \lVert  q_2' \rVert_{L^2} \lVert  q_2 \rVert_{L^2} \leq \varepsilon' \lVert  q_2 \rVert_{\dot{H}^k}^2 + M_2 \lVert  q_2 \rVert_{L^2}. 
\end{align*} 
Thus we have 
\begin{align*}
     \lvert  \langle  (\lambda - \mathbf{L}_\alpha )\mathbf{q} , \mathbf{q} \rangle_{\dot{H}^1 \times  L^2} \rvert  + \varepsilon' \lVert \mathbf{q} \rVert_{\dot{H}^{k + 1} \times  \dot{H}^k} ^2\geq  (\mathrm{Im\,} \lambda - M_1 - M_2) \lVert  \mathbf{q} \rVert_{\dot{H}^1 \times L^2}^2. 
\end{align*} 
Finally, it remains to bound the \(L^2\)-norm of the first term. We have 
\begin{align*}
    \langle ((\lambda - \mathbf{L}_\alpha)\mathbf{q})_1, q_1 \rangle_{L^2} &= \langle  (\lambda q_1 + y q_1 ' - q_2) , q_1 \rangle_{L^2} = \mathbb{R}{\rm e\,}\lambda \lVert  q_1 \rVert_{L^2}^2 + \frac{1}{2}  \lvert q_1 \rvert^2(\pm 1)  - \langle  q_2, q_1 \rangle_{L^2}  \\
    & \qquad\qquad + i\left( \mathrm{Im\,} \lambda \lVert  q_1 \rVert_{L^2}^2 +  \mathrm{Im\,} \langle y q_1' , q_1 \rangle\right)
\end{align*} 
where 
\[  \langle  yq_1', q_1 \rangle_{L^2} \leq \lVert q_1' \rVert_{L^2} \lVert  q_1 \rVert_{L^2} \leq \varepsilon' \lVert  q_1' \rVert_{L^2}^2 + M_3 \lVert  q_1 \rVert_{L^2}^2. 
\] 
We obtain 
\[  \lvert  \langle  ((\lambda - \mathbf{L}_\alpha)\mathbf{q})_1 , q_1 \rangle _{L^2}  \rvert + \varepsilon' \lVert  q_1 \rVert_{\dot{H}^1}^2 \geq (\mathrm{Im\,} \lambda - M_3) \lVert q_1 \rVert_{L^2}^2. 
\] 
Fixing \( \mathrm{Im\,} \lambda \geq n :=  M_1 + M_2 + M_3   \) and \( \mathbb{R}{\rm e\,} \lambda > -w_0\), we obtain the resolvent bound for \( \lambda \in O_{m,n}':\)
\[   \lvert  \langle  (\lambda - \mathbf{L}_\alpha)\mathbf{q} , \mathbf{q} \rangle_{\mathcal{H}^k} \rvert  \geq \frac{1}{2} \lVert  \mathbf{q} \rVert_{\mathcal{H}^k}^2. 
\] 
For \( \lambda \in M_{m,n}'\), we have, by Proposition \ref{prop:5.Semigroup Sα}, 
\[   \lVert  \mathbf{S}_\alpha(s) \rVert_{\mathcal{L}(\mathcal{H}^k)} \lesssim  e^{ ( \lVert  \mathbf{V}_\alpha \rVert_{\mathcal{H}^k} - \frac{1}{2})s}
\] 
where \( S_\alpha\) is the semigroup generated by \( \mathbf{L}_\alpha\). This is equivalent to the following resolvent bound 
\[  \lVert  \mathbf{R}_{\mathbf{L}_\alpha} \rVert_{\mathcal{L}(\mathcal{H}^k)}  \lesssim \frac{1}{\mathbb{R}{\rm e\,} \lambda - ( \lVert  \mathbf{V}_\alpha \rVert_{\mathcal{H}^k} - \frac{1}{2})}. 
\] 
Since 
\[  \lVert  \mathbf{V}_\alpha \rVert_{\mathcal{L}(\mathcal{H}^k)}  \leq 1 + \left\lVert  \frac{2 \alpha}{y \sqrt{ 1- \alpha} + 1} \right\rVert_{C^k} \lesssim  C_{I}, \quad \forall \alpha \in I
\] 
so fixing \(  m \geq  C_{I}\) completes the proof. 
\end{proof}  

\section{Growth bound of the linearised flow}\label{sec:Growth bound of the linearised flow}
The above spectral analysis provides a sufficiently complete description of the linearised flow. 

\begin{proposition}\label{prop:6.Properties of Sα}
  Let \(  I \subset (0,1]\) be compact and \( k \geq 5, 0 < w_0 < 1\). Then the flow \( \mathbf{S}_\alpha(s)\) for \(\alpha \in I \) generated by \( \mathbf{L}_\alpha\) satisfies the following properties: 
  \begin{itemize}
    \item \textbf{(Commutes with projections).} \( [\mathbf{S}_\alpha, \mathbf{P}_{1,\alpha}] = [\mathbf{S}_\alpha, \mathbf{P}_{0,\alpha}] = 0 \). 
    \item \textbf{(Projection on unit eigenspace).} \( \mathbf{S}_\alpha(s) \mathbf{P}_{1,\alpha} = e^s \mathbf{P}_{1,\alpha}\), 
    \item \textbf{(Projection on null eigenspace).} \(\mathbf{S}_\alpha(s) \mathbf{P}_{0,\alpha} = \mathbf{P}_{0,\alpha} + s \mathbf{L}_\alpha \mathbf{P}_{0,\alpha}\), 
    \item \textbf{(Boundedness on remaining subspace).} \( \lVert  \mathbf{S}_\alpha(s) \mathbf{\tilde{P}}_\alpha \mathbf{q} \rVert_{\mathcal{H}^k} \leq M_{I,w_0}e^{-w_0 s} \lVert  \mathbf{\tilde{P}}_\alpha \mathbf{q} \rVert_{\mathcal{H}^k}\),
  \end{itemize}
  for all \(s \geq 0 , \mathbf{q} \in \mathcal{H}^k\) and where \( \mathbf{\tilde{P}}_\alpha = \mathbf{I} - \mathbf{P}_{0,\alpha} - \mathbf{P}_{1,\alpha}\) and \(M_{I,w_0}\) is a constant depending on \(I , w_0\). 

  Moreover, we have the following description of the eigenspaces: 
  \begin{itemize}
    \item \(\mathrm{ran\,}\mathbf{P}_{1,\alpha} = \mathbb{C} \mathbf{f}_{1,\alpha}\), 
    \item \(\mathrm{ran\,} \mathbf{P}_{0,\alpha} = \mathbb{C} \mathbf{f}_{0,\alpha} \oplus \mathbb{C}\mathbf{g}_{0,\alpha}\),
  \end{itemize} 
  where \( (1,\mathbf{f}_{1,\alpha})\) is the eigenpair and \( (0,(\mathbf{g}_{0,\alpha}, \mathbf{f}_{1,\alpha}))\) is the generalised eigen-cycle for \( \mathbf{L}_\alpha\) as shown in Proposition \ref{prop:5.Spectrum of Lα}. 
\end{proposition}
\begin{proof} 
  Since \( \mathbf{S}_\alpha\) is strongly-continuous semigroup so it commutes with its generator and thus with the resolvent of the generator \( \mathbf{R}_{\mathbf{L}_\alpha}\). Since the projections can be expressed as an integral of the resolvent, so \(\mathbf{S}_\alpha\) commutes with \( \mathbf{P}_{0,\alpha}\) and \( \mathbf{P}_{1,\alpha}\). Moreover, Proposition \ref{prop:5.Spectrum of Lα} and Lemma \ref{lem:5.Ranks of Projections} imply 
  \[  \mathrm{ran\,} \mathbf{P}_{0,\alpha} = \mathbb{C} \mathbf{f}_{0,\alpha} \oplus \mathbb{C} \mathbf{g}_{0,\alpha}  , \quad \mathrm{ran\,} \mathbf{P}_{1,\alpha}  = \mathbb{C} \mathbf{f}_{1,\alpha}. 
  \] 
  Thus we obtain 
  \[  \partial_s \mathbf{S}_\alpha \mathbf{P}_{1,\alpha}  =\mathbf{L}_\alpha \mathbf{S}_\alpha \mathbf{P}_{1,\alpha} = \mathbf{S}_\alpha \mathbf{L}_\alpha \mathbf{P}_{1,\alpha} = \mathbf{S}_\alpha \mathbf{P}_{1,\alpha} \implies \mathbf{S}_\alpha \mathbf{P}_{1,\alpha} = e^{ s}\mathbf{P}_{1,\alpha} . 
  \] 
  Similarly, we also have 
  \[  \partial_{ss} \mathbf{S}_\alpha \mathbf{P}_{0,\alpha}   = \partial_s \mathbf{S}_\alpha \mathbf{L} _\alpha \mathbf{P}_{0,\alpha}=  \mathbf{S}_\alpha \mathbf{L}_\alpha ^2 \mathbf{P}_{0,\alpha} = 0 \implies \mathbf{S}_\alpha \mathbf{P}_{0,\alpha} = s \mathbf{L}_\alpha \mathbf{P}_{0,\alpha} + \mathbf{P}_{0,\alpha}. 
  \] 
  Finally, for the restriction to the range of \( \mathbf{\tilde{P}}_\alpha = 1 - \mathbf{P}_{0,\alpha} - \mathbf{P}_{1,\alpha}\), using the separation of spectrum theorem, we obtain 
  \[ \sigma(\mathbf{L}_\alpha|_{\mathrm{ran\,} \mathbf{\tilde{P}}_\alpha}) = \sigma(\mathbf{L}_\alpha)\backslash \{0,1\} \subset \left\{z \in \mathbb{C} : \mathbb{R}{\rm e\,} z \leq - 1\right\}. 
  \]
  This implies \( R_{m,n} \subset \{\mathbb{R}{\rm e\,} z > -1\} \subset \rho(\mathbf{R}_{\mathbf{L}_\alpha |_{\mathrm{ran\,} \mathbf{\tilde{P}}_\alpha}})\) and so \( \mathbf{R}_{\mathbf{L}_\alpha |_{\mathrm{ran\,} \mathbf{\tilde{P}}_\alpha}}\) is analytic in the compact set \( R_{m,n}\) so \( \Sup{\lambda \in R_{m,n}}\, \lVert  \mathbf{R}_{\mathbf{L}_\alpha |_{\mathrm{ran\,} \mathbf{\tilde{P}}_\alpha}} \rVert \lesssim_{I} 1\) and from Proposition \ref{prop:5.Uniform bound on the resolvent outside R}, we have 
  \[  \sup_{\lambda \in R_{m,n}'}\,   \lVert  \mathbf{R}_{\mathbf{L}_\alpha|_{\mathrm{ran\,} \mathbf{\tilde{P}}_\alpha}} \rVert \lesssim _{I} 1. 
  \] 
  Since \( R_{m,n} \cup R_{m,n}' = \{\mathbb{R}{\rm e\,} z \geq - w_0\}\) so
  \[  \sup_{\mathbb{R}{\rm e\,} \lambda \geq - w_0}\,   \lVert  \mathbf{R}_{\mathbf{L}_\alpha|_{\mathrm{ran\,} \mathbf{\tilde{P}}_\alpha}} \rVert \lesssim _{I } 1, \quad \forall  \alpha \in I. 
  \] 
  Then the uniform bound follows from the adaption of the Gearhart-Pruss-Grenier theorem. 
\end{proof}  

\section{Nonlinear Stability}\label{sec:Nonlinear Stability}
Now we derive the nonlinear stability. The full nonlinear equation can be written as 
\begin{equation}\label{eq:7.Nonlinear equation}
  \begin{cases} \partial_s \mathbf{q}   = \mathbf{L}_\alpha \mathbf{q} + \mathbf{N}(\mathbf{q}) , \\
  \mathbf{q}(0,\cdot) = \mathbf{q}_0 \end{cases}  
\end{equation}
where 
\begin{equation}\label{eq:7.Definition of La}
  \mathbf{L}_\alpha := \begin{pmatrix} - y \partial_y & 1 \\ \partial_{yy}  & \frac{2 \alpha}{1+ \sqrt{ 1- \alpha}y} - 1 - y \partial_y   \end{pmatrix} , \quad \mathbf{N}(\mathbf{q}) := \begin{pmatrix} 0  \\ q_2^2 \end{pmatrix}. 
\end{equation}
By Duhamel's principle, we can write down the solution for \eqref{eq:7.Nonlinear equation} 
\begin{equation}\label{eq:7.Solution to Nonlinear equation}
  \mathbf{q}(s,\cdot) = \mathbf{S}_\alpha(s) \mathbf{q}(0,\cdot) + \int_{0}^s \mathbf{S}_\alpha(s - \tau) \mathbf{N}(\mathbf{q}(\tau, \cdot)) \mathrm{\,d} \tau
\end{equation}
where \(\mathbf{S}_\alpha\) is the semigroup generated by \(\mathbf{L}_\alpha\) per Proposition \ref{prop:5.Semigroup Sα}. 

\subsection{Stabilised nonlinear evolution}
\begin{definition}\label{def:7.��ᵏ}
  Fix \( k \in \mathbb{N}, 0 < w_0 < 1\). Define the Banach space \( ( \mathcal{X}^k(B) , \lVert\,\cdot\, \rVert_{\mathcal{X}^k(B)})\) by 
  \begin{align*}
      \mathcal{X}^k(B) &:= \{\mathbf{q} \in C(\mathbb{R}_+ , \mathcal{H}^k(B)) : \lVert  \mathbf{q} \rVert_{\mathcal{X}^k(B)} < \infty\}, \quad \text{where} \\ 
    \lVert  \mathbf{q} \rVert_{\mathcal{X}^k(B)} &:=  \sup_{s \in \mathbb{R}_+}\,  e^{w_0 s} \lVert  \mathbf{q}(s,\cdot) \rVert_{\mathcal{H}^k(B)} . 
  \end{align*} 
\end{definition}
We also denote the \(\varepsilon\)-sized balls around 0 in a Banach space \(X\) as 
\[  \mathrm{B}^{X} _\varepsilon := \left\{ q \in X : \lVert q \rVert_{X} \leq \varepsilon  \right\}. 
\] 
We first consider a modified problem following the Lyapunov-Perron's method from the dynamical systems theory to avoid encountering the unstable modes. To that end, set \( \mathbf{P}_\alpha  = \mathbf{P}_{0,\alpha} + \mathbf{P}_{1,\alpha}\) and denote the correction term 
\begin{equation}\label{eq:7.Correction Term}
  \mathbf{C}_{\alpha}(\mathbf{f},\mathbf{q}) = \mathbf{P}_\alpha \mathbf{f} + \mathbf{P}_{0,\alpha} \int_{0}^\infty \mathbf{N}(\mathbf{q}(\tau, \cdot))  + \mathbf{L}_\alpha \mathbf{P}_{0,\alpha} \int_{0}^\infty (- \tau) \mathbf{N}(\mathbf{q}(\tau, \cdot))  + \mathbf{P}_{1,\alpha} \int_{0}^\infty e^{- \tau} \mathbf{N}(\mathbf{q}(\tau, \cdot)) \mathrm{\,d} \tau. 
\end{equation}

We now derive some basic estimates for the nonlinear term 
\begin{lemma}[Estimates for Nonlinear Term]\label{lem:7.Estimates for Nonlinear Term}
  For all \( k \geq 2\) 
  \[  \lVert  \mathbf{N}(\mathbf{q}) \rVert_{\mathcal{H}^k} \lesssim  \lVert  \mathbf{q} \rVert_{\mathcal{H}^k}^2  
  \] 
  \[  \lVert  \mathbf{N}(\mathbf{q}) - \mathbf{N}(\mathbf{q}') \rVert_{\mathcal{H}^k} \lesssim  ( \lVert  \mathbf{q} \rVert_{\mathcal{H}^k} + \lVert  \mathbf{q}' \rVert_{\mathcal{H}^k})  \lVert  \mathbf{q} - \mathbf{q}' \rVert_{\mathcal{H}^k}. 
  \] 
\end{lemma}
\begin{proof} 
  Using the Sobolev Embedding, we have 
 \[   \lVert  \mathbf{N}(\mathbf{q}) \rVert_{\mathcal{H}^k} = \lVert  q_2^2 \rVert_{H^k}  \leq \lVert  q_2 \rVert_{L^\infty} \lVert  q_2 \rVert_{H^k} \lesssim  \lVert  q_2 \rVert_{H^k}^2 = \lVert \mathbf{q} \rVert_{\mathcal{H}^k}^2 . 
 \] 
 Similarly as above 
 \begin{align*}
     \lVert  \mathbf{N}(\mathbf{q}) - \mathbf{N}(\mathbf{q}') \rVert_{\mathcal{H}^k} &= \lVert  q_2^2 - {q_2'}^2 \rVert_{H^k}\lesssim  \sup_{t \in [0,1]}\,  \lVert  (tq _2 + (1 - t)q_2') (q_2 - q_2') \rVert_{{H}^k} \\
     &\lesssim  (\lVert  q_2 \rVert_{H^k} + \lVert q_2' \rVert_{H^k} ) \lVert  q_2 - q_2' \rVert_{H^k}  \lesssim  (\lVert  \mathbf{q} \rVert_{\mathcal{H}^k} + \lVert  \mathbf{q}' \rVert_{\mathcal{H}^k}) \lVert \mathbf{q}-\mathbf{q}' \rVert_{\mathcal{H}^k}. \qedhere
 \end{align*} 
\end{proof}  
We can now produce a solution to stabilized version of \eqref{eq:7.Nonlinear equation} using a fixed-point argument. 
\begin{proposition}\label{prop:7.Fixed-Point Solution for Corrected}
  Let \( \alpha _0  \in (0,1] , k \geq 5, 0 < w_0 < 1\). Then there exists constants \(\varepsilon  > \varepsilon' > 0 , \) such that \(\forall  \alpha \in B_{\varepsilon }(\alpha _0) \cap(0,1] \) and \(\forall  \mathbf{f} \in \mathrm{B}_{\varepsilon'}^{\mathcal{H}^k(B)}\) there exists a unique \(\mathbf{q} _\alpha \in \mathrm{B}^{\mathcal{X}^k(B)}_{\varepsilon }\) satisfying 
  \begin{equation}\label{eq:7.Solution to Corrected Nonlinear equation}
    \mathbf{q}_\alpha(s)  = \mathbf{S}_\alpha(s) (\mathbf{f} - \mathbf{C}_\alpha(\mathbf{f},\mathbf{q}_\alpha)) + \int_{0}^s \mathbf{S}_\alpha(s - \tau) \mathbf{N}(\mathbf{q}(\tau)) \mathrm{\,d} \tau . 
  \end{equation}
Moreover, the solution-to-data map 
  \begin{align*}
       \mathrm{B}_{\varepsilon'}^{\mathcal{H}^k(B)} &\longrightarrow \mathrm{B}^{\mathcal{X}^k(B)}_{\varepsilon } \\
              \mathbf{f} &\longmapsto     \mathbf{q}_\alpha
  \end{align*}
  is Lipschitz-continuous. 
\end{proposition}      
\begin{proof} 
 For \(\alpha _0 \in (0,1]\), let \( I := [ \alpha _0 - \varepsilon  , 1]\), for some \(\varepsilon \) to be chosen later, which is a compact subset of \( (0,1]\). Then as \( k \geq 5\) and \( 0 < w_0 < 1\) so \(\mathbf{S}_\alpha\) is well-defined and satisfies Proposition \ref{prop:6.Properties of Sα}. Denote the solution operator for the corrected problem by 
 \[  \mathbf{K}_\alpha(\mathbf{f}, \mathbf{q}) (s) := \mathbf{S}_\alpha(s) (\mathbf{f} - \mathbf{C}_\alpha(\mathbf{f},\mathbf{q})) + \int_{0}^s \mathbf{S}_{\alpha}(s - \tau) \mathbf{N}(\mathbf{q}(\tau)) \mathrm{\,d} \tau . 
 \] 
 We claim that there exists constants \(0 < \varepsilon ' < \varepsilon\) small enough such that 
 \[  \mathbf{K}_\alpha(\mathbf{f}, ) : \mathrm{B}^{\mathcal{X}^k} _{\varepsilon } \longrightarrow \mathrm{B}^{\mathcal{X}^k} _{\varepsilon } 
 \] 
 is a contraction for any fixed \( \mathrm{f} \in \mathrm{B}^{\mathcal{H}^k}_{\varepsilon'} \). Then uniqueness and existence for \eqref{eq:7.Solution to Corrected Nonlinear equation} follows from the Banach fixed-point theorem.

 To that end, fix \( \mathbf{f} \in \mathrm{B}^{\mathcal{H}^k}_{\varepsilon'}\) and \(\mathbf{q} \in \mathrm{B}^{\mathcal{X}_k}_{\varepsilon}\). Expanding out the definition for \( \mathbf{C}_\alpha(\mathbf{f}, \mathbf{q})\), we can write 
 \begin{align}\label{eq:7.Expanded form for Ka}
     \begin{split} 
         \mathbf{K}_\alpha(\mathbf{f},\mathbf{q})(s)& = \mathbf{S}_\alpha(s) \tilde{\mathbf{P}}_\alpha \mathbf{f} + \int_{0}^s \mathbf{S}_\alpha(s - \tau) \tilde{\mathbf{P}}_\alpha \mathbf{N}(\mathbf{q}(\tau)) \mathrm{\,d} \tau  - \mathbf{P}_{0,\alpha} \int_{s}^\infty \mathbf{N}(\mathbf{q}(\tau)) \mathrm{\,d} \tau  \\
      &\qquad - \mathbf{L}_\alpha \mathbf{P}_{0,\alpha} \int_{s}^\infty ( s - \tau) \mathbf{N}(\mathbf{q}(\tau)) \mathrm{\,d} \tau - \mathbf{P}_{1,\alpha} \int_{s}^{\infty} e^{ s - \tau} \mathbf{N}(\mathbf{q}(\tau)) \mathrm{\,d} \tau
     \end{split}
 \end{align} 
 where \(\tilde{\mathbf{P}}_\alpha =1 - \mathbf{P}_\alpha\) and we have used Proposition \ref{prop:6.Properties of Sα}. Using Minkowski inequality, Lemma \ref{lem:7.Estimates for Nonlinear Term} and \( \mathbf{L}_\alpha \mathbf{P}_{0,\alpha}\) is a linear finite operator and thus bounded, we have 
 \begin{align*}
   \sup_{\alpha \in I }\,   \lVert  \mathbf{K}_\alpha(\mathbf{f},\mathbf{q}) \rVert_{\mathcal{H}^k} &\lesssim  e^{- w_0 s} \lVert  \tilde{\mathbf{P}}_\alpha \mathbf{f} \rVert_{\mathcal{H}^k} + \int_{0}^s e^{-w_0 (s - \tau)} \lVert \mathbf{q}(\tau)   \rVert_{\mathcal{H}^k}^2 \mathrm{\,d} \tau + \int_{s}^\infty \lVert  \mathbf{q}(\tau) \rVert_{\mathcal{H}^k}^2 \mathrm{\,d} \tau \\
     &\qquad\qquad + \int_{s}^\infty ( \tau - s) \lVert  \mathbf{q}(\tau) \rVert_{\mathcal{H}^k}^2 \mathrm{\,d} \tau  + \int_{s}^\infty e^{s - \tau} \lVert  \mathbf{q}(\tau) \rVert_{\mathcal{H}^k}^2 \mathrm{\,d} \tau \\
     &\lesssim  e^{ - w_0 s} \varepsilon' + \varepsilon ^2  \int_{0}^s e^{w_0 ( \tau - s)} e^{ -2 w_0 \tau } + \varepsilon ^2 \int_{s}^{\infty}e^{ - 2 w_0 \tau} \mathrm{\,d} \tau \\
     & \qquad\qquad + \varepsilon ^2 \int_{s} ^\infty (\tau - s)e^{ -2 w_0 \tau} \mathrm{\,d} \tau + \varepsilon ^2 \int_{s}^\infty e^{ s - \tau} e^{ - 2w_0 \tau} \mathrm{\,d} \tau \\
     &\lesssim  \varepsilon' e^{ - w_0 s} + \varepsilon ^ 2 \frac{e^{ -w_0 s}}{w_0} \left( 1 -   \frac{1}{2}e^{ - w_0 s } + \frac{e^{ -w_0 s }}{ 4 w_0 }  +\frac{e^{ -w_0 s}}{w_0 ^{-1} + 2 }\right) \\
     &\lesssim _{w_0} e^{ -w_0 s } (\varepsilon ' + \varepsilon ^2) \leq C e^{ -w_0 s} (\varepsilon' + \varepsilon ^2). 
  \end{align*} 
  Fixing \(\varepsilon' = \varepsilon ^2  \leq C^{-2}\) gives us \(\sup_{\alpha \in I }\lVert  \mathbf{K}_\alpha(\mathbf{f},\mathbf{q}) \rVert_{\mathcal{X}^k} \leq \varepsilon\).  Thus \(\mathbf{K}_\alpha(\mathbf{f}, -) : \mathrm{B}^{\mathcal{X}^k}_\varepsilon \to \mathrm{B}^{\mathcal{X}^k}_\varepsilon\) is well-defined for any \(\mathbf{f} \in \mathrm{B}^{\mathcal{H}^k}_{\varepsilon'}\). 

  We now show that \( \mathbf{K}(\mathbf{f}, \cdot)\) is a contraction on the same space. Using Proposition \ref{prop:6.Properties of Sα} and Lemma \ref{lem:7.Estimates for Nonlinear Term}, we have for \(f \in \mathrm{B}^{\mathcal{H}^k}_{\varepsilon'}\) and \( \mathbf{q}, \mathbf{q}' \in \mathrm{B}^{\mathcal{X}^k}_{\varepsilon}:\)
  \begin{align*}
      \lVert  \mathbf{K}_\alpha(\mathbf{f} , \mathbf{q}) &-  \mathbf{K}_\alpha(\mathbf{f} , \mathbf{q}') \rVert_{\mathcal{H}^k} \lesssim   \int_{0}^s e^{-w_0(s - \tau)} \lVert (\mathbf{q}, \mathbf{q}')(\tau) \rVert_{\mathcal{H}^k} \lVert  (\mathbf{q} - \mathbf{q}')(\tau) \rVert_{\mathcal{H}^k}\\
      & + \int_{s}^{\infty} \lVert (\mathbf{q}, \mathbf{q}')(\tau) \rVert_{\mathcal{H}^k} \lVert  (\mathbf{q} - \mathbf{q}')(\tau) \rVert_{\mathcal{H}^k} + \int_{s}^\infty (\tau - s)  \lVert (\mathbf{q}, \mathbf{q}')(\tau) \rVert_{\mathcal{H}^k} \lVert  (\mathbf{q} - \mathbf{q}')(\tau) \rVert_{\mathcal{H}^k} \\
      &   + \int_{s}^\infty e^{ s - \tau}  \lVert (\mathbf{q}, \mathbf{q}')(\tau) \rVert_{\mathcal{H}^k} \lVert  (\mathbf{q} - \mathbf{q}')(\tau) \rVert_{\mathcal{H}^k} \\
      &\lesssim \varepsilon \lVert  \mathbf{q} - \mathbf{q}' \rVert_{\mathcal{X}^k} \left(\int_{0}^s e^{ - w_0 (s - \tau)} e^{- 2w_0 \tau }  +  \int_{s}^\infty e^{ - 2w_0 \tau} +  ( \tau - s)e^{ - 2w_0 \tau}  +  e^{ s - \tau} e^{ -2 w_0 \tau} \mathrm{\,d} \tau \right) \\
      &\leq  C' \varepsilon e^{ - w_0 s} \lVert  \mathbf{q} - \mathbf{q}'   \rVert_{\mathcal{X}^k}. 
  \end{align*} 
  Fixing \( \varepsilon \leq \frac{1}{2}  \min \{C^{-1} , {C'}^{-1}\}\) we obtain 
  \[  \lVert  \mathbf{K}_\alpha(\mathbf{f} ,\mathbf{q} ) - \mathbf{K}_\alpha(\mathbf{f},\mathbf{q}') \rVert_{\mathcal{X}^k} \leq  \frac{1}{2} \lVert \mathbf{q} - \mathbf{q}' \rVert_{\mathcal{X}^k} , \quad \forall  \mathbf{f} \in \mathrm{B}^{\mathcal{H}^k}_{\varepsilon'} , \forall  \mathbf{q},\mathbf{q}' \in \mathrm{B}^{\mathcal{X}^k}_{\varepsilon}, \forall  \alpha \in I . 
  \]

  Finally, we will show the continuous dependence on the initial data. For \(\mathbf{f} ,\mathbf{f} \in \mathrm{B}^{\mathcal{H}^k}_{\varepsilon'}\) and \(\mathbf{q},\mathbf{q} ' \in \mathrm{B}^{\mathcal{X}^k}_\varepsilon\), the fixed point solution obtained via Banach fixed-point satisfies 
  \begin{align*}
      \mathbf{q}(s)  - \mathbf{q} '(s)  &=  \mathbf{K} _\alpha(\mathbf{f},\mathbf{q} )(s) - \mathbf{K}_\alpha(\mathbf{f}', \mathbf{q} ')(s) = \mathbf{S}_\alpha \tilde{\mathbf{P}}_\alpha ( \mathbf{f} - \mathbf{f}')    + \mathbf{K}_\alpha(\mathbf{f}', \mathbf{q}) - \mathbf{K}_\alpha(\mathbf{f}', \mathbf{q}'). 
  \end{align*} 
  Using Proposition \ref{prop:6.Properties of Sα} and the contraction property of \(\mathbf{K}_\alpha\), we obtain 
  \[  \lVert  \mathbf{q}   - \mathbf{q}'  \rVert_{\mathcal{X}^k}  \leq  C'' \lVert  \mathbf{f} - \mathbf{f}' \rVert_{\mathcal{H}^k}  +  \frac{1}{2} \lVert  \mathbf{q} - \mathbf{q}' \rVert_{\mathcal{X}^k}. 
  \] 
  In particular, we have 
  \[  \lVert  \mathbf{q} - \mathbf{q}' \rVert_{\mathcal{X}^k}  \lesssim \lVert  \mathbf{f} - \mathbf{f}' \rVert_{\mathcal{H}^k} 
  \] 
  which implies the Lipschitz dependence of \(\mathbf{q}\) on the initial data \(\mathbf{f}\). 
\end{proof}  

\subsection{Stable flow near the blow-up solution}
We will now carefully construct the initial data \(\mathbf{f}\) for our Cauchy problem so as to ensure that the correction term \(\mathbf{C}_\alpha(\mathbf{f},\mathbf{q})\) disappears. To that end, we define the following decomposition: 
\begin{definition}[Decomposition Map]\label{def:7.Decomposition Map}
  Let \((\alpha _0 , \alpha) \in (0,1]^2 , (T_0 , T) \in \mathbb{R}^2_+  \) satisfying \( \sqrt{ 1 - \alpha _0} T < T_0\), \(\kappa _0 , \kappa \in \mathbb{R}\) and \(k \geq 5\). Define the decomposition operator 
  \begin{align*}
      \mathbf{U}_{\alpha,\kappa , T} \colon \mathcal{H}^k(-1,1) &\longrightarrow \mathcal{H}^k(-1,1)\\
              \mathbf{f} &\longmapsto     \mathbf{f}^ T + \mathbf{f}^T_0 - \mathbf{f}_{\alpha,\kappa}.  
  \end{align*}
  where if we write \( \mathbf{f} (y) = (f_1 (y) , f_2(y))\) then 
  \begin{align*}
      \mathbf{f}^T(y) := \begin{pmatrix} f_1(Ty) \\ T f_2(Ty) \end{pmatrix} , \, \mathbf{f}^T_0 (y) := \begin{pmatrix} \tilde{U}_{\alpha _0,\infty}(\frac{T}{T_0}y) + \kappa _0  \\ \frac{T}{T_0}\alpha _0 + (\frac{T}{T_0})^2 y \partial_y \tilde{U}_{\alpha _0 , \infty } (\frac{T}{T_0}y)   \end{pmatrix}, \, \mathbf{f}_{\alpha,\kappa} := \begin{pmatrix}\tilde{U}_{\alpha,\infty }(y) + \kappa \\ \alpha + y \partial_y \tilde{U}_{\alpha,\infty}(y)   \end{pmatrix}
  \end{align*} 
\end{definition}

\begin{lemma}[Taylor Expansion]\label{lem:7.Taylor Expansion}
  Let \( \alpha _0 \in (0,1],  T,T_0>0\) and \( \kappa _0 , \kappa \in \mathbb{R}\) and \( k \geq 5\). Then 
  \[  \mathbf{U}_{\alpha,\kappa,T} (\mathbf{f}) = \mathbf{f}^T +  \left[ (\kappa _0 - \kappa) - \alpha\left(\frac{T}{T_0} - 1\right)\right] \mathbf{f}_{0,\alpha} + \left(\frac{T}{T_0} - 1\right) \mathbf{f}_{1,\alpha} +\frac{\alpha _0 -\alpha}{\sqrt{ 1- \alpha}} \mathbf{g}_{0,\alpha} + \mathbf{r}\left(\alpha, \frac{T}{T_0}\right) 
  \] 
  for any fixed \(\mathbf{f} \in \mathcal{H}^k(-1,1)\) where \( \mathbf{f}_{0,\alpha} , \mathbf{f}_{1,\alpha} \) and \(\mathbf{g}_{0,\alpha}\) are introduced in Proposition \ref{prop:5.Spectrum of Lα} and 
  \[   \left\lVert  \mathbf{r}\left(\alpha, \frac{T}{T_0}\right) \right\rVert_{\mathcal{H}^k} \lesssim_{\alpha_0}   \lvert  \alpha -\alpha _0 \rvert^2 +  \left\lvert \frac{T}{T_0} - 1 \right\rvert^2
  \] 
for all {  \(\alpha \in {\mathrm{B}}_{\varepsilon _1}(\alpha _0)\) and \( T \in \mathrm{B}_{T_0 \varepsilon _1}(T_0)\)} and for any \(\varepsilon _1 \in (0, \min\{\frac{\alpha _0}{2},\frac{1-\alpha_0}{2}\})\) if $\alpha_0 \neq 1$; and
\[  \mathbf{U}_{\alpha,\kappa,T} (\mathbf{f}) = \mathbf{f}^T +  \left[ (\kappa _0 - \kappa) - \alpha\left(\frac{T}{T_0} - 1\right)\right] \mathbf{f}_{0,\alpha} + \left(\frac{T}{T_0} - 1\right) \mathbf{f}_{1,\alpha} + 2\sqrt{ 1- \alpha}  \mathbf{g}_{0,\alpha} + \mathbf{r}\left(\alpha, \frac{T}{T_0}\right) 
  \]
for any fixed \(\mathbf{f} \in \mathcal{H}^k(-1,1)\) where
 \[   \left\lVert  \mathbf{r}\left(\alpha, \frac{T}{T_0}\right) \right\rVert_{\mathcal{H}^k} \lesssim   \lvert  \alpha -1 \rvert +  \left\lvert \frac{T}{T_0} - 1 \right\rvert^2
  \] 
for all {  \(\alpha \in [1-\varepsilon_1,1]\) and \( T \in \mathrm{B}_{T_0 \varepsilon _1}(T_0)\)} and for any \(\varepsilon _1 \in (0, \frac{1}{2})\) if $\alpha_0 = 1$.
\end{lemma} 
\begin{proof} 
  We first treat the case $\alpha_0 \neq 1.$ Recall 
  \[  \tilde{U}_{\alpha,\infty}(y) = - \alpha \ln  \left( 1 + \sqrt{ 1- \alpha} y\right). 
  \] 
  Now fix \( y \in [-1,1], (\alpha _0 , T_0)\in (0,1)\times \mathbb{R}_{+}, \kappa _0 \in \mathbb{R}\) and set \(\varepsilon\) as in Proposition \ref{prop:7.Fixed-Point Solution for Corrected}. Then consider the map 
  \begin{align*}
       \mathrm{B}_\varepsilon(\alpha _0 ) \times \mathbb{R} \times \mathrm{B}_{T_0 \varepsilon}(T_0) &\longrightarrow \mathbb{R}^2 \\
              (\alpha,\kappa,T) &\longmapsto     \mathbf{f}_0^T - \mathbf{f}_{\alpha,\kappa}.  
  \end{align*}
  Applying Taylor's theorem to the first component of this map and writing \(\mathbf{U}_{\alpha,\kappa, T} (\mathbf{f}) - \mathbf{f}^T = \mathbf{f}^T_0  -\mathbf{f}_{\alpha,\kappa}\), we have 
  \begin{align*}
     \left( \mathbf{f}^T_0  -\mathbf{f}_{\alpha,\kappa} \right)_1 &=
     \tilde{U}_{\alpha _0 ,\infty} (\tfrac{T}{T_0} y) -\tilde{U}_{\alpha ,\infty}(y) +  (\kappa _0 - \kappa)\\
     &= \begin{pmatrix}( \frac{T}{T_0} - 1 )y  \\  \alpha _0 - \alpha  \end{pmatrix} \cdot \nabla_{y, \alpha} \tilde{U}_{\alpha,\infty }(y)  +  (\kappa _0 - \kappa)   +\left(\mathbf{r}\left(\alpha,\tfrac{T}{T_0 }\right)\right)_1\\
     &=\begin{pmatrix} \frac{T}{T_0} - 1   \\  \alpha _0 - \alpha  \\  \kappa _0 - \kappa  \end{pmatrix} \cdot \begin{pmatrix} -\alpha + \frac{\alpha}{1 + \sqrt{ 1- \alpha }y}\\ \frac{1}{\sqrt{ 1- \alpha}} (\mathbf{g}_{0,\alpha})_1 \\ 1 \end{pmatrix} +\left(\mathbf{r}\left(\alpha,\tfrac{T}{T_0 }\right)\right)_1\\
     &= \begin{pmatrix} \frac{T}{T_0} - 1 \\  \frac{\alpha _0 - \alpha }{\sqrt{1 - \alpha}} \\ \kappa _0 - \kappa  - \alpha (\frac{T}{T_0 } - 1) \end{pmatrix} \cdot \begin{pmatrix} (\mathbf{f}_{1,\alpha} )_1\\ (\mathbf{g}_{0,\alpha})_1 \\ (\mathbf{f}_{0,\alpha})_1 \end{pmatrix} +\left(\mathbf{r}\left(\alpha,\tfrac{T}{T_0 }\right)\right)_1.
  \end{align*} 
  For the remainder, we have 
  \begin{align*}
       \lvert  (\mathbf{r}(\alpha, \tfrac{T}{T_{0}}))_1(y)  \rvert &\lesssim  (\tfrac{T}{T_0 } -1)^2 \int_{0}^1  \lvert y^2 \partial_{yy} \tilde{U}_{\alpha, \infty}((t \tfrac{T}{T_0} + (1 - t))y)  \rvert\mathrm{\,d} t \\
       &\qquad\qquad  + (\alpha _0 - \alpha)^2 \int_{0}^1  \lvert \partial_{\alpha \alpha } \tilde{U}_{   (1 - t) \alpha  + t \alpha _0 , \infty}(y) \rvert \mathrm{\,d} t  
  \end{align*} 
  Since \( \alpha_0 \in (0,1)\) so \((\alpha, y) \mapsto \tilde{U}_{\alpha,\infty}(y)\) is smooth in \( (0,1) \times  [ -1,1]\) and thus the integrals are uniformly bounded for \( \alpha  \in \mathrm{B}_{\frac{\alpha _0}{2}}(\alpha _0)\) and \( T \in \mathrm{B}_{T_0 }(T_0)\). Running the same argument for the second component completes the proof. 

  For $\alpha_0=1$, $\tilde{U}_{1,\infty} =0 $, we have
    \begin{align*}
     \left( \mathbf{f}^T_0  -\mathbf{f}_{\alpha,\kappa} \right)_1 &=
      -\tilde{U}_{\alpha ,\infty}(y) +  (\kappa _0 - \kappa)\\
     &= \alpha\ln(1+\sqrt{1-\alpha} y)  +  (\kappa _0 - \kappa) \\
     &= \begin{pmatrix} \frac{T}{T_0} - 1 \\  2\sqrt{1 - \alpha} \\ \kappa _0 - \kappa  - \alpha (\frac{T}{T_0 } - 1) \end{pmatrix} \cdot \begin{pmatrix} (\mathbf{f}_{1,\alpha} )_1\\ (\mathbf{g}_{0,\alpha})_1 \\ (\mathbf{f}_{0,\alpha})_1 \end{pmatrix} +\left(\mathbf{r}\left(\alpha,\tfrac{T}{T_0 }\right)\right)_1.
  \end{align*} 
  where 
 \begin{align*}
     &\left(\mathbf{r}\left(\alpha,\tfrac{T}{T_0 }\right)\right)_1 \\
     &= \alpha\ln(1+\sqrt{1-\alpha} y) - \left(\frac{T}{T_0} - 1\right) \frac{\alpha}{1+\sqrt{1-\alpha}y} - \frac{\alpha\sqrt{1-\alpha}y}{1+\sqrt{1-\alpha}y} + \alpha (\frac{T}{T_0 } - 1) +O(|1-\alpha|)\\
     & =\alpha \left(\ln(1+\sqrt{1-\alpha} y) - \frac{\sqrt{1-\alpha}y}{1+\sqrt{1-\alpha}y}\right) + \left(\frac{T}{T_0} - 1\right) \frac{\alpha\sqrt{1-\alpha}y}{1+\sqrt{1-\alpha}y} + O(|1-\alpha|)\\
     & =  O(|1-\alpha| + |\frac{T}{T_0} - 1|^2 ).
 \end{align*} 
 The second component can be estimated similarly.
\end{proof}  

\begin{proposition}\label{prop:7.Fixed Point for Uncorrected}
  Let \( \alpha_0 \in (0,1],  w_0 \in (0,1), T_0 \in \mathbb{R}_+ , \kappa _0 \in \mathbb{R} \) and \( k \geq 5\). There exists constant \( 0<\varepsilon ' < \varepsilon\) such that \( \forall  \mathbf{f} \in \mathrm{B}^{\mathcal{H}^k}_{\varepsilon'}\) there exists \( \alpha ^\star  \in (0,1] , T^\star \in \mathbb{R}_+ , \kappa ^\star \in \mathbb{R}\) and a unique \( \mathbf{q}_{\alpha ^\star , T^\star , \kappa ^\star} \in C(\mathbb{R}_+ ; \mathrm{B}^{\mathcal{H}^k}_\varepsilon)\) such that it satisfies 
  \begin{equation}\label{eq:7.Duhamel for q*}
    \mathbf{q}_{\alpha ^\star, T^\star , \kappa ^\star}(s) := \mathbf{S}_{\alpha ^\star}(s) \mathbf{U}_{\alpha ^\star , T^\star , \kappa ^\star}(\mathbf{f}) + \int_{0}^s \mathbf{S}_{\alpha ^\star}(s - \tau) \mathbf{N}(\mathbf{q}_{a^\star , T^\star , \kappa ^\star}) \mathrm{\,d} \tau  
  \end{equation}
  and the exponential decay estimate 
  \[  \lVert   \mathbf{q}_{\alpha ^\star, T^\star , \kappa ^\star}(s) \rVert_{\mathcal{H}^k}  \leq  \varepsilon e^{ - w_0 s} 
  \] 
  where \( (\alpha ^\star , T^\star , \kappa ^\star)\) are close to the initial parameters \( (\alpha_0 , T_0, \kappa_0):\)
  \[   \lvert \alpha ^\star - \alpha _0 \rvert  +  \lvert  \tfrac{T^\star}{T_0} - 1 \rvert +  \lvert \kappa ^\star - \kappa _0 \rvert \lesssim  \varepsilon'. 
  \] 
\end{proposition}
\begin{proof} 
  Fix \(\alpha _0 , w_0 \in (0,1)\) and \( k \geq 5\). Fix \(\varepsilon ' = \varepsilon ^2\) from Proposition \ref{prop:7.Fixed-Point Solution for Corrected} and then Lemma \ref{lem:7.Taylor Expansion} implies for \(\lVert \mathbf{f} \rVert_{\mathcal{H}^k} \leq  \varepsilon' : \)
  \begin{align*}
      \lVert  \mathbf{U}_{\alpha, \kappa, T} (\mathbf{f}) \rVert_{\mathcal{H}^k} &\lesssim  \lVert  \mathbf{f}^T \rVert_{\mathcal{H}^k} + ( \lvert \kappa _0 - \kappa \rvert + \alpha  \lvert \tfrac{T}{T_0} - 1 \rvert) \lVert  \mathbf{f}_{0,\alpha} \rVert_{\mathcal{H}^k} +  \lvert \tfrac{T}{T_0} - 1 \rvert \lVert  \mathbf{f}_{1,\alpha} \rVert_{\mathcal{H}^k} \\
      & \qquad\qquad +  \lvert \alpha _0 - \alpha \rvert^{\frac12} \lVert \mathbf{g}_{0,\alpha} \rVert_{\mathcal{H}^k} +  \lvert \alpha  - \alpha _0 \rvert +  \lvert \tfrac{T}{T_0} - 1 \rvert^2 \lesssim    \varepsilon '
  \end{align*} 
  provided \(\kappa \in \mathrm{B}_{\varepsilon _1} (\kappa _0), \alpha \in \mathrm{B}_{\varepsilon _1} (\alpha _0), \frac{T}{T_0} \in \mathrm{B}_{\varepsilon _1}(1)\) for some small enough \(\varepsilon _1 < \varepsilon'\). In particular, for \(\mathbf{f} \in \mathrm{B}^{\mathcal{H}^k}_{\varepsilon'}\) and \( \alpha \in \mathrm{B}_{\varepsilon _1}(\alpha _0) , \kappa \in \mathrm{B}_{\varepsilon _1}(\kappa) , \frac{T}{T_0} \in \mathrm{B}_{\varepsilon _1}(1)\), we have \( \mathbf{U}_{\alpha, \kappa, T} (\mathbf{f}) \in \mathrm{B}^{\mathcal{H}^k}_{\varepsilon'}\). Taking \( \mathbf{U}_{\alpha, \kappa, T}(\mathbf{f})\) as the initial data for Proposition \ref{prop:7.Fixed-Point Solution for Corrected}, we obtain the existence of a unique \( \mathbf{q}_{\alpha, \kappa, T} \in \mathrm{B}^{\mathcal{X}^k}_\varepsilon\) such that it satisfies 
  \[  \mathbf{q}_{\alpha, \kappa, T}(s) = \mathbf{S}_{\alpha}(s) (\mathbf{U}_{\alpha, \kappa, T} - \mathbf{C}_\alpha(\mathbf{U}_{\alpha, \kappa, T}(\mathbf{f}), \mathbf{q}_{\alpha, \kappa, T}))  + \int_{0}^s \mathbf{S}_{\alpha} (s - \tau) \mathbf{N}_{\alpha}(\mathbf{q}_{\alpha, \kappa, T})(\tau) \mathrm{\,d} \tau . 
  \] 

  We will now determine parameters \(\alpha, \kappa, T\) such that the correction term \( \mathbf{C}_\alpha(\mathbf{U}_{\alpha, \kappa, T}(\mathbf{f}), \mathbf{q}_{\alpha, \kappa, T}))\) vanishes. Recall that the formula for \(  \mathbf{C}_\alpha(\mathbf{U}_{\alpha, \kappa, T}(\mathbf{f}), \mathbf{q}_{\alpha, \kappa, T}))\) implies that \(  \mathbf{C}_\alpha(\mathbf{U}_{\alpha, \kappa, T}(\mathbf{f}), \mathbf{q}_{\alpha, \kappa, T})) \in \mathrm{ran\,} \mathbf{P}_{0,\alpha} \oplus \mathrm{ran\,}\mathbf{P}_{1,\alpha}\). Per Proposition \ref{prop:5.Spectrum of Lα}, \(\mathrm{ran\,} \mathbf{P}_{0,\alpha} \oplus \mathrm{ran\,}\mathbf{P}_{1,\alpha}\) is spanned by the linearly independent set \( \{\mathbf{f}_{0,\alpha} , \mathbf{f}_{1,\alpha} , \mathbf{g}_{0,\alpha}\}\). Exploiting the duality of the Hilbert space, it is sufficient to show that there exists parameters \( \alpha, \kappa , T\) such that 
  \[  \langle  \mathbf{C}_\alpha(\mathbf{U}_{\alpha, \kappa, T}(\mathbf{f}), \mathbf{q}_{\alpha, \kappa, T})), \mathbf{g} \rangle _{\mathcal{H}^k} = 0 , \quad \forall  \mathbf{g} \in \{\mathbf{f}_{0,\alpha} , \mathbf{f}_{1,\alpha} , \mathbf{g}_{0,\alpha}\}. 
  \] 
  Expanding out this inner-product using \eqref{eq:7.Correction Term} and Lemma \ref{lem:7.Taylor Expansion}, 
  \begin{align*}
      & \langle  \mathbf{C}_\alpha(\mathbf{U}_{\alpha, \kappa, T}(\mathbf{f}), \mathbf{q}_{\alpha, \kappa, T})), \mathbf{g} \rangle _{\mathcal{H}^k} = \langle  \mathbf{P}_\alpha \mathbf{f}^T , \mathbf{g} \rangle_{\mathcal{H}^k} + ( \kappa _0 - \kappa - \alpha (\tfrac{T}{T_0} - 1)) \langle  \mathbf{f}_{0,\alpha} , \mathbf{g} \rangle_{\mathcal{H}^k} \\
      & \qquad\qquad+ (\tfrac{T}{T_0} - 1) \langle  \mathbf{f}_{1,\alpha} , \mathbf{g} \rangle _{\mathcal{H}^k} + \tfrac{\alpha _0 - \alpha}{\sqrt{ 1- \alpha }} \langle  \mathbf{g}_{0,\alpha} , \mathbf{g} \rangle_{\mathcal{H}^k} + \langle  \mathbf{P}_\alpha \mathbf{r}(\alpha, \tfrac{T}{T_0}) , \mathbf{g} \rangle_{\mathcal{H}^k} \\
      &\qquad\qquad + \langle  \mathbf{P}_{0,\alpha} \int_{0}^\infty \mathbf{N}(\mathbf{q}_{\alpha, \kappa, T}) ( \tau) \mathrm{\,d} \tau, \mathbf{g} \rangle_{\mathcal{H}^k} + \langle  \mathbf{L}_\alpha \mathbf{P}_{0,\alpha} \int_{0}^\infty (- \tau) \mathbf{N}(\mathbf{q}_{\alpha,\kappa, T}) (\tau) \mathrm{\,d} \tau ,\mathbf{g} \rangle_{\mathcal{H}^k}\\
      & \qquad\qquad + \langle  \mathbf{P}_{1,\alpha} \int_{0}^\infty e^{ - \tau} \mathbf{N}(\mathbf{q}_{\alpha, \kappa, T})(\tau) \mathrm{\,d} \tau , \mathbf{g} \rangle_{\mathcal{H}^k}. 
  \end{align*} 
  To obtain the vanishing, we will now consider the dual basis \(\{\mathbf{g}_{\alpha}^1 ,\mathbf{g}_{\alpha}^2, \mathbf{g}_\alpha ^3\}\) of \( \{ \mathbf{f}_{0,\alpha} , \mathbf{f}_{1,\alpha} , \mathbf{g}_{0,\alpha}\}\). This dual basis is smoothly dependent on the parameter \(\alpha\) as \( \{ \mathbf{f}_{0,\alpha} , \mathbf{f}_{1,\alpha} , \mathbf{g}_{0,\alpha}\}\) are smooth functions of \(\alpha\) and the dual basis is obtained by inverting a change-of-basis matrix. Moreover, since \( \mathbf{g}^{i}_\alpha\) is dual for \( i = 1,2,3\) it satisfies the following relations: 
  \[  \langle  \mathbf{f}_{0,\alpha} , \mathbf{g}^i_\alpha \rangle  _{\mathcal{H}^k}  = \delta ^i_1 , \quad  \langle  \mathbf{f}_{1,\alpha} , \mathbf{g}^i_\alpha \rangle  _{\mathcal{H}^k}= \delta ^i_2,  \quad  \langle  \mathbf{g}_{0,\alpha} , \mathbf{g}^i_\alpha \rangle  _{\mathcal{H}^k}= \delta ^i_3, \quad \forall  i = 1,2,3.
  \]  
 Define the following functionals: 
 \begin{align*}
     F_n(\alpha, \kappa, T) &:= \langle  \mathbf{P}_\alpha \mathbf{f}^T, \mathbf{g} ^n_\alpha\rangle_{\mathcal{H}^k} + \langle \mathbf{P}_\alpha \mathbf{r}(\alpha, \tfrac{T}{T_0} ) , \mathbf{g} ^n_\alpha\rangle_{\mathcal{H}^k}  + \langle  \mathbf{P}_{0,\alpha} \int_{0}^\infty \mathbf{N}(\mathbf{q}_{\alpha,\kappa, T})(\tau) \mathrm{\,d} \tau , \mathbf{g}  ^n_\alpha\rangle_{\mathcal{H}^k} \\
     & + \langle  \mathbf{L}_\alpha \mathbf{P}_{0,\alpha} \int_{0}^\infty ( - \tau) \mathbf{N}(\mathbf{q}_{\alpha,\kappa, T})(\tau) \mathrm{\,d} \tau , \mathbf{g} ^n_\alpha \rangle_{\mathcal{H}^k}  + \langle  \mathbf{P}_{1,\alpha} \int_{0}^\infty e^ { - \tau} \mathbf{N}(\mathbf{q}_{\alpha ,\kappa , T})(\tau ) \mathrm{\,d} \tau , \mathbf{g}  ^n_\alpha\rangle_{\mathcal{H}^k}
 \end{align*} 
 for \( n = 1,2,3\). 
 Then Lemma \ref{lem:7.Estimates for Nonlinear Term} and Lemma \ref{lem:7.Taylor Expansion} implies that the inner-producted terms are small: 
 \begin{align*}
    \lvert \langle  \mathbf{C}_\alpha(\mathbf{U}_{\alpha, \kappa, T}(\mathbf{f}), \mathbf{q}_{\alpha, \kappa, T}), \mathbf{g} \rangle _{\mathcal{H}^k}  \rvert &\lesssim  \lVert  \mathbf{f}^T \rVert_{\mathcal{H}^k} + \lVert  \mathbf{r}(\alpha, \tfrac{T}{T_0}) \rVert_{\mathcal{H}^k} + \lVert  \mathbf{q}_{\alpha,\kappa, T} \rVert_{\mathcal{X}^k}^2 \\
    &  \lesssim  \varepsilon' +  \lvert \alpha - \alpha _0 \rvert^2 +  \lvert \tfrac{T}{T_0} - 1 \rvert^2 + \varepsilon ^2  \lesssim   \varepsilon ^2,  
 \end{align*} 
 for all \(  \mathbf{g} \in \mathcal{H}^k. \) Thus for \(\alpha _0 \in (0,1)\)
 \[ (\kappa _0 + F_1 +\alpha F_2, T_0 - T_0 F _2 ,\alpha _0 + \sqrt{1 - \alpha} F_3)  : \mathrm{B}_\varepsilon(\kappa _0)  \times  \mathrm{B}_{T_0 \varepsilon}(T_0) \times  \mathrm{B}_{\varepsilon}(\alpha _0) \longrightarrow  \mathrm{B}_\varepsilon(\kappa _0)  \times  \mathrm{B}_{T_0 \varepsilon}(T_0) \times  \mathrm{B}_{\varepsilon}(\alpha _0)
 \] 
 and for \( \alpha _0 = 1\)
 \[   (\kappa _0 + F_1 +\alpha F_2, T_0 - T_0 F _2 ,1 - \tfrac{1}{4} F^2_3)  : \mathrm{B}_\varepsilon(\kappa _0)  \times  \mathrm{B}_{T_0 \varepsilon}(T_0) \times  (1- \varepsilon, 1]\longrightarrow  \mathrm{B}_\varepsilon(\kappa _0)  \times  \mathrm{B}_{T_0 \varepsilon}(T_0) \times  (1- \varepsilon, 1]
 \] 
 provided \(\varepsilon\) is small enough. Since \( F_n\) are continuous, we have that there exists a fixed point for the map \( (F_1, F_2 ,F_3)\). In particular, there exists parameters \(\alpha _\star, T_\star , \kappa _\star \in \mathrm{B}_{\varepsilon}(\alpha _0) \times  \mathrm{B}_{T_0 \varepsilon}(T_0) \times  \mathrm{B}_\varepsilon(\kappa _0)\) such that  
\begin{align*}
   0 &= \begin{cases} \frac{\alpha _0 - \alpha _\star }{\sqrt{1 - \alpha _{\star}}}+ F_3(\alpha _\star , \kappa _\star , T_\star) = \langle  \mathbf{C}_{\alpha _\star}(\mathbf{U}_{\alpha_\star, \kappa_\star, T_\star}(\mathbf{f}), \mathbf{q}_{\alpha_\star, \kappa_\star, T_\star}), \mathbf{g}^3_{\alpha _{\star}}\rangle _{\mathcal{H}^k} & \alpha _0 \in (0,1) , \\
  2 \sqrt{ 1- \alpha _{\star}}+ F_3(\alpha _\star , \kappa _\star , T_\star) = \langle  \mathbf{C}_{\alpha _\star}(\mathbf{U}_{\alpha_\star, \kappa_\star, T_\star}(\mathbf{f}), \mathbf{q}_{\alpha_\star, \kappa_\star, T_\star}), \mathbf{g}^3_{\alpha _{\star}}\rangle _{\mathcal{H}^k} & \alpha _0 =1,  \end{cases}   \\
   0 &= \tfrac{T_\star}{T_0}  -1+ F_2(\alpha _\star, \kappa _\star , T_\star) = \langle  \mathbf{C}_{\alpha _\star}(\mathbf{U}_{\alpha_\star, \kappa_\star, T_\star}(\mathbf{f}), \mathbf{q}_{\alpha_\star, \kappa_\star, T_\star}),\mathbf{g}^2_{\alpha _{\star}}\rangle _{\mathcal{H}^k} , \\
   0 &= \kappa _0 - \kappa _\star - \alpha (\tfrac{T_\star}{T_0} - 1) + F_1(\alpha _\star, \kappa _\star , T_\star) =  \langle  \mathbf{C}_{\alpha _\star}(\mathbf{U}_{\alpha_\star, \kappa_\star, T_\star}(\mathbf{f}), \mathbf{q}_{\alpha_\star, \kappa_\star, T_\star}), \mathbf{g}^1_{\alpha _{\star}}\rangle _{\mathcal{H}^k} . 
\end{align*} 
As \(\{\mathbf{g}^n_{\alpha ^{\star}}\}_{n = 1,2,3}\) span \(\mathrm{ran\,} \mathbf{P}_{0,\alpha} \oplus \mathrm{ran\,} \mathbf{P}_{1,\alpha}\) so \(\mathbf{C}_{\alpha _\star}(\mathbf{U}_{\alpha_\star, \kappa_\star, T_\star}(\mathbf{f}),\mathbf{q}_{\alpha_\star, \kappa_\star, T_\star})\) vanishes on it and since \(\mathbf{C}_{\alpha _\star}(\mathbf{U}_{\alpha_\star, \kappa_\star, T_\star}(\mathbf{f}),\mathbf{q}_{\alpha_\star, \kappa_\star, T_\star}) \in \mathrm{ran\,} \mathbf{P}_{0,\alpha} \oplus \mathrm{ran\,} \mathbf{P}_{1,\alpha}\) so it is identically zero. 
\end{proof}  

\subsection{Proof of stability of the generalized self-similar blow-up.}
\begin{proof} [Proof of Theorem \ref{thm:main-1}]  
Fix \(\delta \in (0,1)\) and let \( w_0 := 1 - \delta\). From Proposition \ref{prop:7.Fixed Point for Uncorrected}, we pick \(\varepsilon'< \varepsilon\) such that for any \(\mathbf{f} = (f,g) \in \mathrm{B}^{\mathcal{H}^{k+1}}_{\varepsilon'}\), there exists a unique \(\mathbf{q}_{\alpha ^{\star},  \kappa ^{\star},T^{\star} } \in C(\mathbb{R}_+, \mathrm{B}^{\mathcal{H}^{k}}_\varepsilon)\) to \eqref{eq:7.Duhamel for q*} and thus to the abstract Cauchy problem \eqref{eq:7.Nonlinear equation}. Moreover, letting 
\begin{align*}
    u(t,x)   := U_{\alpha ^{\star}, \kappa ^{\star} , T^{\star}}  \left( - \log  (1 - \frac{t}{T^{\star}}), \frac{x - x_0}{T^{\star} - t}\right) + (\mathbf{q}_{\alpha ^{\star} , \kappa ^{\star} , T^{\star}})_1   \left( - \log  (1 - \frac{t}{T^{\star}}), \frac{x - x_0}{T^{\star} - t}\right) 
\end{align*} 
gives the unique solution to the Cauchy problem \eqref{eq:1.■u=(∂ₜu)²} in the backward lightcone \(\Gamma(T^{\star} , x_0)\) with initial data 
\[  \begin{cases} u(0,x) = u_{\alpha _0 , \infty, \kappa _0 , T_0 , x_0 }(0,x ) + f(x) , & x \in B_{T^{\star}}(x_0 ), \\
\partial_t u(0,x) = \partial_t u_{\alpha _0 , \infty, \kappa _0 , T_0 , x_0 }(0,x ) + g(x). \end{cases}  
\] 
Finally, Proposition \ref{prop:7.Fixed Point for Uncorrected} gives us the bounds 
\begin{align*}
    &(T^{\star} - t)^{-\frac{1}{2} + s} \left\lVert  u(t,\cdot) -  U_{\alpha ^{\star}, \kappa ^{\star} , T^{\star}}  \left( - \log  (1 - \frac{t}{T^{\star}}), \frac{\cdot - x_0}{T^{\star} - t}\right) \right\rVert_{\dot{H}^s(B_{T^{\star} - t}(x_0))} \\
    &\qquad\qquad= \left\lVert  (\mathbf{q}_{\alpha ^{\star} , \kappa ^{\star} , T^{\star}})_1 \left( - \log (1 - \frac{t}{T^{\star}}), \cdot\right)  \right\rVert_{\dot{H}^s(B)} \lesssim  \varepsilon'(T^{\star} - t)^{1 - \delta}, \\
    &(T^{\star} - t)^{-\frac{1}{2} + s} \left\lVert \partial_t  u(t,\cdot) - \partial_t U_{\alpha ^{\star}, \kappa ^{\star} , T^{\star}}  \left( - \log  (1 - \frac{t}{T^{\star}}), \frac{\cdot - x_0}{T^{\star} - t}\right) \right\rVert_{\dot{H}^s(B_{T^{\star} - t}(x_0))} \\
    &\qquad\qquad= \left\lVert  (\mathbf{q}_{\alpha ^{\star} , \kappa ^{\star} , T^{\star}})_2 \left( - \log (1 - \frac{t}{T^{\star}}), \cdot\right)  \right\rVert_{\dot{H}^s(B)} \lesssim  \varepsilon'(T^{\star} - t)^{1 - \delta}, 
\end{align*} 
for \(s = 1 ,\ldots, k-1.\)
\end{proof} 

Finally, we present a sketch of the proof of Theorem \ref{thm:main-2} as the main idea behind its proof is already covered in Theorem \ref{thm:main-1}. 
\begin{proof}[Proof of Theorem \ref{thm:main-2}]
  Recall that the results of Section \ref{sec:Mode stability of second} implies that the linearisation around \(U_{1,\beta_0,\kappa_0}\) given by \eqref{eq:2.Uαβκ for (∂ₜu)²−(∂ₓu)²} for the equation \eqref{eq:1.■u=(∂ₜu)²−(∂ₓu)²} can be written as 
  \begin{equation}\label{eq:7.PDE for L2β}
    \partial_s \mathbf{q} = \mathbf{L}^2_{1,\beta_0}  \mathbf{q} + \mathbf{M}(\mathbf{q})
  \end{equation}
  where 
  \[  \mathbf{L}^2_{1,\beta_0}  = \begin{pmatrix} q_2 - y \partial_y q_1 \\  - q_2 - y \partial_y q_2  + \partial_{yy} q_1  +  \frac{2(1 + \beta_0)}{1 + \beta_0 + y(1 - \beta_0)} q_2 - \frac{2(1 - \beta_0)}{1 + \beta_0  + y(1 - \beta_0)}\partial_y q_1 \end{pmatrix} , \quad \mathbf{M} = \begin{pmatrix} 0  \\q_2^2 - (\partial_y q_1)^2 \end{pmatrix}
  \] 

  For the previously developed framework to apply to \eqref{eq:7.PDE for L2β} we need to check the following three things: 
  \begin{itemize}
    \item Mode stability of the linearised operator \(\mathbf{L}^2_{1,\beta_0}.\) This was proved in Proposition \ref{prop:4.Mode stability of U1βκ}.  
    \item Spectral gap of the linearised operator \(\mathbf{L}^2_{1,\beta_0}\). For this, we mainly need to prove the maximal dissipativity of \( \mathbf{L}^2_{1,\beta_0 }\) up to a compact perturbation as in Proposition \ref{prop:5.Maximal Dissipativity of L̂α} since the remaining arguments follow as in Section \ref{sec:Spectral Analysis}. One can show that the following estimate holds for the modified operator \( \mathbf{\hat{L}}^2_{1,\beta_0} = \mathbf{L}^2_{1,\beta_0} - \mathbf{\hat{P}}_1: \)
    \[ \mathbb{R}{\rm e\,} \langle  \mathbf{\hat{L}}^2_{1,\beta_0} \mathbf{q}, \mathbf{q} \rangle_{\mathcal{H}^k} \geq \left( k - \frac{1}{2}  - \varepsilon-   \frac{1+ \beta_0 +  \lvert 1 - \beta_0 \rvert}{1 + \beta_0 -  \lvert 1 - \beta_0 \rvert}\right) \lVert  \mathbf{q} \rVert_{\mathcal{H}^k}^2.
    \] 
    Thus for \( k \geq k_{\beta_0} := 2 + \beta_0 ^{-1} + \beta_0\), we have 
    \[ \sigma(\mathbf{L}^2_{1,\beta_0}) \subset \{\mathbb{R}{\rm e\,} z \leq -1\} \cup \{0,1\}
    \] 
    and thus a similar statement like Proposition \ref{prop:5.Spectrum of Lα} holds for \( \mathbf{L}^2_{1,\beta_0}\) for \( k \geq k_{\beta_0}\). 
    \item Nonlinear estimate of \(\mathbf{M}\) as in Lemma \ref{lem:7.Estimates for Nonlinear Term}. Indeed for \(  k \geq 2: \)
  \begin{align*}
      \lVert  \mathbf{M}(\mathbf{r}) \rVert_{\mathcal{H}^k} &\lesssim  \lVert  (r_2)^2 \rVert_{H^k} + \lVert  (\partial_{y'} r_1)^2 \rVert_{H^k} \lesssim  \lVert  r_2 \rVert_{L^\infty} \lVert r_2 \rVert_{H^k} + \lVert  \partial_{y'} r_1 \rVert_{L^\infty} \lVert  \partial_{y'} r_1 \rVert_{H^k} \\
      &\lesssim  \lVert  r_2 \rVert_{H^k}^2 + \lVert  r_1 \rVert_{H^{ k+1}}^2 \lesssim  \lVert  \mathbf{r} \rVert_{\mathcal{H}^k}^2, 
  \end{align*} 
  and similarly, 
  \begin{align*}
      \lVert & \mathbf{M}(\mathbf{r}) - \mathbf{M}(\mathbf{r}') \rVert_{\mathcal{H}^k} \lesssim  \lVert  (r_2)^2 - (r'_2)^2 \rVert_{H^k} + \lVert  (\partial_{y'} r_1)^2 - (\partial_{y'} r_1')^2 \rVert_{H^k} \\
      &\lesssim  (\lVert  r_2 \rVert_{H^k} + \lVert  r_2 '\rVert_{H^k}) \lVert  r_2 - r_2' \rVert_{H^k} + (\lVert  \partial_{y'} r_1 \rVert_{H^k} + \lVert  \partial_{y'} r_1' \rVert_{H^k}) \lVert  \partial_{y'}(r_1 -r_1') \rVert_{H^k} \\
      &\lesssim  ( \lVert  \mathbf{r} \rVert_{\mathcal{H}^k} + \lVert  \mathbf{r}' \rVert_{\mathcal{H}^k}) \lVert  \mathbf{r}-\mathbf{r}' \rVert_{\mathcal{H}^{k+1}}, 
  \end{align*}  
  \end{itemize} 
  As such, we have the essential ingredients required to apply the framework developed for \eqref{eq:1.■u=(∂ₜu)²} and so the result is considered proven. 
\end{proof}  
    \appendix 
\section{Non-existence of other generalised eigenfunctions}
  \begin{lemma}\label{lem:6.Non-existence of other generalised eigenfunctions} There does not exist any \( \mathbf{v} \in \mathcal{H}^4\) such that \( \mathbf{L}_1 \mathbf{v} = \mathbf{g}_{0,1}\). Thus 
    \[ \exists \mathbf{v} \in \mathcal{H}^4 : \mathbf{L}_1 \mathbf{v} = c \mathbf{g}_{0,1} \implies c = 0 \text{ and }\mathbf{v} \in \ker  \mathbf{L}_{1 }. 
    \]  
  \end{lemma}
  \begin{proof} 
    Let \( \mathbf{v} = (v_1, v_2)    \) which satisfies \( \mathbf{L}_1  \mathbf{v} = \mathbf{g}_{0,0}\) then we obtain 
    \[  \mathbf{L}_1 \mathbf{v} = \begin{pmatrix} - y v_1' + v_2 \\ v_1'' + v_2 - y v_2'   \end{pmatrix}  = -\frac{y}{2}\begin{pmatrix}1 \\1   \end{pmatrix} .
    \] 
    Plugging the first equation into the second equation we obtain 
    \[  v_1 '' + y v_1' - \frac{y}{2} - y ( v_1' + y v_1'') -\frac{1}{2} = - \frac{y}{ 2} \implies  v_1''  = \frac{1}{2(1 - y^2)} 
    \] 
    whose general solution is given by 
    \[  v_1 = c_1 + c_2 y + \frac{1}{4}(1 + y) \log (1 + y) + \frac{1}{4}(1 - y) \log  (1 - y) . 
    \] 
    Then \( v_1 ' = c_2 + \frac{1}{4} \ln \frac{1 + y}{1 - y}\) which fails to be finite at \( \pm 1\). Thus \( \mathbf{v} \not\in \mathcal{H}^4\). 
  \end{proof} 

  \subsection*{Acknowledgment}
The authors would like to thank Professor Tej-eddine Ghoul for introducing this problem and for their invaluable discussions and insightful comments.

\subsection*{Funding}
Faiq Raees was supported by Tamkeen under the NYU Abu Dhabi Research Institute Award CG0002 to Stability, Instability, and Turbulence (SITE).

\printbibliography
  
\end{document}